\documentclass[10pt]{article}
\usepackage{enumitem}
\usepackage{amsfonts}
\usepackage{amsmath}
\usepackage{graphicx}
\usepackage{epstopdf}
\usepackage{fancyhdr}
\usepackage{amsthm}
\usepackage{geometry}
\usepackage{subfigure}
\usepackage[figuresright]{rotating}
\usepackage{caption}
\usepackage{bm}
\usepackage{amssymb}
\usepackage{amscd}
\usepackage{float}
\usepackage{color}
\usepackage{tabularx}
\usepackage{diagbox}
\usepackage{tabularx}
\usepackage{mathtools}
\usepackage{multicol}
\usepackage{multirow}
\usepackage{makecell}
\usepackage[justification=centering]{caption}
\usepackage{rotfloat}
\usepackage{longtable}
\usepackage{arydshln}
\usepackage{algorithm}
\usepackage{algpseudocode}
\allowdisplaybreaks[4]

\newtheorem{theorem}{Theorem}[section]
\newtheorem{corollary}{Corollary}[section]
\newtheorem{lemma}{Lemma}[section]

\newtheorem{remark}{Remark}[section]
\newtheorem{proposition}{Proposition}[section]
\newtheorem{definition}{Definition}[section]

\newcounter{nextauthor}
\def\mathrm{\mbox}

\numberwithin{remark}{section}

\begin{document}
\title{{\bf  Time-inconsistent reinsurance and investment optimization problem with delay under random risk aversion}\thanks{This work was supported by the National Natural Science Foundation of China (12171339, 12471296, 12501438), the Natural Science Foundation of Sichuan Province, China (2026NSFSC0793), Chongqing Natural Science Foundation General Project, China (CSTB2025NSCQ-GPX0815), the Scientific and Technological Research Program of Chongqing Municipal Education Commission, China (KJQN202400819) and the grant from Chongqing Technology and Business University, China (2356004).}}
\author{Jian-hao Kang$^a$, Zhun Gou$^b$ and Nan-jing Huang$^c$ \thanks{Corresponding author: nanjinghuang@hotmail.com; njhuang@scu.edu.cn}\\
{\small a. School of Mathematics, Southwest Jiaotong University, Chengdu, Sichuan 610031, P.R. China}\\
{\small b.  School of Mathematics and Statistics, and Chongqing Key Laboratory of Statistical Intelligent}\\
{\small Computing and Monitoring, Chongqing Technology and Business University, Chongqing 400067, P.R. China}\\
{\small c. Department of Mathematics, Sichuan University, Chengdu, Sichuan 610064, P.R. China}}
\date{}
\maketitle \vspace*{-9mm}
\begin{abstract}
\noindent
This paper considers a newly delayed reinsurance and investment optimization problem incorporating random risk aversion, in which an insurer pursues maximization of the expected certainty equivalent of her/his terminal wealth and the cumulative delayed information of the wealth over a period. Specially, the insurer's surplus dynamics are approximated using a drifted Brownian motion, while the financial market is described by the constant elasticity of variance (CEV) model. Moreover, the performance-linked capital flow feature is incorporated and the wealth process is formulated via a stochastic delay differential equation (SDDE). By adopting a game-theoretic approach, a verification theorem with rigorous proofs is established to capture the equilibrium reinsurance and investment strategy along with the equilibrium value function. Furthermore, analytical or semi-analytical equilibrium reinsurance and investment strategies, together with their equilibrium value functions, are obtained under the CEV model for the exponential utility and derived under the Black-Scholes model for both exponential and power utilities. Finally, several numerical experiments are conducted to analyze the behavioral characteristics of the freshly-derived equilibrium reinsurance and investment strategy.
 \\ \ \\
\noindent {\bf Keywords}: Risk management; reinsurance and investment; delay; random risk aversion; time-inconsistency; equilibrium strategy.
\\ \ \\
\noindent \textbf{AMS Subject Classification:} 62P05, 91B30, 93E20, 91G10.
\end{abstract}

\section{Introduction}
Risk management has always been a crucial subject for insurers. Reinsurance serves as an efficient mean for insurers to manage claim-related risk exposure and safeguard themselves against catastrophic loss scenarios. Within a reinsurance contract, an insurer transfers a specified portion of claim losses to a reinsurer and remits the reinsurance premium for the risk transfer. Conversely, the reinsurer has the obligation to reimburse a portion of the claims that the insurer has incurred. Generally, the magnitude of reinsurance premium rises in line with the proportion of ceded claims, which compels the insurer to strike a balance between the reinsurance premium and ceded claims in order to attain the optimally achieved state of the reinsurance arrangement. Moreover, nowadays, insurers are actively engaged in investment pursuits within financial markets. Thus, it is also highly relevant to formulate suitable investment strategies to boost returns to the fullest extent or/and cut down risks to the minimum when dealing with different investment tools.

A substantial number of studies have been conducted on formulating optimal reinsurance and investment strategies for insurers. In these studies, many scholars have dedicated to expanding the previous works to include a variety of optimality criteria and distinct foundational models for the insurer's surplus dynamics and risky asset evolution. For example, Browne \cite{Browne1995} studied the optimal portfolio allocation problem for an insurer, taking into account the criteria of utility maximization and minimization of ruin probability separately; H{\o}jgaard and Taksar \cite{Hojgaard1998} aimed at enabling an insurer to pursue maximization of the return function for the proportional reinsurance problem; Yang and Zhang \cite{Yang2005} adopted a jump-diffusion model to capture an insurer's surplus process and investigated the optimal investment issue. Many more studies and outcomes associated with the reinsurance and investment problems are available in the literature (see, for instance, Bai and Guo \cite{Bai2008}, Elliott and Siu \cite{Elliott2011}, Li et al. \cite{Li2012}, Wang et al. \cite{Wang2024} and the references therein).

Within the realm of reinsurance and investment optimization research, the issue of time inconsistency will arise when the insurer's objective function contains the non-exponential discounting \cite{Wang2024} or the mean-variance criterion \cite{Li2012}, which suggests that the optimal strategy obtained at the current moment might cease to be optimal at a future point. In order to find time-consistent strategies for the problems of time inconsistency, scholars have attempted to utilize the framework of game theory to obtain equilibrium strategies, rather than optimal strategies in the traditional sense. For example, following the classical dynamic programming framework, Bj{\"{o}}rk et al. \cite{Bjork2017} along with Bj{\"{o}}rk and Murgoci \cite{Bjork2014} derived an extended Hamilton-Jacobi-Bellman (HJB) equation to describe the equilibrium solution for a general Markovian time-inconsistent control problem; Hu et al. \cite{Hu2012} defined the equilibrium in the context of open-loop control framework for a general time-inconsistent stochastic linear-quadratic control problem and got the equilibrium with the help of the stochastic maximum principle; Yong \cite{Yong2012} proposed the so-called equilibrium HJB equation in an attempt to devise the equilibrium strategy in the setting of a multi-person differential game featuring a hierarchical structure for a time-inconsistent optimal control problem. For more studies and findings regarding the time-inconsistent problems, we refer the interested readers to Balter et al. \cite{Balter2021}, Bensoussan et al. \cite{Bensoussan2022}, Bj{\"{o}}rk et al. \cite{Bjork2021}, Hu et al. \cite{Hu2017}, Strotz \cite{Strotz1955}, Yong \cite{Yong2017}, Zhang et al. \cite{Zhang2024} and the references therein.

One typical feature of the above researches is that these researches concentrated on the controlled state system in the absence of delay considerations. Specially, it was commonly supposed that the system's future state has no connection with past states and is governed exclusively by the present situation. Nevertheless, in real-world systems, insurers often take past information into account when formulating their reinsurance and investment strategies, as past information can help insurers understand the cyclical patterns of markets and enhance the robustness of decision-making. This feature is frequently termed as the memory feature or the delay feature. Chang et al. \cite{Chang2011} used the HJB equation to study a consumption and investment problem featuring bounded memory, where the return of the risky asset was associated with the historic performance. Then, Shen and Zeng \cite{Shen2014} pioneered the incorporation of bounded wealth memory into a reinsurance and investment problem subject to the mean-variance criterion, deriving the pre-committed strategy through the stochastic maximum principle. It was pointed out in \cite{Chang2011, Shen2014} that since the delayed optimal control problems are usually infinite-dimensional, a significant challenge emerges, namely, the non-existence of the explicit solutions. Therefore, some specific constraints on model parameters were given in \cite{Chang2011, Shen2014} to render the original problems finite-dimensional and solvable. Further research works and results related to delayed optimal control problems  (including delayed optimal reinsurance/investment problems) are well-documented in the literature (see, for example, A et al. \cite{A2018}, A and Shao \cite{A2020}, A et al. \cite{A2022}, Bai et al. \cite{Bai2022}, Elsanosi et al. \cite{Elsanosi2000}, Federico \cite{Federico2011}, Meng et al. \cite{Meng2025}, Yan and Wong \cite{Yan2022}, Yuan et al. \cite{Yuan2023} and the references therein).

It is worthy of note that the majority of aforementioned studies assumed that investors/insurers had a constant risk aversion coefficient. In reality, investors/insurers cannot easily determine their exact level of risk aversion. This has motivated extensive research on risk aversion estimation through diverse methodologies such as employing data from the market on labor income \cite{Chetty2006}, using data pertaining to implied and realized volatilities from the market \cite{Bollerslev2011} and examining survey forms of representative samples \cite{Burgaard2020}. Recently, Desmettre and Steffensen \cite{Desmettre2023} proposed that it is more reasonable to characterize the risk aversion coefficient of investors with a random variable and they formulated a new investment optimization problem in the Black-Scholes financial market by employing expected certainty equivalent, which is time inconsistent. They also provided a strict verification theorem to capture the equilibrium investment strategy, deriving (semi-)analytical strategies for both power and exponential utility specifications. More recently, based on the research of \cite{Desmettre2023}, Kang et al. \cite{Kang2026} considered an optimal reinsurance and investment problem incorporating random risk aversion in the financial market characterized by the Heston stochastic volatility model. They also presented the pseudo HJB equation to characterize the equilibrium reinsurance and investment strategy, from which the semi-analytical equilibrium strategy was obtained under the exponential utility.

To the best knowledge of the authors, there is currently no literature that adopts the method of expected certainty equivalent to study the delayed optimal reinsurance and investment problem under random risk aversion. Thus, this paper systematically addresses this gap by applying expected certainty equivalent to consider a new delayed optimal reinsurance and investment problem incorporating random risk aversion. In particular, under the diffusion approximation modeling of the insurer's surplus process, the insurer can cede risk exposure to a reinsurer by employing proportional reinsurance strategy or seek growth through new business expansion and can also optimize asset allocations in a financial market following the CEV model. Furthermore, it is hypothesized that capital inflows and outflows affect the insurer's present wealth, with the associated wealth dynamics formulated via an SDDE. Then, an approach based on expected certainty equivalent is employed to formulate the general delayed optimal reinsurance and investment problem considering random risk aversion that can be characterized as a combination of nonlinear functions of expectations.

The above newly established problem exhibits time inconsistency because Bellman's dynamic programming principle is inapplicable. In order to obtain the (semi-)analytical equilibrium reinsurance and investment strategy for such a problem, we identify the following challenges. Firstly, due to the introduction of delay, the methodologies in \cite{Desmettre2023, Kang2026} for providing the verification theorem are no longer directly applicable. In this paper, we draw on \cite{Elsanosi2000} to derive the It\^{o} formula under delay scenario, and then adopt the methodologies from \cite{Bjork2017, Bjork2014, Desmettre2023, Kang2026}, integrating them with the aforementioned It\^{o} formula to provide a verification theorem in delay context. Secondly, as discussed earlier, the inclusion of delay leads to the problem being infinite-dimensional. To overcome this difficulty, we follow \cite{Chang2011, Shen2014} to impose specific assumptions on model parameters to transform the problem into a finite-dimensional framework, which facilitates our derivation of (semi-)analytical equilibrium reinsurance and investment strategies under the exponential utility and the power utility. Finally, we find that only a small number of studies, such as \cite{Bai2022, Chang2011}, have proven the admissibility of their delayed optimal (or equilibrium) strategies and the conditions in their verification theorem. However, these proofs have certain limitations and cannot be directly applied to this paper. For example, substituting equations (56) and (57) into equation (12) in \cite{Chang2011}, we discover that the proof process of Lemma 2.2 is no longer applicable and thus $c^{*}(t)\geq 0$ in the proof of Theorem 4.2 cannot be easy to check. What is more, Bai et al. \cite{Bai2022} did not consider the verification of the integrability of the value function (see the assumption $(A1)$ in this paper). To address this issue, we utilize the properties of solutions to stochastic differential equations (SDEs) and the assumption of model parameters to complete the corresponding proofs. After overcoming these challenges, we would like to point out that the newly derived equilibrium reinsurance and investment strategies extend some results in \cite{A2018, A2020, Desmettre2023, Kang2026} and also arrive at several novel and interesting conclusions. For instance, the expressions of the equilibrium reinsurance strategy and the equilibrium investment strategy under the power utility function in the Black-Scholes financial market demonstrate that even under the independence assumption between financial and insurance markets, both strategies are unexpectedly influenced by the model parameters from the reinsurance and financial market setups.

The primary contributions of this paper can be outlined from four aspects: (i) a newly delayed optimal reinsurance and investment model with random risk aversion is established, which fully generalizes the model in \cite{Desmettre2023} and partially generalizes the model in \cite{Kang2026}, while also enriching the research on time-inconsistent problems under delay scenarios; (ii) a complete verification theorem is provided to capture the equilibrium reinsurance and investment strategy and the associated equilibrium value function; (iii) analytical or semi-analytical equilibrium reinsurance and investment strategies, together with their equilibrium value functions, are obtained under the CEV model for the exponential utility and derived under the Black-Scholes model for both exponential and power utilities, followed by the proofs of their admissibility and the validity of the verification theorem's conditions; (iv) computational experiments quantify parameter sensitivity in the equilibrium reinsurance and investment strategy, which yields novel and intriguing findings.

The subsequent sections of this paper are structured as outlined below. Section 2 details the model framework and sets forth the fundamental assumptions. Section 3 poses the delayed optimal reinsurance and investment problem under random risk aversion and gives a verification theorem capturing the equilibrium reinsurance and investment strategy along with the equilibrium value function. Section 4 derives the (semi-)closed-form solutions to equilibrium reinsurance and investment strategies under the CEV model for the exponential utility. Section 5 obtains (semi-)closed-form equilibrium reinsurance and investment strategies under the Black-Scholes model for exponential and power utilities separately. Section 6 presents numerical simulations to quantify how model parameters influence the equilibrium reinsurance and investment strategy. Section 7 offers a comprehensive summary of this study and Appendices include the detailed proofs.

\section{Model setup}
Consider a filtered probability space $(\Omega,\mathcal{F},\mathbb{F}:=\{\mathcal{F}_{t}\}_{t\in[0,T]},\mathbb{P})$ that fulfills the standard conditions of completeness and right-continuity, with $T >0$ representing a finite final time. Throughout the entire study, we consider every stochastic process as adapted to the given probability space. Furthermore, our analysis is based on the assumption of frictionless insurance and financial markets with continuous trading opportunities, where neither transaction costs nor taxes exist. Moreover, for an arbitrary positive integer $j$, we denote the set $\mathbb{S}^{j}_{\mathcal{F}}(0,T;\mathbb{R})$ as the family of $\mathcal{F}_{t}$-adapted
c\`{a}dl\`{a}g processes $Z(\cdot)$ satisfying $\mathbb{E}\left[\sup\limits_{0\leq t\leq T}|Z(t)|^{j}\right]<\infty$ and let the set $\mathbb{L}^{j}_{\mathcal{F}}(0,T;\mathbb{R})$ be defined as the class of $\mathcal{F}_{t}$-adapted stochastic processes $Z(\cdot)$ with the property that $\mathbb{E}\left[\int^{T}_{0}| Z(t)|^{j}\mathrm{d}t\right]<\infty$, and for any continuous function $\psi$ on $[-h,0]$, we denote by $\|\psi\|=\max\limits_{t\in[-h,0]}|\psi(t)|$ and $\psi_{s}(t)=\psi(s+t)$  for any $t\in[-h,0]$.

\subsection{Surplus process}
In scenarios without reinsurance or investment activities, the dynamics of the insurer's surplus are modeled within the traditional Crem$\acute{e}$r-Lundberg framework, which can be formally described as
\begin{equation*}\label{eq1}
  R(t)=\tilde{u}_{0}+c t-\sum\limits_{i=1}^{N(t)}Z_{i}.
\end{equation*}
Here, $\tilde{u}_{0}\geq0$ stands for the initial surplus, $c$ represents the premium rate, and the aggregate claims process is modeled as a compound Poisson process $\sum\limits_{i=1}^{N(t)}Z_{i}$, in which $\{N(t)\}_{t\in[0,T]}$ follows a homogeneous Poisson process with intensity $\lambda_{1}>0$ and counts claims in $[0,t]$, and random variables $\{Z_{i}\}_{i\in\mathbb{N}}$ are positive, independent and identically distributed and denote claim sizes, satisfying $\mathbb{E}[Z_{i}]=\mu_{1}>0$ and $\mathbb{E}[Z^{2}_{i}]=\mu_{2}>0$ as well as $\{Z_{i}\}_{i\in\mathbb{N}}$ and $\{N(t)\}_{t\in[0,T]}$ are independent of each other. Moreover, we postulate that the premium rate $c$ is set in accordance with the expected value principle, which takes the form $c=(1+\eta_{1})\lambda_{1}\mu_{1}$, where $\eta_{1}>0$ represents the safety loading factor of the insurer.

We further assume that the insurer's risk management strategy incorporates two operational modes through the risk retention level $q(t)$ at time $t$. For $q(t)\in[0,1]$, the insurer engages in proportional reinsurance, retaining $100q(t)\%$ portion of each claim while the reinsurer covers the remaining $100(1-q(t))\%$ portion.
It is postulated that the reinsurance premium is also determined in line with the expected value principle, which can be expressed as $(1+\eta_{2})(1-q(t))\lambda_{1}\mu_{1}$, where $\eta_{2}(\geq\eta_{1})$ represents the reinsurer's safety loading and $\eta_{2}\geq\eta_{1}$ ensures no arbitrage opportunities. Additionally, for $q(t)>1$, the insurer assumes a reinsurer's role by acquiring new business from other insurance providers, thereby expanding its risk and generating additional premium income. For notational convenience, we refer to the stochastic process $\{q(t)\}_{t\in[0,T]}$ as the reinsurance strategy hereafter. If the insurer adopts such a strategy, then the behavior of the insurer's surplus process can be characterized by
\begin{align}\label{eq2}
   R(t)&=\tilde{u}_{0}+\int^{t}_{0}\left[c-(1+\eta_{2})(1-q(s))\lambda_{1}\mu_{1}\right]\mathrm{d}s-\sum\limits_{i=1}^{N(t)}q(T_{i})Z_{i}\nonumber\\
   &=\tilde{u}_{0}+\int^{t}_{0}\left[\eta+(1+\eta_{2})q(s)\right]\lambda_{1}\mu_{1}\mathrm{d}s-\sum\limits_{i=1}^{N(t)}q(T_{i})Z_{i},
\end{align}
where $\eta=\eta_{1}-\eta_{2}\leq0$ and $T_{i}$ denotes the $i$-th claim's occurring time. Following the methodology established in \cite{Grandell1990}, we can approximate the surplus process $R(t)$ in \eqref{eq2} by using a diffusion model, which yields the following SDE
\begin{align*}
   \mathrm{d}R(t)=a[\eta+\eta_{2}q(t)]\mathrm{d}t+bq(t)\mathrm{d}W_{1}(t),
\end{align*}
where $a=\lambda_{1}\mu_{1}$, $b=\sqrt{\lambda_{1}\mu_{2}}$ and $W_{1}:=\{W_{1}(t)\}_{t\in[0,T]}$ is a standard $\mathbb{P}$-Brownian motion.

\subsection{Financial market}
The financial market under consideration comprises two primary investment instruments: a risk-free asset (e.g., bank deposit or government bond) and a risky asset (e.g., equity securities). The price dynamics of the risk-free asset, denoted by $S_{0}(t)$, evolve according to the subsequent ordinary differential equation (ODE)
\begin{equation*}\label{eq4}
  \mathrm{d}S_{0}(t)=rS_{0}(t)\mathrm{d}t, \quad S_{0}(0)=s_{0}>0,
\end{equation*}
where $r>0$ is a constant denoting the risk-free interest rate. In addition, the price process $S_{1}(t)$ of the risky asset is governed by the CEV model, which satisfies
\begin{equation*}\label{eq5}
  \mathrm{d}S_{1}(t)=\mu S_{1}(t)\mathrm{d}t+\sigma S^{\delta+1}_{1}(t)\mathrm{d}W_{2}(t), \quad S_{1}(0)=\bar{s}_{1}>0,
\end{equation*}
where $\mu$ and $\sigma S^{\delta}_{1}(t)$ characterize respectively the expected return rate and instantaneous volatility of $S_{1}(t)$ with $\mu>r$ and $\sigma>0$, $\delta>0$ denotes the elasticity parameter, and $W_{2}:=\{W_{2}(t)\}_{t\in[0,T]}$ is a standard $\mathbb{P}$-Brownian motion independent of $W_{1}$. We note that the CEV model degenerates into geometric Brownian motion when $\delta=0$.

\subsection{Wealth process with delay}
We hypothesize that the insurer's strategic decision-making involves a dual approach: transferring risks by means of proportional reinsurance or growing the business through new policy acquisitions, while simultaneously optimizing financial market investment portfolios. Let $\pi(t)$ represent the fraction of the insurer's wealth placed into the risky asset at time $t$, with $1-\pi(t)$ denoting the fraction allocated to the risk-free asset. Then a pair of stochastic processes $u(t):=(q(t),\pi(t))$ can be used to stand for the insurer's trading strategy and the wealth process $X^{u}(t)$ of the insurer can be expressed by the subsequent SDE
\begin{align*}
   \mathrm{d}X^{u}(t)&=\frac{(1-\pi(t))X^{u}(t)}{S_{0}(t)}\mathrm{d}S_{0}(t)+\frac{\pi(t)X^{u}(t)}{S_{1}(t)}\mathrm{d}S_{1}(t)+\mathrm{d}R(t)\\
   &=[X^{u}(t)\left(r+\pi(t)(\mu-r)\right)+a\eta+a\eta_{2}q(t)]\mathrm{d}t+bq(t)\mathrm{d}W_{1}(t)+\pi(t)X^{u}(t)\sigma S^{\delta}_{1}(t)\mathrm{d}W_{2}(t),
\end{align*}
where the wealth value at time 0 is $X^{u}(0)=x_{0}>0$.

Now, we consider the impact of historical performance on the wealth development of the insurer. As mentioned in \cite{Shen2014}, the historical trajectory of the insurer's wealth position influences capital movements, resulting in potential injections or withdrawals from the current wealth accumulation process. Similar to \cite{Chang2011, Federico2011, Shen2014}, we apply $M_{1}^{u}(t)$ and $M_{2}^{u}(t)$ to denote the cumulative and instantaneous delayed information of the insurer's wealth within the past horizon $[t-h,t]$, respectively, with these defined by
$M_{1}^{u}(t)=\int^{0}_{-h}e^{\alpha s}X^{u}(t+s)\mathrm{d}s$ and
$M_{2}^{u}(t)=X^{u}(t-h)$ for all $t\in[\tau,T]$, where $\tau\in[0,T]$ is an initial time, $h>0$ is a delay parameter and $\alpha\geq 0$ is an average parameter. Then the differential form of $M_{1}^{u}(t)$ can be derived as follows
\begin{align}\label{eq9}
\frac{\mathrm{d}M_{1}^{u}(t)}{\mathrm{d}t}=&\frac{\mathrm{d}}{\mathrm{d}t}\left[\int^{0}_{-h}e^{\alpha s}X^{u}(t+s)\mathrm{d}s\right]
=\frac{\mathrm{d}}{\mathrm{d}t}\left[\int^{t}_{t-h}e^{\alpha (\upsilon-t)}X^{u}(\upsilon)\mathrm{d}\upsilon\right]\nonumber\\
=&X^{u}(t)-e^{-\alpha h }X^{u}(t-h)-\alpha\int^{t}_{t-h}e^{\alpha (\upsilon-t)}X^{u}(\upsilon)\mathrm{d}\upsilon\nonumber\\
=&X^{u}(t)-e^{-\alpha h }X^{u}(t-h)-\alpha\int^{0}_{-h}e^{\alpha s}X^{u}(t+s)\mathrm{d}s\nonumber\\
=&X^{u}(t)-\alpha M_{1}^{u}(t)-e^{-\alpha h }M_{2}^{u}(t).
\end{align}
Moreover, we adopt the function $f(t, X^{u}(t)-M_{1}^{u}(t), X^{u}(t)-M_{2}^{u}(t))$ to stand for the amount of injections/withdrawals, where $X^{u}(t)-M_{1}^{u}(t)$ measures the average performance of the insurer's wealth over the period $[t-h,t]$ and $X^{u}(t)-M_{2}^{u}(t)$ represents the absolute performance during the same time interval.

To ensure the solvability of the optimization problem, inspired by previous studies (see, for example, \cite{Chang2011, Shen2014}), we adopt the assumption that the magnitude of injections/withdrawals is determined by a linear combination of average and absolute performance of the insurer's wealth with respective weights applied, which is given by
\begin{equation*}\label{eq910}
f(t, X^{u}(t)-M_{1}^{u}(t), X^{u}(t)-M_{2}^{u}(t))=B(X^{u}(t)-M_{1}^{u}(t))+C(X^{u}(t)-M_{2}^{u}(t)),
\end{equation*}
where $B$ and $C$ are nonnegative constants.

Next, taking account of the amount of injections or withdrawals, the wealth process of the insurer $X^{u}(t)$ can be rewritten via the subsequent SDDE
\begin{align}\label{eq11}
   \mathrm{d}X^{u}(t)=&[X^{u}(t)\left(r+\pi(t)(\mu-r)\right)+a\eta+a\eta_{2}q(t)]\mathrm{d}t+bq(t)\mathrm{d}W_{1}(t)+\pi(t)X^{u}(t)\sigma S^{\delta}_{1}(t)\mathrm{d}W_{2}(t)\nonumber\\
   &-f(t, X^{u}(t)-M_{1}^{u}(t), X^{u}(t)-M_{2}^{u}(t))\mathrm{d}t\nonumber\\
   =&[X^{u}(t)\left(A+\pi(t)(\mu-r)\right)+BM_{1}^{u}(t)+CM_{2}^{u}(t)+a\eta+a\eta_{2}q(t)]\mathrm{d}t\nonumber\\
   &+bq(t)\mathrm{d}W_{1}(t)+\pi(t)X^{u}(t)\sigma S^{\delta}_{1}(t)\mathrm{d}W_{2}(t),
\end{align}
where $A=r-B-C$. It can be seen from the equation \eqref{eq11} that $f>0$ means capital withdrawal and $f<0$ implies capital injection. Based on the expression of $f$, it can be deduced that if $X^{u}(t)>M_{1}^{u}(t)$ and $X^{u}(t)>M_{2}^{u}(t)$, then $f>0$. Actually, if the previous performance is good, which results in that the current wealth exceeds the cumulative delayed information of the wealth for the insurer, then she/he may utilize a part of the gain to distribute dividends to her/his shareholders and other interested parties or offer incentive payments to her/his management. This is the scenario of capital withdrawal. On the other hand, if $X^{u}(t)<M_{1}^{u}(t)$ and $X^{u}(t)<M_{2}^{u}(t)$, then $f<0$. In reality, if the previous performance is bad, which leads to that the current wealth falls below the cumulative delayed information of the wealth for the insurer, then she/he may employ her/his reserves as an injection of capital or raise funds from the capital market to compensate for the loss, ensuring that the ultimate performance target remains achievable. This is tantamount to a capital injection.

We suppose that the initial condition for $X^{u}(t)$ is its historical performance from $\tau-h$ to $\tau$ formulated as $X^{u}(t)=\psi(t-\tau)$ for $t\in[\tau-h,\tau]$, where $\psi>0$ is a continuous function on $[-h,0]$. If $\tau=0$, then we suppose that the insurer's initial wealth is constant, namely, for any $t\in[-h,0]$, $X^{u}(t)=\psi(t)=x_{0}$, which implies that the insurer enters the market with initial capital $x_{0}$, commencing all business activities ((re)insurance and investment) precisely at time $t=0$. Therefore, the initial condition for $M_{1}^{u}(t)$ is $M_{1}^{u}(0)=\frac{x_{0}(1-e^{-\alpha h})}{\alpha}$. Furthermore, we adopt the following notations for convenience. For $t\in[-h,0]$, denote $M_{1}^{u}(t)=m_{10}$ and $M_{2}^{u}(t)=m_{20}$. For $\tau\in[0,T]$ and $t\in[\tau-h,\tau]$, denote $S_{1}(t)=s_{1}$, $X^{u}(t)=x$, $M_{1}^{u}(t)=m_{1}$ and $M_{2}^{u}(t)=m_{2}$.

\section{Problem formulation and verification theorem}
This section will establish a theoretical framework for optimizing delayed reinsurance and investment decisions under random risk aversion. We will first formalize the optimization problem and then derive a verification theorem to characterize equilibrium reinsurance and investment strategies via game theory.

Among existing studies, the objective function of reinsurance and investment optimization problems with delay incorporates both the insurer's terminal wealth $X^{u}(T)$ and the cumulative delayed information of the wealth over the period $[T-h,T]$, i.e., $M_{1}^{u}(T)$. Especially, the corresponding reinsurance and investment optimization problem is presented as follows
\begin{equation}\label{eq12}
  \widehat{V}(t,x,s_{1},m_{1},m_{2}):=\sup\limits_{u} \; \mathbb{E}_{t}\left[\varphi\left(X^{u}(T)+\beta M_{1}^{u}(T)\right)\right],
\end{equation}
where $\varphi:\mathbb{R}\rightarrow \mathbb{R}$ stands for a utility function, the operator $\mathbb{E}_{t}[\cdot]$ computes the $\mathcal{F}_{t}$-conditional expectation under the probability measure $\mathbb{P}$ and the weight of $M_{1}^{u}(T)$ is determined by the constant $\beta\geq 0$, which quantifies the influence of $M_{1}^{u}(T)$ on the final performance. It is widely known that if a reinsurance-investment strategy $u$ attains the supremum value in problem \eqref{eq12}, then it is deemed an optimal strategy. Therefore, the insurer's optimization problem can be reformulated as maximizing the certainty equivalent of $X^{u}(T)$ and $M_{1}^{u}(T)$, which is mathematically equivalent to the following reward functional maximization
\begin{equation}\label{eq13}
  \widehat{J}^{u}(t,x,s_{1},m_{1},m_{2}):=\varphi^{-1}\left(\mathbb{E}_{t}\left[\varphi\left(X^{u}(T)+\beta M_{1}^{u}(T)\right)\right]\right),
\end{equation}
where $\varphi^{-1}(\cdot)$ is the inverse function of $\varphi(\cdot)$.

Regarding the reward functional \eqref{eq13}, we adopt the framework of \cite{Desmettre2023, Kang2026} by modeling the utility function's risk aversion coefficient $\gamma$ as a random variable. Building upon the certainty equivalent maximization principle for the terminal wealth and the cumulative delayed performance under random risk aversion, our objective becomes the optimization of the subsequent reward functional
\begin{equation}\label{eq14}
  J^{u}(t,x,s_{1},m_{1},m_{2}):=\int(\varphi^{\gamma})^{-1}\left(\mathbb{E}_{t}\left[\varphi^{\gamma}\left(X^{u}(T)+\beta M_{1}^{u}(T)\right)\right]\right)\mathrm{d}\Gamma(\gamma).
\end{equation}
Here, the random variable characterized by $\gamma$ has $\Gamma$ as its cumulative distribution function and the integration is performed throughout the region $supp(\Gamma)=\{\gamma \in(0,\infty)|\;\Gamma(\gamma)>0\}$. We would like to point out that we employ the superscript notation $\varphi^{\gamma}$ to explicitly denote the functional dependence of $\varphi$ on the risk aversion parameter $\gamma$. To simplify notation in subsequent analysis, we introduce the following symbolic conventions:
\begin{equation}\label{eq15}
  y^{u,\gamma}(t,x,s_{1},m_{1},m_{2}):=\mathbb{E}_{t}\left[\varphi^{\gamma}\left(X^{u}(T)+\beta M_{1}^{u}(T)\right)\right]
\end{equation}
and $\iota^{\gamma}(y):=\frac{\mathrm{d}}{\mathrm{d} y}(\varphi^{\gamma})^{-1}(y)$. Thus, the reward functional \eqref{eq14} can be expressed alternatively as
\begin{equation}\label{eq17}
  J^{u}(t,x,s_{1},m_{1},m_{2})=\int(\varphi^{\gamma})^{-1}\left(y^{u,\gamma}(t,x,s_{1},m_{1},m_{2})\right)\mathrm{d}\Gamma(\gamma)
\end{equation}
and then the insurer aims to achieve the supremum of functional \eqref{eq17}.

From equations \eqref{eq15} and \eqref{eq17}, we observe that the equation \eqref{eq17} can be regarded as a summation of nonlinear functions dependent on expectations, leading to time-inconsistency in the optimization problem. To address the time-inconsistency issue, we build upon the studies of \cite{Bjork2017, Bjork2021, Desmettre2023, Kang2026, Yuan2023} to pursue the time-consistent equilibrium strategy for reinsurance and investment by employing a game-theoretic framework. To begin with, letting $\mathcal{Q}:=[0,T]\times\mathcal{O}$ with $\mathcal{O}:=\mathbb{R}\times[0,\infty)\times\mathbb{R}\times\mathbb{R}$, we formally define the class of admissible reinsurance and investment strategies as follows.
\begin{definition}\label{definition3.1}
A reinsurance and investment strategy $u(t)=(q(t),\pi(t))$ is admissible on $[0,T]$ if it satisfies the following conditions:
\begin{itemize}
\item[(i)] for each initial point $(t,x,s_{1},m_{1},m_{2})\in\mathcal{Q}$, the SDDE \eqref{eq11} possesses a unique strong solution $X^{u}$;
\item[(ii)] $\{u(t)\}_{t\in[0,T]}$ is $\mathcal{F}_{t}$-progressively measurable and continuous, $\mathbb{E}\left[\int^{T}_{0}\left(q^{2}(t)+\pi^{2}(t)(X^{u}(t))^{2}(S^{\delta}_{1}(t))^{2}\right)\mathrm{d}t\right]<\infty$ and $q(t)\geq0$;
\item[(iii)] $\int\left|(\varphi^{\gamma})^{-1}\left(\mathbb{E}_{t}\left[\varphi^{\gamma}(X^{u}(T)+\beta M_{1}^{u}(T))\right]\right)\right|\mathrm{d}\Gamma(\gamma)<\infty$.
\end{itemize}
We denote by $\Pi$ the collection of all admissible reinsurance and investment strategies.
\end{definition}

Under the assumption $(i)$ of Definition \ref{definition3.1}, when the solution to the SDDE \eqref{eq11} belongs to $\mathbb{S}^{2}_{\mathcal{F}}(0,T;\mathbb{R})$, we can immediately obtain the subsequent results.

\begin{lemma}\label{lemma3.2}
If $X^{u}(\cdot)\in\mathbb{S}^{2}_{\mathcal{F}}(0,T;\mathbb{R})$, then $X^{u}(\cdot),M_{1}^{u}(\cdot),M_{2}^{u}(\cdot)\in\mathbb{S}^{1}_{\mathcal{F}}(0,T;\mathbb{R})$ and $X^{u}(\cdot),M_{1}^{u}(\cdot),M_{2}^{u}(\cdot)\in\mathbb{L}^{1}_{\mathcal{F}}(0,T;\mathbb{R})$.
\end{lemma}
\begin{proof}
See Appendix A.
\end{proof}

Then following game-theoretic ways, we will define the equilibrium reinsurance and investment strategy.
\begin{definition}\label{definition3.2}
For an arbitrary initial point $(t,x,s_{1},m_{1},m_{2})\in\mathcal{Q}$, denote an admissible reinsurance and investment strategy by $\hat{u}(t,x,s_{1},m_{1},m_{2})=(\hat{q}(t,x,s_{1},m_{1},m_{2}),\hat{\pi}(t,x,s_{1},m_{1},m_{2}))$. Consider arbitrary parameters $q\geq0$, $\pi\in\mathbb{R}$ and $\varepsilon>0$, then define a reinsurance and investment strategy through
\begin{align*}
\left(q_{\varepsilon}(s,\widetilde{x},\tilde{s}_{1},\widetilde{m}_{1},\widetilde{m}_{2}),\pi_{\varepsilon}(s,\widetilde{x},\tilde{s}_{1},\widetilde{m}_{1},\widetilde{m}_{2})\right)=
\begin{cases}
(q,\pi), \quad\quad\quad\quad\quad\;\quad\quad\;\quad\quad\quad\quad\quad\quad\quad\;\;\text{for} \quad\, t\leq s < t+\varepsilon,\\
(\hat{q}(s,\widetilde{x},\tilde{s}_{1},\widetilde{m}_{1},\widetilde{m}_{2}),\hat{\pi}(s,\widetilde{x},\tilde{s}_{1},\widetilde{m}_{1},\widetilde{m}_{2})), \quad\text{for}\;\quad  t+\varepsilon\leq s < T,
\end{cases}
\end{align*}
where $(\widetilde{x},\tilde{s}_{1},\widetilde{m}_{1},\widetilde{m}_{2})\in \mathcal{O}$.
Let $u_{\varepsilon}(s,\widetilde{x},\tilde{s}_{1},\widetilde{m}_{1},\widetilde{m}_{2}):=(q_{\varepsilon}(s,\widetilde{x},\tilde{s}_{1},\widetilde{m}_{1},\widetilde{m}_{2}),\pi_{\varepsilon}(s,\widetilde{x},\tilde{s}_{1},\widetilde{m}_{1},\widetilde{m}_{2}))$. If $$\underset{\varepsilon\rightarrow0}\liminf\frac{J^{\hat{u}}(t,x,s_{1},m_{1},m_{2})-J^{u_{\varepsilon}}(t,x,s_{1},m_{1},m_{2})}{\varepsilon}
\geq0
$$
for all $(q,\pi)\in\Pi$, then $\hat{u}(t,x,s_{1},m_{1},m_{2})$ is defined as an equilibrium reinsurance and investment strategy, with its associated equilibrium value function $V(t,x,s_{1},m_{1},m_{2}):=J^{\hat{u}}(t,x,s_{1},m_{1},m_{2})$.
\end{definition}

Pursuant to Definition \ref{definition3.2}, the insurer seeks to maximize the functional \eqref{eq17} by determining the equilibrium reinsurance and investment strategy. To streamline our presentation, we establish the following notations. For $\mathcal{D}=[0,T]\times\mathbb{R}\times[0,\infty)\times\mathbb{R}$, denote
\begin{align*}
C^{1,2,2,1}(\mathcal{D})=&\left\{\phi(t,x,s_{1},m_{1})|\phi(t,\cdot,\cdot,\cdot)\; \text{is once continuously differentiable on} \;[0,T],\right. \\
 &\;\left.\phi(\cdot,x,\cdot,\cdot)\;\text{is twice continuously differentiable on}\; \mathbb{R}, \; \right.  \\
 &\;\left.\phi(\cdot,\cdot,s_{1},\cdot)\;\text{is twice continuously differentiable on}\; [0,\infty) \;  \text{and} \right. \\
 &\;\left. \phi(\cdot,\cdot,\cdot,m_{1})\; \text{is once continuously differentiable on} \;\mathbb{R}\right\}
\end{align*}
and for any $\phi(t,x,s_{1},m_{1})\in C^{1,2,2,1}(\mathcal{D})$ and $u=(q,\pi)$, let
\begin{align*}
  \mathcal{A}^{u}\phi(t,x,s_{1},m_{1})=&\phi_{t}(t,x,s_{1},m_{1})+\left[(A+\pi(\mu-r))x+Bm_{1}+Cm_{2}+a\eta+a\eta_{2}q\right]\phi_{x}(t,x,s_{1},m_{1})\\
  &+0.5\left(b^{2}q^{2}+\pi^{2}x^{2}\sigma^{2}s^{2 \delta}_{1}\right)\phi_{xx}(t,x,s_{1},m_{1})+\mu s_{1}\phi_{s_{1}}(t,x,s_{1},m_{1})+0.5\sigma^{2}s^{2 \delta+2}_{1}\phi_{s_{1}s_{1}}(t,x,s_{1},m_{1})\\
  &+\pi x \sigma^{2}s^{2 \delta+1}_{1}\phi_{xs_{1}}(t,x,s_{1},m_{1})+\left(x-\alpha m_{1}-e^{-\alpha h}m_{2}\right)\phi_{m_{1}}(t,x,s_{1},m_{1}).
\end{align*}
Furthermore, we assume that
\begin{itemize}
  \item[$(A1)$] $\mathcal{A}^{u}\phi(t,X^{u}(t),S_{1}(t),M_{1}^{u}(t))\in\mathbb{L}^{1}_{\mathcal{F}}(0,T;\mathbb{R})$ and $\sigma S^{\delta+1}_{1}(t)\phi_{s_{1}}(t,X^{u}(t),S_{1}(t),M_{1}^{u}(t))$, $bq(t)\phi_{x}(t,X^{u}(t),S_{1}(t),M_{1}^{u}(t))$, $\pi(t) X^{u}(t)\sigma S^{\delta}_{1}(t)\phi_{x}(t,X^{u}(t),S_{1}(t),M_{1}^{u}(t))\in\mathbb{L}^{2}_{\mathcal{F}}(0,T;\mathbb{R})$.
\end{itemize}

Next, with the above notations, we will present the It\^{o} formula for the delayed reinsurance and investment problem.

\begin{lemma}\label{lemma3.1}
If $G(t,x,s_{1},m_{1})\in C^{1,2,2,1}(\mathcal{D})$ and $G(t,X^{u}(t),S_{1}(t),M_{1}^{u}(t))$ satisfies the condition $(A1)$, then
\begin{align*}
\mathrm{d}G(t,x,s_{1},m_{1})=&\mathcal{A}^{u}G(t,x,s_{1},m_{1})\mathrm{d}t+bqG_{x}(t,x,s_{1},m_{1})\mathrm{d}W_{1}(t)\\
&+\left[\sigma s^{\delta+1}_{1}G_{s_{1}}(t,x,s_{1},m_{1})+\pi x\sigma s^{\delta}_{1}G_{x}(t,x,s_{1},m_{1})\right]\mathrm{d}W_{2}(t).
\end{align*}
\end{lemma}
\begin{proof}
See Appendix B.
\end{proof}


We now establish a verification theorem to characterize the equilibrium strategy for reinsurance and investment.
\begin{theorem}\label{theorem3.1}
Suppose there exist functions $U, Y^{\gamma}, H\in C^{1,2,2,1}(\mathcal{D})$ satisfying
\begin{align}\label{eq20}
\sup\limits_{u} \;\left\{\mathcal{A}^{u}U(t,x,s_{1},m_{1})-\mathcal{A}^{u}H(t,x,s_{1},m_{1})
+\int\iota^{\gamma}(Y^{\gamma}(t,x,s_{1},m_{1}))\mathcal{A}^{u}Y^{\gamma}(t,x,s_{1},m_{1})\mathrm{d}\Gamma(\gamma)\right\}=0
\end{align}
and
\begin{equation}\label{eq21}
  \mathcal{A}^{\hat{u}}Y^{\gamma}(t,x,s_{1},m_{1})=0
\end{equation}
with boundary conditions $U(T,x,s_{1},m_{1})=x+\beta m_{1}$ and $Y^{\gamma}(T,x,s_{1},m_{1})=\varphi^{\gamma}(x+\beta m_{1})$ for all $\gamma$, where
\begin{equation}\label{eq22}
H(t,x,s_{1},m_{1})=\int(\varphi^{\gamma})^{-1}\left(Y^{\gamma}(t,x,s_{1},m_{1})\right)\mathrm{d}\Gamma(\gamma)
\end{equation}
and
\begin{align*}\label{eq23}
\hat{u}=\arg\sup\limits_{u} \;\left\{\mathcal{A}^{u}U(t,x,s_{1},m_{1})-\mathcal{A}^{u}H(t,x,s_{1},m_{1})
+\int\iota^{\gamma}(Y^{\gamma}(t,x,s_{1},m_{1}))\mathcal{A}^{u}Y^{\gamma}(t,x,s_{1},m_{1})\mathrm{d}\Gamma(\gamma)\right\}.
\end{align*}
Suppose further that $U$, $Y^{\gamma}$ and $H$ satisfy the condition $(A1)$ for all $\gamma$. Then $\hat{u}=(\hat{q},\hat{\pi})$ is an equilibrium
reinsurance and investment strategy and $V(t,x,s_{1},m_{1},m_{2})=U(t,x,s_{1},m_{1})$ as well as $y^{\hat{u},\gamma}(t,x,s_{1},m_{1},m_{2})=Y^{\gamma}(t,x,s_{1},m_{1})$.
\end{theorem}

\begin{proof}
See Appendix C.
\end{proof}

\begin{remark}
It follows from the proof of Theorem \ref{theorem3.1} that $Y^{\gamma}(t,x,s_{1},m_{1})= y^{\hat{u},\gamma}(t,x,s_{1},m_{1},m_{2})$ and $U(t,x,s_{1},m_{1})=J^{\hat{u}}(t,x,s_{1},m_{1},m_{2})=V(t,x,s_{1},m_{1},m_{2})$. This implies $y^{\hat{u},\gamma}$, $J^{\hat{u}}$ and $V$ are independent of $m_{2}$. In the sequel, we will adopt notations $y^{\hat{u},\gamma}(t,x,s_{1},m_{1})$, $J^{\hat{u}}(t,x,s_{1},m_{1})$ and $V(t,x,s_{1},m_{1})$ for simplicity.
\end{remark}

\section{Equilibrium strategies under the CEV model}
In this section, we first seek to tackle the optimization problem of delayed reinsurance and investment with random risk aversion under the CEV model, adopting a general utility function framework to the greatest extent possible. Afterward, to derive (semi-)analytical equilibrium strategy for reinsurance and investment, we narrow our scope to analyze the exponential utility function.

Following \cite{Desmettre2023, Kang2026}, one can formulate a more explicit expression of the pseudo HJB equation to characterize the equilibrium strategy for reinsurance and investment. For this purpose, by the equation \eqref{eq22}, one has
\begin{eqnarray}\label{eq35}
\left\{
\begin{aligned}
H_{t}(t,x,s_{1},m_{1})=&\int\iota^{\gamma}\left(Y^{\gamma}(t,x,s_{1},m_{1})\right)Y_{t}^{\gamma}(t,x,s_{1},m_{1})\mathrm{d}\Gamma(\gamma),\\
H_{x}(t,x,s_{1},m_{1})=&\int\iota^{\gamma}(Y^{\gamma}(t,x,s_{1},m_{1}))Y_{x}^{\gamma}(t,x,s_{1},m_{1})\mathrm{d}\Gamma(\gamma),\\
H_{s_{1}}(t,x,s_{1},m_{1})=&\int\iota^{\gamma}(Y^{\gamma}(t,x,s_{1},m_{1}))Y_{s_{1}}^{\gamma}(t,x,s_{1},m_{1})\mathrm{d}\Gamma(\gamma),\\
H_{m_{1}}(t,x,s_{1},m_{1})=&\int\iota^{\gamma}(Y^{\gamma}(t,x,s_{1},m_{1}))Y_{m_{1}}^{\gamma}(t,x,s_{1},m_{1})\mathrm{d}\Gamma(\gamma),\\
H_{xx}(t,x,s_{1},m_{1})=&\int\iota^{\gamma}(Y^{\gamma}(t,x,s_{1},m_{1}))Y_{xx}^{\gamma}(t,x,s_{1},m_{1})\mathrm{d}\Gamma(\gamma)\\
&+\int(\iota^{\gamma})^{'}(Y^{\gamma}(t,x,s_{1},m_{1}))(Y_{x}^{\gamma}(t,x,s_{1},m_{1}))^{2}\mathrm{d}\Gamma(\gamma),\\
H_{s_{1}s_{1}}(t,x,s_{1},m_{1})=&\int\iota^{\gamma}(Y^{\gamma}(t,x,s_{1},m_{1}))Y_{s_{1}s_{1}}^{\gamma}(t,x,s_{1},m_{1})\mathrm{d}\Gamma(\gamma)\\
&+\int(\iota^{\gamma})^{'}(Y^{\gamma}(t,x,s_{1},m_{1}))(Y_{s_{1}}^{\gamma}(t,x,s_{1},m_{1}))^{2}\mathrm{d}\Gamma(\gamma),\\
H_{xs_{1}}(t,x,s_{1},m_{1})=&\int\iota^{\gamma}(Y^{\gamma}(t,x,s_{1},m_{1}))Y_{xs_{1}}^{\gamma}(t,x,s_{1},m_{1})\mathrm{d}\Gamma(\gamma)\\
&+\int(\iota^{\gamma})^{'}(Y^{\gamma}(t,x,s_{1},m_{1}))Y_{x}^{\gamma}(t,x,s_{1},m_{1})Y_{s_{1}}^{\gamma}(t,x,s_{1},m_{1})\mathrm{d}\Gamma(\gamma),
\end{aligned}
\right.
\end{eqnarray}
where $(\iota^{\gamma})^{'}$ represents the derivative of $\iota^{\gamma}$. Therefore, by substituting the system of equations \eqref{eq35} into the pseudo HJB equation \eqref{eq20}, the equation \eqref{eq20} admits the following equivalent representation
\begin{align*}
  \sup\limits_{u} \;&\bigg\{\mathcal{A}^{u}U(t,x,s_{1},m_{1})-0.5\sigma^{2}s^{2 \delta+2}_{1}\int(\iota^{\gamma})^{'}(Y^{\gamma}(t,x,s_{1},m_{1}))(Y_{s_{1}}^{\gamma}(t,x,s_{1},m_{1}))^{2}\mathrm{d}\Gamma(\gamma)\nonumber\\
  &\quad-0.5(b^{2}q^{2}+\pi^{2}x^{2}\sigma^{2}s^{2 \delta}_{1})\int(\iota^{\gamma})^{'}(Y^{\gamma}(t,x,s_{1},m_{1}))(Y_{x}^{\gamma}(t,x,s_{1},m_{1}))^{2}\mathrm{d}\Gamma(\gamma)\nonumber\\
  &\quad-\pi x\sigma^{2}s^{2 \delta+1}_{1}\int(\iota^{\gamma})^{'}(Y^{\gamma}(t,x,s_{1},m_{1}))Y_{x}^{\gamma}(t,x,s_{1},m_{1})Y_{s_{1}}^{\gamma}(t,x,s_{1},m_{1})\mathrm{d}\Gamma(\gamma)\bigg\}=0.
\end{align*}

After deriving the above pseudo HJB equation, we can take the partial derivatives of the terms inside the brackets of the above equation with respect to $q$ and $\pi$ individually. This procedure enables us to get
\begin{equation}\label{eq72}
  \hat{q}(t)=-\frac{a \eta_{2}U_{x}(t,x,s_{1},m_{1})}{b^{2}\left[U_{xx}(t,x,s_{1},m_{1})-\int(\iota^{\gamma})^{'}(Y^{\gamma}(t,x,s_{1},m_{1}))(Y_{x}^{\gamma}(t,x,s_{1},m_{1}))^{2}\mathrm{d}\Gamma(\gamma)\right]}
\end{equation}
and
\begin{align}\label{eq73}
\hat{\pi}(t)=&-\frac{(\mu- r)U_{x}(t,x,s_{1},m_{1})}
{x\sigma^{2}s^{2\delta}_{1}\left[U_{xx}(t,x,s_{1},m_{1})-\int(\iota^{\gamma})^{'}(Y^{\gamma}(t,x,s_{1},m_{1}))(Y_{x}^{\gamma}(t,x,s_{1},m_{1}))^{2}\mathrm{d}\Gamma(\gamma)\right]}\nonumber\\
&-\frac{s_{1}\left[U_{xs_{1}}(t,x,s_{1},m_{1})-\int(\iota^{\gamma})^{'}(Y^{\gamma}(t,x,s_{1},m_{1}))Y_{x}^{\gamma}(t,x,s_{1},m_{1})Y_{s_{1}}^{\gamma}(t,x,s_{1},m_{1})\mathrm{d}\Gamma(\gamma)\right]}
{x\left[U_{xx}(t,x,s_{1},m_{1})-\int(\iota^{\gamma})^{'}(Y^{\gamma}(t,x,s_{1},m_{1}))(Y_{x}^{\gamma}(t,x,s_{1},m_{1}))^{2}\mathrm{d}\Gamma(\gamma)\right]},
\end{align}
where we denote by $(\hat{q}(t),\hat{\pi}(t)):=(\hat{q}(t,x,s_{1},m_{1}),\hat{\pi}(t,x,s_{1},m_{1}))$ the candidate equilibrium reinsurance and investment strategy.

Due to the equation \eqref{eq28}, we know that $
U(t,x,s_{1},m_{1})=H(t,x,s_{1},m_{1})
=\int(\varphi^{\gamma})^{-1}\left(Y^{\gamma}(t,x,s_{1},m_{1})\right)\mathrm{d}\Gamma(\gamma)$. Therefore, equations \eqref{eq72} and \eqref{eq73} take the following forms
\begin{equation}\label{eq74}
  \hat{q}(t)=-\frac{a \eta_{2}\int\iota^{\gamma}(Y^{\gamma}(t,x,s_{1},m_{1}))Y_{x}^{\gamma}(t,x,s_{1},m_{1})\mathrm{d}\Gamma(\gamma)}{b^{2}\int\iota^{\gamma}(Y^{\gamma}(t,x,s_{1},m_{1}))
  Y_{xx}^{\gamma}(t,x,s_{1},m_{1})\mathrm{d}\Gamma(\gamma)}
\end{equation}
and
\begin{align}\label{eq75}
   \hat{\pi}(t)=&-\frac{(\mu- r)\int\iota^{\gamma}(Y^{\gamma}(t,x,s_{1},m_{1}))Y_{x}^{\gamma}(t,x,s_{1},m_{1})\mathrm{d}\Gamma(\gamma)}
   {x\sigma^{2}s^{2\delta}_{1}\int\iota^{\gamma}(Y^{\gamma}(t,x,s_{1},m_{1}))Y_{xx}^{\gamma}(t,x,s_{1},m_{1})\mathrm{d}\Gamma(\gamma)}\nonumber\\
   &-\frac{s_{1}\int\iota^{\gamma}(Y^{\gamma}(t,x,s_{1},m_{1}))Y_{xs_{1}}^{\gamma}(t,x,s_{1},m_{1})\mathrm{d}\Gamma(\gamma)}
   {x\int\iota^{\gamma}(Y^{\gamma}(t,x,s_{1},m_{1}))Y_{xx}^{\gamma}(t,x,s_{1},m_{1})\mathrm{d}\Gamma(\gamma)}.
\end{align}
Hereafter, our analysis centers on characterizing $Y^{\gamma}(t,x,s_{1},m_{1})$ for all $\gamma$.

Under the assumption that $\varphi^{\gamma}$ follows an exponential utility
\begin{equation}\label{eq49}
\varphi^{\gamma}(\cdot)=-\frac{1}{\gamma}e^{-\gamma (\cdot)},
\end{equation}
where $\gamma>0$, one can give its inverse function and the derivative of the inverse function as follows
$$
(\varphi^{\gamma})^{-1}(\cdot)=-\frac{1}{\gamma}\ln(-\gamma (\cdot)),\quad
\frac{\mathrm{d}}{\mathrm{d} y}(\varphi^{\gamma})^{-1}(y)=-\frac{1}{\gamma y}.
$$
Consequently, equations \eqref{eq74} and \eqref{eq75} can be rewritten as
$$
   \hat{q}(t)=-\frac{a \eta_{2} \int \frac{Y_{x}^{\gamma}(t,x,s_{1},m_{1})}{\gamma Y^{\gamma}(t,x,s_{1},m_{1})}\mathrm{d}\Gamma(\gamma)}
{b^{2}\int \frac{Y_{xx}^{\gamma}(t,x,s_{1},m_{1})}{\gamma Y^{\gamma}(t,x,s_{1},m_{1})}\mathrm{d}\Gamma(\gamma)},\;
   \hat{\pi}(t)=-\frac{(\mu- r) \int \frac{Y_{x}^{\gamma}(t,x,s_{1},m_{1})}{\gamma Y^{\gamma}(t,x,s_{1},m_{1})}\mathrm{d}\Gamma(\gamma)}
{x\sigma^{2}s^{2\delta}_{1}\int \frac{Y_{xx}^{\gamma}(t,x,s_{1},m_{1})}{\gamma Y^{\gamma}(t,x,s_{1},m_{1})}\mathrm{d}\Gamma(\gamma)}
-\frac{s^{}_{1}
   \int \frac{Y_{xs_{1}}^{\gamma}(t,x,s_{1},m_{1})}{\gamma Y^{\gamma}(t,x,s_{1},m_{1})}\mathrm{d}\Gamma(\gamma)}
{x\int \frac{Y_{xx}^{\gamma}(t,x,s_{1},m_{1})}{\gamma Y^{\gamma}(t,x,s_{1},m_{1})}\mathrm{d}\Gamma(\gamma)}.
$$

Now, we propose an ansatz $Y^{\gamma}(t,x,s_{1},m_{1})=-\frac{1}{\gamma}e^{\widehat{g}_{1}^{\gamma}(t)(x+\beta m_{1})+\widehat{g}_{2}^{\gamma}(t)+\widehat{g}_{3}^{\gamma}(t,s_{1})}$, where $\widehat{g}_{1}^{\gamma}(t)$ and $\widehat{g}_{2}^{\gamma}(t)$ are deterministic functions of $t\in[0,T]$ with $\widehat{g}_{1}^{\gamma}(T)=-\gamma$ and $\widehat{g}_{2}^{\gamma}(T)=0$, and $\widehat{g}_{3}^{\gamma}(t,s_{1})$ is a deterministic function of $t\in[0,T]$ and $s_{1}\in[0,\infty)$ with $\widehat{g}_{3}^{\gamma}(T,s_{1})=0$. It should be noted that the superscript $\gamma$ on functions $\widehat{g}_{1}^{\gamma}(t)$, $\widehat{g}_{2}^{\gamma}(t)$ and $\widehat{g}_{3}^{\gamma}(t,s_{1})$ signifies their dependence on $\gamma$. These notations do not imply that $\widehat{g}_{1}(t)$, $\widehat{g}_{2}(t)$ and $\widehat{g}_{3}(t,s_{1})$ are raised to the $\gamma$-th power. Similar notations will be used in Section 5 when there is no ambiguity. Thus, one can derive the subsequent derivatives
\begin{eqnarray*}
\left\{
\begin{aligned}
   &Y^{\gamma}_{t}(t,x,s_{1},m_{1})=Y^{\gamma}(t,x,s_{1},m_{1})\left[\frac{\partial \widehat{g}_{1}^{\gamma}(t)}{\partial t}\left(x+\beta m_{1}\right)+\frac{\partial \widehat{g}_{2}^{\gamma}(t)}{\partial t}+\frac{\partial \widehat{g}_{3}^{\gamma}(t,s_{1})}{\partial t}\right],\nonumber\\
   &Y^{\gamma}_{x}(t,x,s_{1},m_{1})=\widehat{g}_{1}^{\gamma}(t)Y^{\gamma}(t,x,s_{1},m_{1}),\quad
   Y^{\gamma}_{xx}(t,x,s_{1},m_{1})=\left(\widehat{g}_{1}^{\gamma}(t)\right)^{2}Y^{\gamma}(t,x,s_{1},m_{1}),
   \nonumber\\
   &Y^{\gamma}_{s_{1}}(t,x,s_{1},m_{1})=Y^{\gamma}(t,x,s_{1},m_{1})\frac{\partial \widehat{g}_{3}^{\gamma}(t,s_{1})}{\partial s_{1}},\;
   Y^{\gamma}_{xs_{1}}(t,x,s_{1},m_{1})=\widehat{g}_{1}^{\gamma}(t)Y^{\gamma}(t,x,s_{1},m_{1})\frac{\partial \widehat{g}_{3}^{\gamma}(t,s_{1})}{\partial s_{1}},\nonumber\\
   &Y^{\gamma}_{s_{1}s_{1}}(t,x,s_{1},m_{1})=Y^{\gamma}(t,x,s_{1},m_{1})\left[\frac{\partial^{2} \widehat{g}_{3}^{\gamma}(t,s_{1})}{\partial s^{2}_{1}}+\left(\frac{\partial \widehat{g}_{3}^{\gamma}(t,s_{1})}{\partial s_{1}}\right)^{2}\right],\nonumber\\
   &Y^{\gamma}_{m_{1}}(t,x,s_{1},m_{1})=\beta\widehat{g}_{1}^{\gamma}(t)Y^{\gamma}(t,x,s_{1},m_{1}).
\end{aligned}
\right.
\end{eqnarray*}
Moreover,  the candidate equilibrium reinsurance and investment strategy can be given by
\begin{align}\label{eq76}
   \hat{q}(t)=-\frac{a \eta_{2} \int \frac{\widehat{g}_{1}^{\gamma}(t)}{\gamma }\mathrm{d}\Gamma(\gamma)}
{b^{2}\int \frac{(\widehat{g}_{1}^{\gamma}(t))^{2}}{\gamma }\mathrm{d}\Gamma(\gamma)},\quad
   \hat{\pi}(t)=-\frac{(\mu- r) \int \frac{\widehat{g}_{1}^{\gamma}(t)}{\gamma }\mathrm{d}\Gamma(\gamma)}
{x\sigma^{2}s^{2\delta}_{1}\int \frac{(\widehat{g}_{1}^{\gamma}(t))^{2}}{\gamma }\mathrm{d}\Gamma(\gamma)}
-\frac{s^{}_{1}
   \int \frac{\widehat{g}_{1}^{\gamma}(t)}{\gamma }\frac{\partial \widehat{g}_{3}^{\gamma}(t,s_{1})}{\partial s_{1}}\mathrm{d}\Gamma(\gamma)}
{x\int \frac{(\widehat{g}_{1}^{\gamma}(t))^{2}}{\gamma }\mathrm{d}\Gamma(\gamma)}.
\end{align}
Through the insertion of the partial derivatives of $Y^{\gamma}(t,x,s_{1},m_{1})$ into \eqref{eq21}, we arrive at
\begin{align*}
  0=&\frac{\partial \widehat{g}_{1}^{\gamma}(t)}{\partial t}(x+\beta m_{1})+\frac{\partial \widehat{g}_{2}^{\gamma}(t)}{\partial t}+\frac{\partial \widehat{g}_{3}^{\gamma}(t,s_{1})}{\partial t}
  +\left[(A+\hat{\pi}(\mu-r))x+Bm_{1}+Cm_{2}+a\eta+a\eta_{2}\hat{q}\right]\widehat{g}_{1}^{\gamma}(t)\nonumber\\
&+0.5(b^{2}\hat{q}^{2}+\hat{\pi}^{2}x^{2}\sigma^{2}s^{2\delta}_{1})(\widehat{g}_{1}^{\gamma}(t))^{2}+\mu s_{1}\frac{\partial \widehat{g}_{3}^{\gamma}(t,s_{1})}{\partial s_{1}}
+0.5\sigma^{2}s^{2\delta+2}_{1}\left[\frac{\partial^{2} \widehat{g}_{3}^{\gamma}(t,s_{1})}{\partial s^{2}_{1}}+\left(\frac{\partial \widehat{g}_{3}^{\gamma}(t,s_{1})}{\partial s_{1}}\right)^{2}\right]\nonumber\\
&+\hat{\pi}x\sigma^{2}s^{2\delta+1}_{1}\widehat{g}_{1}^{\gamma}(t)\frac{\partial \widehat{g}_{3}^{\gamma}(t,s_{1})}{\partial s_{1}}
+\left(x-\alpha m_{1}-e^{-\alpha h}m_{2}\right)\beta \widehat{g}_{1}^{\gamma}(t),
\end{align*}
where $\hat{q}$ and $\hat{\pi}$ are presented by the equation \eqref{eq76}. By adopting the following assumption
\begin{itemize}
  \item[$(A2)$] $C=\beta e^{-\alpha h}$ and $B e^{-\alpha h}=(\alpha+A+\beta)C$,
\end{itemize}
the above equation becomes
\begin{align}\label{eq77}
  0=&\frac{\partial \widehat{g}_{1}^{\gamma}(t)}{\partial t}(x+\beta m_{1})+\frac{\partial \widehat{g}_{2}^{\gamma}(t)}{\partial t}+\frac{\partial \widehat{g}_{3}^{\gamma}(t,s_{1})}{\partial t}
  +\left[\hat{\pi}(\mu-r)x+a\eta+a\eta_{2}\hat{q}\right]\widehat{g}_{1}^{\gamma}(t)+\mu s_{1}\frac{\partial \widehat{g}_{3}^{\gamma}(t,s_{1})}{\partial s_{1}}\nonumber\\
&+0.5(b^{2}\hat{q}^{2}+\hat{\pi}^{2}x^{2}\sigma^{2}s^{2\delta}_{1})(\widehat{g}_{1}^{\gamma}(t))^{2}
+0.5\sigma^{2}s^{2\delta+2}_{1}\left[\frac{\partial^{2} \widehat{g}_{3}^{\gamma}(t,s_{1})}{\partial s^{2}_{1}}+\left(\frac{\partial \widehat{g}_{3}^{\gamma}(t,s_{1})}{\partial s_{1}}\right)^{2}\right]\nonumber\\
&+\hat{\pi}x\sigma^{2}s^{2\delta+1}_{1}\widehat{g}_{1}^{\gamma}(t)\frac{\partial \widehat{g}_{3}^{\gamma}(t,s_{1})}{\partial s_{1}}
+\left(A+\beta\right)\left(x+\beta m_{1}\right) \widehat{g}_{1}^{\gamma}(t).
\end{align}

\begin{remark}
Note that the assumption $(A2)$ can be used to ensure the existence of a (semi-)closed-form solution for the equation \eqref{eq21}. Although such an assumption limits the generality to some extent, it is beneficial to transform the delayed reinsurance and investment optimization problem into a finite-dimensional one. Furthermore, it is a sufficient condition to make sure that $V$ depends solely on $(t,x,s_{1},m_{1})$.
\end{remark}

Next, letting $a_{1}=\frac{\int \frac{\widehat{g}_{1}^{\gamma}(t)}{\gamma }\mathrm{d}\Gamma(\gamma)}
{\int \frac{(\widehat{g}_{1}^{\gamma}(t))^{2}}{\gamma }\mathrm{d}\Gamma(\gamma)}$ and $a_{2}=\frac{
   \int \frac{\widehat{g}_{1}^{\gamma}(t)}{\gamma }\frac{\partial \widehat{g}_{3}^{\gamma}(t,s_{1})}{\partial s_{1}}\mathrm{d}\Gamma(\gamma)}
{\int \frac{(\widehat{g}_{1}^{\gamma}(t))^{2}}{\gamma }\mathrm{d}\Gamma(\gamma)}$, one can substitute the equation \eqref{eq76} into the equation \eqref{eq77} to obtain
\begin{align*}
  0=&\frac{\partial \widehat{g}_{1}^{\gamma}(t)}{\partial t}(x+\beta m_{1})+\frac{\partial \widehat{g}_{2}^{\gamma}(t)}{\partial t}+\frac{\partial \widehat{g}_{3}^{\gamma}(t,s_{1})}{\partial t}
  +\left[-\frac{(\mu-r)^{2}a_{1}}{\sigma^{2}s^{2\delta}_{1}}-(\mu-r)s_{1}a_{2}
  +a\eta-\frac{a^{2}\eta^{2}_{2}a_{1}}{b^{2}}\right]\widehat{g}_{1}^{\gamma}(t)\nonumber\\
&+\mu s_{1}\frac{\partial \widehat{g}_{3}^{\gamma}(t,s_{1})}{\partial s_{1}}+0.5\left(\frac{a^{2}\eta^{2}_{2}a^{2}_{1}}{b^{2}}+\frac{(\mu-r)^{2}a^{2}_{1}}{\sigma^{2}s^{2\delta}_{1}}+2(\mu-r)s_{1}a_{1}a_{2}+
\sigma^{2}s^{2\delta+2}_{1}a^{2}_{2}\right)(\widehat{g}_{1}^{\gamma}(t))^{2}
\nonumber\\
&+0.5\sigma^{2}s^{2\delta+2}_{1}\left[\frac{\partial^{2} \widehat{g}_{3}^{\gamma}(t,s_{1})}{\partial s^{2}_{1}}+\left(\frac{\partial \widehat{g}_{3}^{\gamma}(t,s_{1})}{\partial s_{1}}\right)^{2}\right]
+\left(A+\beta\right)\left(x+\beta m_{1}\right) \widehat{g}_{1}^{\gamma}(t)\nonumber\\
&-\left[(\mu-r)s_{1}a_{1}+\sigma^{2}s^{2\delta+2}_{1}a_{2}\right]\widehat{g}_{1}^{\gamma}(t)\frac{\partial \widehat{g}_{3}^{\gamma}(t,s_{1})}{\partial s_{1}}.
\end{align*}
By the method of comparing coefficients, we can decompose the above equation into $\frac{\partial \widehat{g}_{1}^{\gamma}(t)}{\partial t}+(A+\beta) \widehat{g}_{1}^{\gamma}(t)=0$, $\frac{\partial \widehat{g}_{2}^{\gamma}(t)}{\partial t}+\left(a \eta-\frac{a^{2}\eta^{2}_{2}a_{1}}{b^{2}}\right)\widehat{g}_{1}^{\gamma}(t)+0.5\frac{a^{2}\eta^{2}_{2}a^{2}_{1}}{b^{2}}(\widehat{g}_{1}^{\gamma}(t))^{2}=0$ and
\begin{align*}
\frac{\partial \widehat{g}_{3}^{\gamma}(t,s_{1})}{\partial t}&+\left[-\frac{(\mu-r)^{2}a_{1}}{\sigma^{2}s^{2\delta}_{1}}-(\mu-r)s_{1}a_{2}\right]\widehat{g}_{1}^{\gamma}(t)+\mu s_{1}\frac{\partial \widehat{g}_{3}^{\gamma}(t,s_{1})}{\partial s_{1}}+0.5\left[\frac{(\mu-r)^{2}a^{2}_{1}}{\sigma^{2}s^{2\delta}_{1}}\right.\nonumber \\
&+2(\mu-r)s_{1}a_{1}a_{2}+
\sigma^{2}s^{2\delta+2}_{1}a^{2}_{2}\bigg](\widehat{g}_{1}^{\gamma}(t))^{2}+0.5\sigma^{2}s^{2\delta+2}_{1}\left[\frac{\partial^{2} \widehat{g}_{3}^{\gamma}(t,s_{1})}{\partial s^{2}_{1}}+\left(\frac{\partial \widehat{g}_{3}^{\gamma}(t,s_{1})}{\partial s_{1}}\right)^{2}\right]\nonumber \\
&-\left[(\mu-r)s_{1}a_{1}+\sigma^{2}s^{2\delta+2}_{1}a_{2}\right]\widehat{g}_{1}^{\gamma}(t)\frac{\partial \widehat{g}_{3}^{\gamma}(t,s_{1})}{\partial s_{1}}=0.
\end{align*}
Since $\widehat{g}_{1}^{\gamma}(T)=-\gamma$, it is easy to derive that $\widehat{g}_{1}^{\gamma}(t)=-\gamma e^{(A+\beta)(T-t)}$ and $a_{1}=-\frac{1}{\mathbb{E}[\gamma]}e^{-(A+\beta)(T-t)}$ with $\mathbb{E}[\gamma]=\int\gamma \mathrm{d}\Gamma(\gamma)$. Moreover, by $\widehat{g}_{2}^{\gamma}(T)=0$, one can derive that $\widehat{g}_{2}^{\gamma}(t)=\frac{\gamma a \eta}{A+\beta}\left[1-e^{(A+\beta)(T-t)}\right]+\frac{a^{2}\eta^{2}_{2}}{b^{2}}(\frac{0.5\gamma^{2}}{\mathbb{E}^{2}[\gamma]}-\frac{\gamma}{\mathbb{E}[\gamma]})(T-t)$. Then, we assume that $\widehat{g}_{3}^{\gamma}(t,s_{1})=\widehat{g}_{4}^{\gamma}(t,z)$ with $\widehat{g}_{4}^{\gamma}(T,z)=0$, where $z=s^{-2 \delta}_{1}$. Therefore, $\frac{\partial \widehat{g}_{3}^{\gamma}(t,s_{1})}{\partial t}=\frac{\partial \widehat{g}_{4}^{\gamma}(t,z)}{\partial t}$, $\frac{\partial \widehat{g}_{3}^{\gamma}(t,s_{1})}{\partial s_{1}}=-2\delta s^{-2 \delta-1}_{1}\frac{\partial \widehat{g}_{4}^{\gamma}(t,z)}{\partial z}$, $\frac{\partial^{2} \widehat{g}_{3}^{\gamma}(t,s_{1})}{\partial s^{2}_{1}}=2\delta(2\delta+1) s^{-2 \delta-2}_{1}\frac{\partial \widehat{g}_{4}^{\gamma}(t,z)}{\partial z}+4\delta^{2} s^{-4 \delta-2}_{1}\frac{\partial^{2} \widehat{g}_{4}^{\gamma}(t,z)}{\partial z^{2}}$ and $a_{2}=\frac{2\delta s^{-2 \delta-1}_{1}}{\mathbb{E}[\gamma]}\int \frac{\partial \widehat{g}_{4}^{\gamma}(t,z)}{\partial z}\mathrm{d}\Gamma(\gamma)e^{-(A+\beta)(T-t)}$. Thus, the equation for $\widehat{g}_{3}^{\gamma}(t,s_{1})$ can be rewritten as
\begin{align*}
&\frac{\partial \widehat{g}_{4}^{\gamma}(t,z)}{\partial t}-\frac{(\mu-r)^{2}\gamma z}{\sigma^{2}\mathbb{E}[\gamma]}+\frac{2(\mu-r)\delta\gamma z}{\mathbb{E}[\gamma]}\int \frac{\partial \widehat{g}_{4}^{\gamma}(t,z)}{\partial z}\mathrm{d}\Gamma(\gamma)-2\mu\delta z\frac{\partial \widehat{g}_{4}^{\gamma}(t,z)}{\partial z}+\frac{0.5(\mu-r)^{2}\gamma^{2} z}{\sigma^{2}\mathbb{E}^{2}[\gamma]}\nonumber \\
&\quad\quad-\frac{2(\mu-r)\delta\gamma^{2} z}{\mathbb{E}^{2}[\gamma]}\int \frac{\partial \widehat{g}_{4}^{\gamma}(t,z)}{\partial z}\mathrm{d}\Gamma(\gamma)
+\frac{2\delta^{2}\sigma^{2}\gamma^{2} z}{\mathbb{E}^{2}[\gamma]}\left(\int \frac{\partial \widehat{g}_{4}^{\gamma}(t,z)}{\partial z}\mathrm{d}\Gamma(\gamma)\right)^{2}
+2\delta^{2}\sigma^{2}z\frac{\partial^{2} \widehat{g}_{4}^{\gamma}(t,z)}{\partial z^{2}}\nonumber \\
&\quad\quad +\delta(2\delta+1)\sigma^{2}\frac{\partial \widehat{g}_{4}^{\gamma}(t,z)}{\partial z}+2\delta^{2}\sigma^{2}z\left(\frac{\partial \widehat{g}_{4}^{\gamma}(t,z)}{\partial z}\right)^{2}+\frac{2\delta \gamma z}{\mathbb{E}[\gamma]}\left[\mu-r-2\sigma^{2} \delta\int \frac{\partial \widehat{g}_{4}^{\gamma}(t,z)}{\partial z}\mathrm{d}\Gamma(\gamma)\right]\frac{\partial \widehat{g}_{4}^{\gamma}(t,z)}{\partial z}=0.
\end{align*}
In addition, we assume that $\widehat{g}_{4}^{\gamma}(t,z)=\widehat{g}_{5}^{\gamma}(t)+\widehat{g}_{6}^{\gamma}(t)z$ with $\widehat{g}_{5}^{\gamma}(T)=\widehat{g}_{6}^{\gamma}(T)=0$. By substituting the assumed solution form for $\widehat{g}_{4}^{\gamma}(t,z)$ into the above equation and using the method of comparing coefficients, we can obtain that $\frac{\partial \widehat{g}_{5}^{\gamma}(t)}{\partial t}+\delta(2\delta+1)\sigma^{2}\widehat{g}_{6}^{\gamma}(t)=0$ and
\begin{align}\label{eq78}
\frac{\partial \widehat{g}_{6}^{\gamma}(t)}{\partial t}&+\frac{(\mu-r)^{2}}{\sigma^{2}}\left[\frac{0.5\gamma^{2}}{\mathbb{E}^{2}[\gamma]}-\frac{\gamma}{\mathbb{E}[\gamma]}\right]-2(\mu-r)\delta \int \widehat{g}_{6}^{\gamma}(t)\mathrm{d}\Gamma(\gamma)\left[\frac{\gamma^{2}}{\mathbb{E}^{2}[\gamma]}-\frac{\gamma}{\mathbb{E}[\gamma]}\right]+2\delta^{2}\sigma^{2}\left( \widehat{g}_{6}^{\gamma}(t)\right)^{2}\nonumber \\
&+\frac{2\delta^{2}\sigma^{2}\gamma^{2}}{\mathbb{E}^{2}[\gamma]}\left(\int \widehat{g}_{6}^{\gamma}(t)\mathrm{d}\Gamma(\gamma)\right)^{2}-2\delta\widehat{g}_{6}^{\gamma}(t)\left[\mu-\frac{(\mu-r)\gamma}{\mathbb{E}[\gamma]}+
\frac{2\sigma^{2}\delta\gamma}{\mathbb{E}[\gamma]}\int \widehat{g}_{6}^{\gamma}(t)\mathrm{d}\Gamma(\gamma)\right]=0.
\end{align}

We note that the equation \eqref{eq78} is infinite-dimensional on account of the distribution of $\Gamma$. Although (semi-)closed-form solutions for $\widehat{g}_{6}^{\gamma}(t)$ may be derived under specific distributional assumptions regarding $\Gamma$, the general scenario defies analytical handling. Consequently, we seek to discretize $\Gamma$, the integral of $\widehat{g}_{6}^{\gamma}(t)$ is converted into a finite sum, thereby converting the ODE into a form that is more amenable to analysis.

\subsection{Special case with $n$ possible risk aversions}
We assume that the random risk aversion conforms to an $n$-point discrete distribution, featuring possible positive outcomes $\gamma_{i}\;(i=1,\cdot\cdot\cdot,n)$ and their respective probabilities $p_{i}=p(\gamma_{i})$ with $\sum\limits_{i=1}^{n}p_{i}=1$. Then, for $i=1,\cdot\cdot\cdot,n$, the equations of $\widehat{g}_{5}^{\gamma_{i}}(t)$ and $\widehat{g}_{6}^{\gamma_{i}}(t)$ can be given by $\frac{\partial \widehat{g}_{5}^{\gamma_{i}}(t)}{\partial t}+\delta(2\delta+1)\sigma^{2}\widehat{g}_{6}^{\gamma_{i}}(t)=0$ and
\begin{align}\label{eq79}
-\frac{\partial \widehat{g}_{6}^{\gamma_{i}}(t)}{\partial t}=&\frac{(\mu-r)^{2}}{\sigma^{2}}\left(0.5a^{2}_{3i}-a_{3i}\right)-2(\mu-r)\delta\left(a^{2}_{3i}-a_{3i}\right) \sum\limits_{i=1}^{n}\widehat{g}_{6}^{\gamma_{i}}(t)p_{i}+2\delta^{2}\sigma^{2}\left( \widehat{g}_{6}^{\gamma_{i}}(t)\right)^{2}\nonumber \\
&+2\delta^{2}\sigma^{2}a^{2}_{3i}\left(\sum\limits_{i=1}^{n}\widehat{g}_{6}^{\gamma_{i}}(t)p_{i}\right)^{2}-2\delta\widehat{g}_{6}^{\gamma_{i}}(t)\left[\mu-(\mu-r)a_{3i}+
2\sigma^{2}\delta a_{3i}\sum\limits_{i=1}^{n}\widehat{g}_{6}^{\gamma_{i}}(t)p_{i}\right],
\end{align}
where $a_{3i}=\frac{\gamma_{i}}{\sum\limits_{i=1}^{n}\gamma_{i}p_{i}}$. It is not difficult to discern that the right-hand side of \eqref{eq79} is a quadratic polynomial in
$\widehat{g}_{6}^{\gamma_{i}}(t)\;(i=1,\cdot\cdot\cdot,n)$, whose coefficient functions are bounded and Lipschitz continuous with regard to $t$. Therefore, by the uniqueness theorem \cite{Zwillinger1998}, the equation \eqref{eq79} has a unique solution. Subsequently, the uniqueness of the solution for $\widehat{g}_{5}^{\gamma_{i}}(t)$ can be secured. Furthermore, the candidate equilibrium reinsurance and investment strategy can be characterized by
\begin{equation}\label{eq80}
\hat{q}(t)=\frac{a \eta_{2}}{b^{2}\sum\limits_{i=1}^{n}\gamma_{i}p_{i}}e^{-(A+\beta)(T-t)},\quad
   \hat{\pi}(t)=\left(\frac{\mu-r}{\sigma^{2}}-2\delta\sum\limits_{i=1}^{n}\widehat{g}_{6}^{\gamma_{i}}(t)p_{i}\right)
   \frac{e^{-(A+\beta)(T-t)}}{xs^{2\delta}_{1}\sum\limits_{i=1}^{n}\gamma_{i}p_{i}}.
\end{equation}

Under mild conditions, it will be showed that assumptions of Theorem \ref{theorem3.1} are met and the candidate strategy \eqref{eq80} is admissible, thereby establishing it as an equilibrium reinsurance and investment strategy.

\begin{theorem}\label{theorem4.4}
Assume that the utility function $\varphi^{\gamma}$ satisfies the equation \eqref{eq49} and the assumption $(A2)$ holds. Moreover, assume that there are $n$ possible risk aversions $\gamma_{i}\in[\epsilon_{1},\infty)$ with $0<\epsilon_{1}<\infty$ for $i=1,\cdot\cdot\cdot,n$, the solution to the equation \eqref{eq79} is non-positive and the parameters satisfy $-32\gamma_{i}(\mu-r)\delta^{2}\widetilde{\pi}(s)
  +128\gamma_{i}^{2}\delta^{2}\sigma^{2}\widetilde{\pi}^{2}(s)\leq\frac{\kappa^{2}}{2\sigma^{2}}$ for $s\in[0,T]$, where $\widetilde{\pi}(s)$ is given by the equation \eqref{eq84}. Then the equilibrium reinsurance and investment strategy satisfies the equation \eqref{eq80} and the associated equilibrium value function is
$V(t,x,s_{1},m_{1})=-\sum\limits_{i=1}^{n}\frac{1}{\gamma_{i}}\left[\widehat{g}_{1}^{\gamma_{i}}(t)(x+\beta m_{1})+\widehat{g}_{2}^{\gamma_{i}}(t)+\widehat{g}_{5}^{\gamma_{i}}(t)+\widehat{g}_{6}^{\gamma_{i}}(t)s^{-2\delta}_{1}\right]p_{i}$, where $\widehat{g}_{1}^{\gamma_{i}}(t)=-\gamma_{i} e^{(A+\beta)(T-t)}$, $\widehat{g}_{2}^{\gamma_{i}}(t)=\frac{\gamma_{i} a \eta}{A+\beta}\left[1-e^{(A+\beta)(T-t)}\right]+\frac{a^{2}\eta^{2}_{2}}{b^{2}}(0.5a^{2}_{3i}-a_{3i})(T-t)$, $\widehat{g}_{5}^{\gamma_{i}}(t)$ is given by $\frac{\partial \widehat{g}_{5}^{\gamma_{i}}(t)}{\partial t}+\delta(2\delta+1)\sigma^{2}\widehat{g}_{6}^{\gamma_{i}}(t)=0$ and $\widehat{g}_{6}^{\gamma_{i}}(t)$ meets the equation \eqref{eq79}.
\end{theorem}
\begin{proof}
See Appendix D.
\end{proof}

\begin{corollary}
If $\alpha=h=B=C=\beta=0$, which implies that delay is not embedded in the model, then the equilibrium reinsurance and investment strategy is
$$
\hat{q}(t)=\frac{a \eta_{2}}{b^{2}\sum\limits_{i=1}^{n}\gamma_{i}p_{i}}e^{-r(T-t)},\quad
   \hat{\pi}(t)=\left(\frac{\mu-r}{\sigma^{2}}-2\delta\sum\limits_{i=1}^{n}\widehat{g}_{6}^{\gamma_{i}}(t)p_{i}\right)
   \frac{e^{-r(T-t)}}{xs^{2\delta}_{1}\sum\limits_{i=1}^{n}\gamma_{i}p_{i}}
$$
and the equilibrium value function becomes $
V(t,x,s_{1},m_{1})=-\sum\limits_{i=1}^{n}\frac{1}{\gamma_{i}}\left[\widehat{g}_{1}^{\gamma_{i}}(t)x+\widehat{g}_{2}^{\gamma_{i}}(t)+\widehat{g}_{5}^{\gamma_{i}}(t)+\widehat{g}_{6}^{\gamma_{i}}(t)s^{-2\delta}_{1}\right]p_{i}$,
where $\widehat{g}_{1}^{\gamma_{i}}(t)=-\gamma_{i} e^{r(T-t)}$, $\widehat{g}_{2}^{\gamma_{i}}(t)=\frac{\gamma_{i} a \eta}{r}\left[1-e^{r(T-t)}\right]+\frac{a^{2}\eta^{2}_{2}}{b^{2}}(0.5a^{2}_{3i}-a_{3i})(T-t)$, $\widehat{g}_{5}^{\gamma_{i}}(t)$ and $\widehat{g}_{6}^{\gamma_{i}}(t)$ are as given in Theorem \ref{theorem4.4}.
\end{corollary}

\subsection{Special case with one possible risk aversion}
When the distribution of risk aversion is reduced to a single realization $\gamma$, one has
$\frac{\partial \widehat{g}_{6}^{\gamma}(t)}{\partial t}-2r \delta \widehat{g}_{6}^{\gamma}(t)-\frac{0.5(\mu-r)^{2}}{\sigma^{2}}=0$. Moreover, by using $\widehat{g}_{5}^{\gamma}(T)=\widehat{g}_{6}^{\gamma}(T)=0$, we can obtain that $\widehat{g}_{6}^{\gamma}(t)=\frac{(\mu-r)^{2}}{4 r \delta \sigma^{2}}\left(e^{-2r \delta(T-t)}-1\right)$ and $\widehat{g}_{5}^{\gamma}(t)=\frac{(2\delta+1)(\mu-r)^{2}}{4r}\left[\frac{1-e^{-2r \delta(T-t)}}{2r \delta}-T+t\right]$. Therefore, the candidate equilibrium reinsurance and investment strategy satisfies
\begin{equation}\label{eq81}
\hat{q}(t)=\frac{a \eta_{2}}{b^{2}\gamma}e^{-(A+\beta)(T-t)},\quad
   \hat{\pi}(t)=\left[\frac{\mu-r}{\sigma^{2}}-\frac{(\mu-r)^{2}}{2r \sigma^{2}}\left(e^{-2r \delta(T-t)}-1\right)\right]
   \frac{e^{-(A+\beta)(T-t)}}{xs^{2\delta}_{1}\gamma}
\end{equation}
and the candidate equilibrium value function is
\begin{align}\label{eq82}
U(t,x,s_{1},m_{1})=&(x+\beta m_{1})e^{(A+\beta)(T-t)}-\frac{a \eta}{A+\beta}\left[1-e^{(A+\beta)(T-t)}\right]
\nonumber \\
&+\frac{0.5a^{2} \eta^{2}_{2}}{b^{2}\gamma}(T-t)-\frac{1}{\gamma}\left[\widehat{g}_{5}^{\gamma}(t)+\widehat{g}_{6}^{\gamma}(t) s^{-2\delta}_{1}\right].
\end{align}

Given that $\widehat{g}_{5}^{\gamma}(t)$ and $\widehat{g}_{6}^{\gamma}(t)$ exhibit unique global solutions, the methodology applied to $n$ potential risk aversions can be fully followed to verify the admissibility conditions and assumptions outlined in Theorem \ref{theorem3.1}, but with greater simplicity. As a result, the subsequent findings are achievable.

\begin{theorem}\label{theorem4.5}
Assume that the utility function $\varphi^{\gamma}$ satisfies the equation \eqref{eq49} and the assumption $(A2)$ holds. Moreover, assume that there is one possible risk aversion $\gamma\in[\epsilon_{1},\infty)$ and the parameters satisfy $-32\gamma(\mu-r)\delta^{2}\widetilde{\pi}(s)
  +128\gamma^{2}\delta^{2}\sigma^{2}\widetilde{\pi}^{2}(s)\leq\frac{\kappa^{2}}{2\sigma^{2}}$ for $s\in[0,T]$, where $\widetilde{\pi}(s)=\frac{\mu-r}{\gamma\sigma^{2}}-\frac{(\mu-r)^{2}}{2\gamma r \sigma^{2}}\left(e^{-2r \delta(T-s)}-1\right)$. Then the equilibrium reinsurance and investment strategy and the associated equilibrium value function are characterized by equations \eqref{eq81} and \eqref{eq82}, respectively.
\end{theorem}

\begin{corollary}
If $\alpha=h=B=C=\beta=0$, which indicates that delay is not incorporated into the model, then the equilibrium reinsurance and investment strategy becomes
\begin{equation*}
\hat{q}(t)=\frac{a \eta_{2}}{b^{2}\gamma}e^{-r(T-t)},\quad
   \hat{\pi}(t)=\left[\frac{\mu-r}{\sigma^{2}}-\frac{(\mu-r)^{2}}{2r \sigma^{2}}\left(e^{-2r \delta(T-t)}-1\right)\right]
   \frac{e^{-r(T-t)}}{xs^{2\delta}_{1}\gamma}
\end{equation*}
and the equilibrium value function is $V(t,x,s_{1})=xe^{r(T-t)}-\frac{a \eta}{r}\left[1-e^{r(T-t)}\right]
+\frac{0.5a^{2} \eta^{2}_{2}}{b^{2}\gamma}(T-t)-\frac{1}{\gamma}\left[\widehat{g}_{5}^{\gamma}(t)+\widehat{g}_{6}^{\gamma}(t) s^{-2\delta}_{1}\right]$.
\end{corollary}

\section{Equilibrium strategies under the geometric Brownian motion}
In this section, we first attempt to address the delayed reinsurance and investment problem with random risk aversion under the geometric Brownian motion in a general utility function framework as far as possible. Subsequently, in order to obtain (semi-)analytical equilibrium reinsurance and investment strategy, we narrow our focus to analyze the exponential utility function and the power utility function, respectively.

If $\delta=0$, which implies that the CEV model degenerates into geometric Brownian motion, then $U(t,x,s_{1},m_{1})$, $H(t,x,s_{1},m_{1})$ and $y^{\hat{u},\gamma}(t,x,s_{1},m_{1})$ are independent of $s_{1}$. In the sequel, we will adopt the notations $U(t,x,m_{1})$, $H(t,x,m_{1})$ and $y^{\hat{u},\gamma}(t,x,m_{1})$ for simplicity. Then, one can derive the derivatives $H_{t}(t,x,m_{1})$, $H_{x}(t,x,m_{1})$, $H_{m_{1}}(t,x,m_{1})$ and $H_{xx}(t,x,m_{1})$ that are similar to those shown in the equation \eqref{eq35}. Thus, the equation \eqref{eq20} can be rewritten as
\begin{equation}\label{eq36}
  \sup\limits_{u} \;\left\{\mathcal{\widetilde{A}}^{u}U(t,x,m_{1})-0.5(b^{2}q^{2}+\pi^{2}x^{2}\sigma^{2})\int(\iota^{\gamma})^{'}(Y^{\gamma}(t,x,m_{1}))
  (Y_{x}^{\gamma}(t,x,m_{1}))^{2}\mathrm{d}\Gamma(\gamma)\right\}=0,
\end{equation}
where
\begin{align*}
  \mathcal{\widetilde{A}}^{u}\widetilde{\phi}(t,x,m_{1})=&\widetilde{\phi}_{t}(t,x,m_{1})+\left[(A+\pi(\mu-r))x+Bm_{1}+Cm_{2}+a\eta+a\eta_{2}q\right]\widetilde{\phi}_{x}(t,x,m_{1})\\
  &+0.5\left(b^{2}q^{2}+\pi^{2}x^{2}\sigma^{2}\right)\widetilde{\phi}_{xx}(t,x,m_{1})+\left(x-\alpha m_{1}-e^{-\alpha h}m_{2}\right)\widetilde{\phi}_{m_{1}}(t,x,m_{1})
\end{align*}
for any $\widetilde{\phi}(t,x,m_{1})\in C^{1,2,1}(\mathcal{\widetilde{D}})$ and $u=(q,\pi)$ with $\mathcal{\widetilde{D}}=[0,T]\times\mathbb{R}\times\mathbb{R}$.

Having derived the equation \eqref{eq36} that characterizes the equilibrium reinsurance and investment strategy, we can perform partial differentiation on the terms within the brackets of equation \eqref{eq36} with respect to $q$ and $\pi$ separately. This step yields the candidate equilibrium reinsurance and investment strategy
\begin{equation}\label{eq37}
  \hat{q}(t)=-\frac{a \eta_{2}U_{x}(t,x,m_{1})}{b^{2}\left[U_{xx}(t,x,m_{1})-\int(\iota^{\gamma})^{'}(Y^{\gamma}(t,x,m_{1}))(Y_{x}^{\gamma}(t,x,m_{1}))^{2}\mathrm{d}\Gamma(\gamma)\right]}
\end{equation}
and
\begin{equation}\label{eq38}
\hat{\pi}(t)=-\frac{(\mu- r)U_{x}(t,x,m_{1})}{\sigma^{2}x\left[U_{xx}(t,x,m_{1})-\int(\iota^{\gamma})^{'}(Y^{\gamma}(t,x,m_{1}))(Y_{x}^{\gamma}(t,x,m_{1}))^{2}\mathrm{d}\Gamma(\gamma)\right]}.
\end{equation}
It follows from the equation \eqref{eq28} that
$
U(t,x,m_{1})=H(t,x,m_{1})
=\int(\varphi^{\gamma})^{-1}\left(Y^{\gamma}(t,x,m_{1})\right)\mathrm{d}\Gamma(\gamma).
$
Then, through the system of equations \eqref{eq35}, equations \eqref{eq37} and \eqref{eq38} admit the following reformulation
\begin{equation}\label{eq39}
  \hat{q}(t)=-\frac{a \eta_{2}\int\iota^{\gamma}(Y^{\gamma}(t,x,m_{1}))Y_{x}^{\gamma}(t,x,m_{1})\mathrm{d}\Gamma(\gamma)}{b^{2}\int\iota^{\gamma}(Y^{\gamma}(t,x,m_{1}))Y_{xx}^{\gamma}(t,x,m_{1})\mathrm{d}\Gamma(\gamma)}
\end{equation}
and
\begin{equation}\label{eq40}
   \hat{\pi}(t)=-\frac{(\mu- r)\int\iota^{\gamma}(Y^{\gamma}(t,x,m_{1}))Y_{x}^{\gamma}(t,x,m_{1})\mathrm{d}\Gamma(\gamma)}{\sigma^{2}x\int\iota^{\gamma}(Y^{\gamma}(t,x,m_{1}))Y_{xx}^{\gamma}(t,x,m_{1})\mathrm{d}\Gamma(\gamma)}.
\end{equation}
Henceforth, our analysis focuses on the characterization of $Y^{\gamma}(t,x,m_{1})$ for all $\gamma$.
\subsection{The exponential utility}
If $\varphi^{\gamma}$ follows the exponential utility given by the equation \eqref{eq49}, then equations \eqref{eq39} and \eqref{eq40} reduce to
$$
   \hat{q}(t)=-\frac{a \eta_{2} \int \frac{Y_{x}^{\gamma}(t,x,m_{1})}{\gamma Y^{\gamma}(t,x,m_{1})}\mathrm{d}\Gamma(\gamma)}
{b^{2}\int \frac{Y_{xx}^{\gamma}(t,x,m_{1})}{\gamma Y^{\gamma}(t,x,m_{1})}\mathrm{d}\Gamma(\gamma)},\quad
   \hat{\pi}(t)=-\frac{(\mu- r) \int \frac{Y_{x}^{\gamma}(t,x,m_{1})}{\gamma Y^{\gamma}(t,x,m_{1})}\mathrm{d}\Gamma(\gamma)}
{\sigma^{2}x\int \frac{Y_{xx}^{\gamma}(t,x,m_{1})}{\gamma Y^{\gamma}(t,x,m_{1})}\mathrm{d}\Gamma(\gamma)}.
$$
We postulate an ansatz where $Y^{\gamma}(t,x,m_{1})=-\frac{1}{\gamma}e^{g_{1}^{\gamma}(t)(x+\beta m_{1})+g_{2}^{\gamma}(t)}$ with $g_{1}^{\gamma}(t)$ and $g_{2}^{\gamma}(t)$ being deterministic functions of $t\in[0,T]$ with $g_{1}^{\gamma}(T)=-\gamma$ and $g_{2}^{\gamma}(T)=0$.

Next, the partial derivatives associated with the ansatz are derived as
\begin{eqnarray*}
\left\{
\begin{aligned}
   &Y^{\gamma}_{t}(t,x,m_{1})=Y^{\gamma}(t,x,m_{1})\left[\frac{\partial g_{1}^{\gamma}(t)}{\partial t}\left(x+\beta m_{1}\right)+\frac{\partial g_{2}^{\gamma}(t)}{\partial t}\right],\quad
   Y^{\gamma}_{x}(t,x,m_{1})=g_{1}^{\gamma}(t)Y^{\gamma}(t,x,m_{1}),\nonumber\\
   &Y^{\gamma}_{xx}(t,x,m_{1})=(g_{1}^{\gamma}(t))^{2}Y^{\gamma}(t,x,m_{1}),\quad
   Y^{\gamma}_{m_{1}}(t,x,m_{1})=\beta g_{1}^{\gamma}(t)Y^{\gamma}(t,x,m_{1})
\end{aligned}
\right.
\end{eqnarray*}
and the candidate equilibrium reinsurance and investment strategy can be given by
\begin{equation}\label{eq41}
  \hat{q}(t)=-\frac{a \eta_{2} \int \frac{g_{1}^{\gamma}(t)}{\gamma }\mathrm{d}\Gamma(\gamma)}
{b^{2}\int \frac{(g_{1}^{\gamma}(t))^{2}}{\gamma }\mathrm{d}\Gamma(\gamma)}
\end{equation}
and
\begin{equation}\label{eq42}
 \hat{\pi}(t)=-\frac{(\mu- r) \int \frac{g_{1}^{\gamma}(t)}{\gamma }\mathrm{d}\Gamma(\gamma)}
{\sigma^{2}x\int \frac{(g_{1}^{\gamma}(t))^{2}}{\gamma }\mathrm{d}\Gamma(\gamma)}.
\end{equation}
Inserting the derivatives of $Y^{\gamma}(t,x,m_{1})$ into the equation \eqref{eq21}, where $\mathcal{A}^{u}$ is replaced by $\mathcal{\widetilde{A}}^{u}$, we have
\begin{align*}
  0=&\frac{\partial g_{1}^{\gamma}(t)}{\partial t}(x+\beta m_{1})+\frac{\partial g_{2}^{\gamma}(t)}{\partial t}
  +\left[(A+\hat{\pi}(\mu-r))x+Bm_{1}+Cm_{2}+a\eta+a\eta_{2}\hat{q}\right]g_{1}^{\gamma}(t)\nonumber\\
&+0.5(b^{2}\hat{q}^{2}+\hat{\pi}^{2}x^{2}\sigma^{2})(g_{1}^{\gamma}(t))^{2}
+\left(x-\alpha m_{1}-e^{-\alpha h}m_{2}\right)\beta g_{1}^{\gamma}(t),
\end{align*}
where $\hat{q}$ and $\hat{\pi}$ are provided by equations \eqref{eq41} and \eqref{eq42}. Then, under the assumption $(A2)$, the above equation arrives at
\begin{align}\label{eq44}
  0=&\frac{\partial g_{1}^{\gamma}(t)}{\partial t}(x+\beta m_{1})+\frac{\partial g_{2}^{\gamma}(t)}{\partial t}
  +\left(A+\beta\right)\left(x+\beta m_{1}\right) g_{1}^{\gamma}(t)\nonumber\\
  &+\left[\hat{\pi}(\mu-r)x+a\eta+a\eta_{2}\hat{q}\right]g_{1}^{\gamma}(t)+0.5(b^{2}\hat{q}^{2}+\hat{\pi}^{2}x^{2}\sigma^{2})(g_{1}^{\gamma}(t))^{2}.
\end{align}

One can see that $\hat{q}$ and $\hat{\pi}x$ are independent of $x$ and $m_{1}$ due to equations \eqref{eq41} and \eqref{eq42}. Thus, applying the coefficient comparison method to the terms containing $x+\beta m_{1}$ in the equation \eqref{eq44} generates
\begin{equation}\label{eq45}
 \frac{\partial g_{1}^{\gamma}(t)}{\partial t}+\left(A+\beta\right) g_{1}^{\gamma}(t)=0
\end{equation}
and
\begin{equation}\label{eq46}
 \frac{\partial g_{2}^{\gamma}(t)}{\partial t}+\left[\hat{\pi}(\mu-r)x+a\eta+a\eta_{2}\hat{q}\right]g_{1}^{\gamma}(t)+0.5(b^{2}\hat{q}^{2}+\hat{\pi}^{2}x^{2}\sigma^{2})(g_{1}^{\gamma}(t))^{2}=0.
\end{equation}
With the boundary condition $g_{1}^{\gamma}(T)=-\gamma$, solving the equation \eqref{eq45} gives $g_{1}^{\gamma}(t)=-\gamma e^{(A+\beta)(T-t)}$. Since $g_{2}^{\gamma}(T)=0$, substituting the expression of $g_{1}^{\gamma}(t)$ into equations \eqref{eq41}, \eqref{eq42} and \eqref{eq46} yields
\begin{equation}\label{eq47}
\hat{q}(t)=\frac{a \eta_{2}}{b^{2}\mathbb{E}[\gamma]}e^{-(A+\beta)(T-t)},\quad
   \hat{\pi}(t)=\frac{\mu-r}{\sigma^{2}x\mathbb{E}[\gamma]}e^{-(A+\beta)(T-t)},
\end{equation}
and
\begin{equation}\label{eq48}
g_{2}^{\gamma}(t)=\frac{\gamma a \eta}{A+\beta}\left[1-e^{(A+\beta)(T-t)}\right]-D(T-t),
\end{equation}
where
$
D=\gamma\left[\frac{(\mu-r)^{2}}{\sigma^{2}\mathbb{E}[\gamma]}+\frac{a^{2} \eta^{2}_{2}}{b^{2}\mathbb{E}[\gamma]}\right]
-0.5\gamma^{2}\left[\frac{(\mu-r)^{2}}{\sigma^{2}\mathbb{E}^{2}[\gamma]}+\frac{a^{2} \eta^{2}_{2}}{b^{2}\mathbb{E}^{2}[\gamma]}\right].
$
By using the equation \eqref{eq28}, one can further derive that the candidate equilibrium value function is
\begin{equation}\label{eq50}
U(t,x,m_{1})=(x+\beta m_{1})e^{(A+\beta)(T-t)}-\frac{a \eta}{A+\beta}\left[1-e^{(A+\beta)(T-t)}\right]
+0.5\left[\frac{(\mu-r)^{2}}{\sigma^{2}\mathbb{E}[\gamma]}+\frac{a^{2} \eta^{2}_{2}}{b^{2}\mathbb{E}[\gamma]}\right](T-t).
\end{equation}

In addition, provided the expression for the candidate equilibrium reinsurance and investment strategy, one can present the following sensitivity analysis on the effects of the delay parameters on $(\hat{q}(t), \hat{\pi}(t)x)$.

\begin{proposition}\label{proposition4.1}
If the candidate equilibrium reinsurance and investment strategy is captured by the equation \eqref{eq47}, then one has $\frac{\partial \hat{q}(t)}{\partial h}<0$, $\frac{\partial [\hat{\pi}(t)x]}{\partial h}<0$ and
\begin{align}\label{eq70}
sgn\left(\frac{\partial \hat{f}(t)}{\partial \alpha}\right)=
\begin{cases}
1, \quad\quad\text{for}\; \quad\, \alpha>-\frac{1}{h}\ln\frac{1}{h},\\
0, \quad\quad\text{for}\;\quad\,  \alpha=-\frac{1}{h}\ln\frac{1}{h},\\
-1,  \quad\;\text{for}\;\quad\,  \alpha<-\frac{1}{h}\ln\frac{1}{h},
\end{cases}
\end{align}
where $\hat{f}(t)\in\left\{\hat{q}(t),\hat{\pi}(t)x\right\}$. Furthermore, if $r+\alpha<1$, then one can obtain that
\begin{align}\label{eq71}
sgn\left(\frac{\partial \hat{f}(t)}{\partial \beta}\right)=
\begin{cases}
1, \quad\quad\text{for}\; \quad\, h<-\frac{1}{\alpha}\ln(1-r-\alpha),\\
0, \quad\quad\text{for}\;\quad\,  h=-\frac{1}{\alpha}\ln(1-r-\alpha),\\
-1,  \quad\;\text{for}\;\quad\,  h>-\frac{1}{\alpha}\ln(1-r-\alpha).
\end{cases}
\end{align}
\end{proposition}
\begin{proof}
See Appendix E.
\end{proof}
Based on the above investigations, we are now in a position to derive the following result addressing the equilibrium reinsurance and investment strategy when the utility function $\varphi^{\gamma}$ follows the exponential utility.
\begin{theorem}\label{theorem4.1}
Assume that the utility function $\varphi^{\gamma}$ is given by the equation \eqref{eq49} with $\gamma\in[\epsilon_{1},\infty)$ and the assumption $(A2)$ holds. Then the equilibrium reinsurance and investment strategy and the corresponding equilibrium value function are given by equations \eqref{eq47} and \eqref{eq50}, respectively.
\end{theorem}
\begin{proof}
See Appendix F.
\end{proof}

\begin{corollary}
If $\alpha=h=B=C=\beta=0$, which implies that the model does not include delay, then the equilibrium reinsurance and investment strategy reduces to
\begin{equation}\label{eq64}
 \hat{q}(t)=\frac{a \eta_{2}}{b^{2}\mathbb{E}[\gamma]}e^{-r(T-t)},\quad
   \hat{\pi}(t)=\frac{\mu-r}{\sigma^{2}x\mathbb{E}[\gamma]}e^{-r(T-t)}
\end{equation}
and the corresponding equilibrium value function becomes
$$
V(t,x)=xe^{r(T-t)}-\frac{a \eta}{r}\left[1-e^{r(T-t)}\right]
+0.5\left[\frac{(\mu-r)^{2}}{\sigma^{2}\mathbb{E}[\gamma]}+\frac{a^{2} \eta^{2}_{2}}{b^{2}\mathbb{E}[\gamma]}\right](T-t).
$$
\end{corollary}

\begin{remark}
Note that the equilibrium reinsurance strategy in \eqref{eq64} is the same as the one in the equation (4.11) of \cite{Kang2026} and the equilibrium investment strategy in \eqref{eq64} is equivalent to the one on page 966 of \cite{Desmettre2023}.
\end{remark}

\begin{corollary}
If $\gamma$ is assumed to be a constant, then the equilibrium reinsurance and investment strategy can be given by
\begin{equation}\label{eq65}
\hat{q}(t)=\frac{a \eta_{2}}{b^{2}\gamma}e^{-(A+\beta)(T-t)},\quad
   \hat{\pi}(t)=\frac{\mu-r}{\sigma^{2}x\gamma}e^{-(A+\beta)(T-t)}
\end{equation}
and the corresponding equilibrium value function satisfies
$$
V(t,x,m_{1})=(x+\beta m_{1})e^{(A+\beta)(T-t)}-\frac{a \eta}{A+\beta}\left[1-e^{(A+\beta)(T-t)}\right]
+0.5\left[\frac{(\mu-r)^{2}}{\sigma^{2}\gamma}+\frac{a^{2} \eta^{2}_{2}}{b^{2}\gamma}\right](T-t).
$$
\end{corollary}

\begin{remark}
Note that the equilibrium reinsurance strategy in \eqref{eq65} is consistent with the one in the equation (2.9) of \cite{A2020} and the equilibrium investment amount $\hat{\pi}(t)x$ in \eqref{eq65} coincides with the ones with $\beta=0$ in equations (3.11) and (3.14) of \cite{A2018}, which means that under singular risk aversion in the Black-Scholes financial market, the equilibrium strategy solving the problem \eqref{eq13} is mathematically equivalent to the optimal strategy solving the problem \eqref{eq12}.
\end{remark}

\subsection{The power utility}
Assuming that the function $\varphi^{\gamma}$ adopts a power utility form
\begin{equation}\label{eq51}
\varphi^{\gamma}(\cdot)=\frac{1}{1-\gamma}(\cdot)^{1-\gamma},
\end{equation}
where $\gamma>0$, we can derive its inverse function and the derivative of the inverse function as shown below
$$
(\varphi^{\gamma})^{-1}(\cdot)=(1-\gamma)^{\frac{1}{1-\gamma}}(\cdot)^{\frac{1}{1-\gamma}},\quad
\frac{\mathrm{d}}{\mathrm{d} y}(\varphi^{\gamma})^{-1}(y)=(1-\gamma)^{\frac{\gamma}{1-\gamma}}(y)^{\frac{\gamma}{1-\gamma}}.
$$
Equations \eqref{eq39} and \eqref{eq40} consequently yield the reduced formulation
\begin{equation*}
   \left\{ \begin{aligned}
   &\hat{q}(t)=-\frac{a \eta_{2} \int (1-\gamma)^{\frac{\gamma}{1-\gamma}}\left(Y^{\gamma}(t,x,m_{1})\right)^{\frac{\gamma}{1-\gamma}} Y_{x}^{\gamma}(t,x,m_{1})\mathrm{d}\Gamma(\gamma)}
{b^{2}\int (1-\gamma)^{\frac{\gamma}{1-\gamma}}\left(Y^{\gamma}(t,x,m_{1})\right)^{\frac{\gamma}{1-\gamma}} Y_{xx}^{\gamma}(t,x,m_{1})\mathrm{d}\Gamma(\gamma)},\\
   &\hat{\pi}(t)=-\frac{(\mu- r) \int (1-\gamma)^{\frac{\gamma}{1-\gamma}}\left(Y^{\gamma}(t,x,m_{1})\right)^{\frac{\gamma}{1-\gamma}} Y_{x}^{\gamma}(t,x,m_{1})\mathrm{d}\Gamma(\gamma)}
{\sigma^{2}x\int (1-\gamma)^{\frac{\gamma}{1-\gamma}}\left(Y^{\gamma}(t,x,m_{1})\right)^{\frac{\gamma}{1-\gamma}} Y_{xx}^{\gamma}(t,x,m_{1})\mathrm{d}\Gamma(\gamma)}.
  \end{aligned}\right.
\end{equation*}
We propose an ansatz in which $Y^{\gamma}(t,x,m_{1})=\frac{1}{1-\gamma}(g^{\gamma}(t))^{\gamma}(x+\beta m_{1})^{1-\gamma}$ with $g^{\gamma}(t)$ being a deterministic function of $t\in[0,T]$ with $g^{\gamma}(T)=1$.

Now we carry out the derivation of partial derivatives linked to the above ansatz, arriving at
\begin{eqnarray*}
\left\{
\begin{aligned}
   &Y^{\gamma}_{t}(t,x,m_{1})=\frac{\gamma}{1-\gamma}(g^{\gamma}(t))^{\gamma-1}(x+\beta m_{1})^{1-\gamma}\frac{\partial g^{\gamma}(t)}{\partial t},\quad
   Y^{\gamma}_{x}(t,x,m_{1})=(g^{\gamma}(t))^{\gamma}(x+\beta m_{1})^{-\gamma},\nonumber\\
   &Y^{\gamma}_{xx}(t,x,m_{1})=-\gamma(g^{\gamma}(t))^{\gamma}(x+\beta m_{1})^{-\gamma-1},\quad
   Y^{\gamma}_{m_{1}}(t,x,m_{1})=\beta(g^{\gamma}(t))^{\gamma}(x+\beta m_{1})^{-\gamma}.
\end{aligned}
\right.
\end{eqnarray*}
Substituting the above derivatives into the candidate equilibrium reinsurance and investment strategy yields
\begin{equation}\label{eq52}
  \hat{q}(t)=\frac{a \eta_{2} (x+\beta m_{1})}
{b^{2}\varpi(t,\Gamma)}
\end{equation}
and
\begin{equation}\label{eq53}
 \hat{\pi}(t)=\frac{(\mu- r) (x+\beta m_{1})}
{\sigma^{2}x\varpi(t,\Gamma)},
\end{equation}
where
\begin{equation}\label{eq54}
\varpi(t,\Gamma)=\frac{\int (g^{\gamma}(t))^{\frac{\gamma}{1-\gamma}}\gamma\mathrm{d}\Gamma(\gamma)}{\int (g^{\gamma}(t))^{\frac{\gamma}{1-\gamma}}\mathrm{d}\Gamma(\gamma)}.
\end{equation}
In addition, replacing $\mathcal{A}^{u}$ by $\mathcal{\widetilde{A}}^{u}$ in the equation \eqref{eq21} and inserting the above partial derivatives of $Y^{\gamma}(t,x,m_{1})$ into the equation \eqref{eq21} allow us to derive
\begin{align*}
  0=&\frac{\gamma}{1-\gamma}(g^{\gamma}(t))^{\gamma-1}(x+\beta m_{1})^{1-\gamma}\frac{\partial g^{\gamma}(t)}{\partial t}\nonumber\\
  &+\left[(A+\hat{\pi}(\mu-r))x+Bm_{1}+Cm_{2}+a\eta+a\eta_{2}\hat{q}\right](g^{\gamma}(t))^{\gamma}(x+\beta m_{1})^{-\gamma}\nonumber\\
&-0.5\gamma(b^{2}\hat{q}^{2}+\hat{\pi}^{2}x^{2}\sigma^{2})(g^{\gamma}(t))^{\gamma}(x+\beta m_{1})^{-\gamma-1}
+\beta\left(x-\alpha m_{1}-e^{-\alpha h}m_{2}\right)(g^{\gamma}(t))^{\gamma}(x+\beta m_{1})^{-\gamma}.
\end{align*}
If we assume that $\eta=\eta_{1}-\eta_{2}=0$ and the condition $(A2)$ holds, then we can obtain that
\begin{align}\label{eq55}
  \frac{\partial g^{\gamma}(t)}{\partial t}=\frac{1-\gamma}{\gamma}\left[-A-\beta-\frac{(\mu- r)^{2}}
{\sigma^{2}\varpi(t,\Gamma)}-\frac{a^{2} \eta^{2}_{2}}{b^{2}\varpi(t,\Gamma)}+\frac{0.5\gamma a^{2} \eta^{2}_{2}}
{b^{2}(\varpi(t,\Gamma))^{2}}+\frac{0.5\gamma(\mu- r)^{2}}
{\sigma^{2}(\varpi(t,\Gamma))^{2}}\right] g^{\gamma}(t).
\end{align}
By differentiating $\varpi(t,\Gamma)$ with respect to $t$, we further have
\begin{align*}
  \frac{\partial \varpi(t,\Gamma)}{\partial t}=&\frac{\left(\int (g^{\gamma}(t))^{\frac{\gamma}{1-\gamma}}\mathrm{d}\Gamma(\gamma)\right)\left(\int \gamma\frac{\gamma}{1-\gamma} (g^{\gamma}(t))^{\frac{\gamma}{1-\gamma}-1}\frac{\partial g^{\gamma}(t)}{\partial t}\mathrm{d}\Gamma(\gamma)\right)}{\left(\int (g^{\gamma}(t))^{\frac{\gamma}{1-\gamma}}\mathrm{d}\Gamma(\gamma)\right)^{2}}\nonumber\\
  &-\frac{\left(\int \frac{\gamma}{1-\gamma}(g^{\gamma}(t))^{\frac{\gamma}{1-\gamma}-1}\frac{\partial g^{\gamma}(t)}{\partial t}\mathrm{d}\Gamma(\gamma)\right)\left(\int (g^{\gamma}(t))^{\frac{\gamma}{1-\gamma}}\gamma\mathrm{d}\Gamma(\gamma)\right)}{\left(\int (g^{\gamma}(t))^{\frac{\gamma}{1-\gamma}}\mathrm{d}\Gamma(\gamma)\right)^{2}}\nonumber\\
  =&\int \gamma (g^{\gamma}(t))^{\frac{\gamma}{1-\gamma}}\left[-A-\beta-\frac{(\mu- r)^{2}}
{\sigma^{2}\varpi(t,\Gamma)}-\frac{a^{2} \eta^{2}_{2}}{b^{2}\varpi(t,\Gamma)}+\frac{0.5\gamma a^{2} \eta^{2}_{2}}
{b^{2}(\varpi(t,\Gamma))^{2}}+\frac{0.5\gamma(\mu- r)^{2}}
{\sigma^{2}(\varpi(t,\Gamma))^{2}}\right]\mathrm{d}\Gamma(\gamma)\nonumber\\
&\times\frac{\int (g^{\gamma}(t))^{\frac{\gamma}{1-\gamma}}\mathrm{d}\Gamma(\gamma)}{\left(\int (g^{\gamma}(t))^{\frac{\gamma}{1-\gamma}}\mathrm{d}\Gamma(\gamma)\right)^{2}}-\frac{\int (g^{\gamma}(t))^{\frac{\gamma}{1-\gamma}}\gamma\mathrm{d}\Gamma(\gamma)}{\left(\int (g^{\gamma}(t))^{\frac{\gamma}{1-\gamma}}\mathrm{d}\Gamma(\gamma)\right)^{2}}\nonumber\\
&\times\int (g^{\gamma}(t))^{\frac{\gamma}{1-\gamma}}\left[-A-\beta-\frac{(\mu- r)^{2}}
{\sigma^{2}\varpi(t,\Gamma)}-\frac{a^{2} \eta^{2}_{2}}{b^{2}\varpi(t,\Gamma)}+\frac{0.5\gamma a^{2} \eta^{2}_{2}}
{b^{2}(\varpi(t,\Gamma))^{2}}+\frac{0.5\gamma(\mu- r)^{2}}
{\sigma^{2}(\varpi(t,\Gamma))^{2}}\right]\mathrm{d}\Gamma(\gamma),
\end{align*}
where we apply the equation \eqref{eq55} in the second equality. By further simplification, we can get
\begin{align}\label{eq56}
  \frac{\partial \varpi(t,\Gamma)}{\partial t}=&\frac{ 0.5}
{(\varpi(t,\Gamma))^{2}}\left(\frac{ a^{2} \eta^{2}_{2}}
{b^{2}}+\frac{(\mu- r)^{2}}
{\sigma^{2}}\right)\times\frac{\left(\int (g^{\gamma}(t))^{\frac{\gamma}{1-\gamma}}\mathrm{d}\Gamma(\gamma)\right)\left(\int \gamma^{2}(g^{\gamma}(t))^{\frac{\gamma}{1-\gamma}}\mathrm{d}\Gamma(\gamma)\right)}{\left(\int (g^{\gamma}(t))^{\frac{\gamma}{1-\gamma}}\mathrm{d}\Gamma(\gamma)\right)^{2}}\nonumber\\
&-\frac{ 0.5}
{(\varpi(t,\Gamma))^{2}}\left(\frac{ a^{2} \eta^{2}_{2}}
{b^{2}}+\frac{(\mu- r)^{2}}
{\sigma^{2}}\right)\times\left(\frac{\int (g^{\gamma}(t))^{\frac{\gamma}{1-\gamma}}\gamma\mathrm{d}\Gamma(\gamma)}{\int (g^{\gamma}(t))^{\frac{\gamma}{1-\gamma}}\mathrm{d}\Gamma(\gamma)}\right)^{2}\nonumber\\
=&\frac{ 0.5}
{(\varpi(t,\Gamma))^{2}}\left(\frac{ a^{2} \eta^{2}_{2}}
{b^{2}}+\frac{(\mu- r)^{2}}
{\sigma^{2}}\right)\times\left[\frac{\int \gamma^{2}(g^{\gamma}(t))^{\frac{\gamma}{1-\gamma}}\mathrm{d}\Gamma(\gamma)}{\int (g^{\gamma}(t))^{\frac{\gamma}{1-\gamma}}\mathrm{d}\Gamma(\gamma)}-(\varpi(t,\Gamma))^{2}\right].
\end{align}
Through the expression of the equation \eqref{eq54}, we find that equations \eqref{eq55} and \eqref{eq56} are infinite-dimensional. Moreover, by the terminal condition $g^{\gamma}(T)=1$, it follows from equations \eqref{eq54} and \eqref{eq55} that
\begin{align*}
  g^{\gamma}(t)=\exp\left\{-\int^{T}_{t}\frac{1-\gamma}{\gamma}\left[-A-\beta-\frac{(\mu- r)^{2}}
{\sigma^{2}\varpi(s,\Gamma)}-\frac{a^{2} \eta^{2}_{2}}{b^{2}\varpi(s,\Gamma)}+\frac{0.5\gamma a^{2} \eta^{2}_{2}}
{b^{2}(\varpi(s,\Gamma))^{2}}+\frac{0.5\gamma(\mu- r)^{2}}
{\sigma^{2}(\varpi(s,\Gamma))^{2}}\right]\mathrm{d}s \right\}
\end{align*}
and
\begin{equation}\label{eq57}
\varpi(t,\Gamma)=\frac{\int \exp\left\{-\int^{T}_{t}\left[-A-\beta-\frac{(\mu- r)^{2}}
{\sigma^{2}\varpi(s,\Gamma)}-\frac{a^{2} \eta^{2}_{2}}{b^{2}\varpi(s,\Gamma)}+\frac{0.5\gamma a^{2} \eta^{2}_{2}}
{b^{2}(\varpi(s,\Gamma))^{2}}+\frac{0.5\gamma(\mu- r)^{2}}
{\sigma^{2}(\varpi(s,\Gamma))^{2}}\right]\mathrm{d}s \right\}\gamma\mathrm{d}\Gamma(\gamma)}{\int \exp\left\{-\int^{T}_{t}\left[-A-\beta-\frac{(\mu- r)^{2}}
{\sigma^{2}\varpi(s,\Gamma)}-\frac{a^{2} \eta^{2}_{2}}{b^{2}\varpi(s,\Gamma)}+\frac{0.5\gamma a^{2} \eta^{2}_{2}}
{b^{2}(\varpi(s,\Gamma))^{2}}+\frac{0.5\gamma(\mu- r)^{2}}
{\sigma^{2}(\varpi(s,\Gamma))^{2}}\right]\mathrm{d}s \right\}\mathrm{d}\Gamma(\gamma)}.
\end{equation}
Similarly, the equation \eqref{eq56} can be rewritten as
\begin{align}\label{eq58}
  \frac{\partial \varpi(t,\Gamma)}{\partial t}
=&\frac{ 0.5}
{(\varpi(t,\Gamma))^{2}}\left(\frac{ a^{2} \eta^{2}_{2}}
{b^{2}}+\frac{(\mu- r)^{2}}
{\sigma^{2}}\right)\times\bigg[-(\varpi(t,\Gamma))^{2}\nonumber\\
&+\left.\frac{\int \gamma^{2}\exp\left\{-\int^{T}_{t}\left[-A-\beta-\frac{(\mu- r)^{2}}
{\sigma^{2}\varpi(s,\Gamma)}-\frac{a^{2} \eta^{2}_{2}}{b^{2}\varpi(s,\Gamma)}+\frac{0.5\gamma a^{2} \eta^{2}_{2}}
{b^{2}(\varpi(s,\Gamma))^{2}}+\frac{0.5\gamma(\mu- r)^{2}}
{\sigma^{2}(\varpi(s,\Gamma))^{2}}\right]\mathrm{d}s \right\}\mathrm{d}\Gamma(\gamma)}{\int \exp\left\{-\int^{T}_{t}\left[-A-\beta-\frac{(\mu- r)^{2}}
{\sigma^{2}\varpi(s,\Gamma)}-\frac{a^{2} \eta^{2}_{2}}{b^{2}\varpi(s,\Gamma)}+\frac{0.5\gamma a^{2} \eta^{2}_{2}}
{b^{2}(\varpi(s,\Gamma))^{2}}+\frac{0.5\gamma(\mu- r)^{2}}
{\sigma^{2}(\varpi(s,\Gamma))^{2}}\right]\mathrm{d}s \right\}\mathrm{d}\Gamma(\gamma)}\right].
\end{align}

It is worth mentioning that equations \eqref{eq57} and \eqref{eq58} demonstrate that because of the linear architecture of the ODE for $g^{\gamma}(t)$, which is presented in the equation \eqref{eq55}, we are able to transform the infinite-dimensional system of ODEs into a description using a single equation. Nevertheless, the equation \eqref{eq58} implicitly defines $\varpi(t,\Gamma)$ through a nonlinear operator equation. While (semi-)closed-form solutions for $\varpi(t,\Gamma)$ may exist under special distributional assumptions on $\Gamma$, the general case resists analytical treatment. Therefore, we attempt to discretize $\Gamma$. For discrete distributions, the outer integral over $\Gamma$ reduces to a finite sum, transforming the system of ODEs into a more tractable form.

\subsubsection{Special case with $n$ possible risk aversions}
Suppose the random risk aversion is considered to conform to an $n$-point distribution having possible outcomes $\gamma_{i}>0\;(i=1,\cdot\cdot\cdot,n)$, which occur with probabilities $p_{i}=p(\gamma_{i})$ and $\sum\nolimits_{i=1}^{n}p_{i}=1$. Then, one has
\begin{equation}\label{eq59}
\varpi(t,\Gamma)=\frac{\sum\limits_{i=1}^{n}(g^{\gamma_{i}}(t))^{\frac{\gamma_{i}}{1-\gamma_{i}}}\gamma_{i}p_{i}}{\sum\limits_{i=1}^{n} (g^{\gamma_{i}}(t))^{\frac{\gamma_{i}}{1-\gamma}_{i}}p_{i}},
\end{equation}
where, for $i=1,\cdot\cdot\cdot,n$,
\begin{align}\label{eq60}
  \frac{\partial g^{\gamma_{i}}(t)}{\partial t}=\frac{1-\gamma_{i}}{\gamma_{i}}\left[-A-\beta-\frac{(\mu- r)^{2}}
{\sigma^{2}\varpi(t,\Gamma)}-\frac{a^{2} \eta^{2}_{2}}{b^{2}\varpi(t,\Gamma)}+\frac{0.5\gamma_{i} a^{2} \eta^{2}_{2}}
{b^{2}(\varpi(t,\Gamma))^{2}}+\frac{0.5\gamma_{i}(\mu- r)^{2}}
{\sigma^{2}(\varpi(t,\Gamma))^{2}}\right] g^{\gamma_{i}}(t).
\end{align}
In addition, the equation \eqref{eq56} reduces to
\begin{align}\label{eq61}
  \frac{\partial \varpi(t,\Gamma)}{\partial t}
=\frac{ 0.5}
{(\varpi(t,\Gamma))^{2}}\left(\frac{ a^{2} \eta^{2}_{2}}
{b^{2}}+\frac{(\mu- r)^{2}}
{\sigma^{2}}\right)\times\left[\frac{\sum\limits_{i=1}^{n} \gamma_{i}^{2}(g^{\gamma_{i}}(t))^{\frac{\gamma_{i}}{1-\gamma_{i}}}p_{i}}{\sum\limits_{i=1}^{n} (g^{\gamma_{i}}(t))^{\frac{\gamma_{i}}{1-\gamma_{i}}}p_{i}}-(\varpi(t,\Gamma))^{2}\right].
\end{align}
Consequently, equations \eqref{eq60} and \eqref{eq61} constitute an $n$-dimensional system of ODEs, as opposed to the infinite-dimensional system when
$\Gamma$ follows the general distribution. As a result, the inquiry regarding the existence of a solution within the general infinite-dimensional ODE system is essentially reduced to the question of the solvability of this $n$-dimensional system.

Having established these foundations, we proceed to characterize the equilibrium strategy for reinsurance and investment with $\varphi^{\gamma}$ being the power utility in the case of $n$ possible risk aversions.

\begin{theorem}\label{theorem4.2}
In the situation where there are $n$ possible risk aversions $\gamma_{i}\in(0,\epsilon_{2}]$ with $0<\epsilon_{2}<1$ for $i=1,\cdot\cdot\cdot,n$, assume that the utility function $\varphi^{\gamma}$ is given by the equation \eqref{eq51}, $\eta=0$ and the condition $(A2)$ holds. Moreover, assume that there exists at least a local solution to $g^{\gamma_{i}}(t)$ in the equation \eqref{eq60} for $i=1,\cdot\cdot\cdot,n$. Then the equilibrium reinsurance and investment strategy is captured by
\begin{equation}\label{eq66}
\hat{q}(t)=\frac{a \eta_{2} (x+\beta m_{1})}{b^{2}\varpi(t,\Gamma)},\quad
\hat{\pi}(t)=\frac{(\mu- r) (x+\beta m_{1})}{\sigma^{2}x\varpi(t,\Gamma)}
\end{equation}
and the corresponding equilibrium value function is
$V(t,x,m_{1})=(x+\beta m_{1})\sum\limits_{i=1}^{n}(g^{\gamma_{i}}(t))^{\frac{\gamma_{i}}{1-\gamma_{i}}}p_{i},$
where $\varpi(t,\Gamma)$ is given by the equation \eqref{eq59}.
\end{theorem}
\begin{proof}
See Appendix G.
\end{proof}

\begin{remark}
In order to ensure the existence of equilibrium reinsurance and investment strategy, we assume that the equation \eqref{eq60} for $g^{\gamma_{i}}(t)$ has a local solution in Theorem \ref{theorem4.2}, which seems difficult to be verified theoretically. Thus, similar to the work \cite{Desmettre2023}, we will conduct an analysis by numerical experiments in Section 5. On the other hand, as pointed out in \cite{Desmettre2023}, if the solution to the equation \eqref{eq60} explodes, then one can set a minor constraint that the reinsurance and investment horizon is shorter than the explosion time of the local solution.
\end{remark}

\begin{corollary}
If $\alpha=h=B=C=\beta=0$, which indicates that there is no inclusion of any delay within the model, then the equilibrium reinsurance and investment strategy becomes
$
\left(\hat{q}(t),\hat{\pi}(t)\right)=\left(\frac{a \eta_{2} x}
{b^{2}\varpi(t,\Gamma)},\frac{(\mu- r) }
{\sigma^{2}\varpi(t,\Gamma)}\right)
$
and the corresponding equilibrium value function is characterized by
$V(t,x)=x\sum\limits_{i=1}^{n}(g^{\gamma_{i}}(t))^{\frac{\gamma_{i}}{1-\gamma_{i}}}p_{i},$
where $\varpi(t,\Gamma)$ is  expressed by the equation \eqref{eq59} with $g^{\gamma_{i}}(t)$ satisfying
\begin{align*}
  \frac{\partial g^{\gamma_{i}}(t)}{\partial t}=\frac{1-\gamma_{i}}{\gamma_{i}}\left[-r-\frac{(\mu- r)^{2}}
{\sigma^{2}\varpi(t,\Gamma)}-\frac{a^{2} \eta^{2}_{2}}{b^{2}\varpi(t,\Gamma)}+\frac{0.5\gamma_{i} a^{2} \eta^{2}_{2}}
{b^{2}(\varpi(t,\Gamma))^{2}}+\frac{0.5\gamma_{i}(\mu- r)^{2}}
{\sigma^{2}(\varpi(t,\Gamma))^{2}}\right] g^{\gamma_{i}}(t).
\end{align*}
\end{corollary}

\begin{remark}
 It follows from equations \eqref{eq59} and \eqref{eq60} that the solution to $g^{\gamma_{i}}(t)$ is influenced by the model parameters of both the risky asset and the insurance market. As a result, $\varpi(t,\Gamma)$ and $(\hat{q}(t),\hat{\pi}(t))$ are also affected by these parameters. This implies that under the Black-Scholes financial market, when the financial market and the insurance market are independent, the equilibrium investment strategy is influenced not only by the parameters of the risky asset but also by those of the insurance market. The same phenomenon exists for the equilibrium reinsurance strategy. However, these phenomena do not exist in the case of the exponential utility function with random risk aversion (see Theorem \ref{theorem4.1}). Furthermore, we would like to mention that in previous literature considering insurance and investment problems under constant risk aversion coefficients, when the financial market and the insurance market are independent, the investment strategy is not affected by the parameters of the insurance market and the reinsurance strategy is not influenced by the parameters of the risky asset (see, for instance, the equation \eqref{eq62} hereafter). Thus, our newly derived results suggest that taking into account the random risk aversion in reinsurance and investment problems is both meaningful and essential. On the other hand, we find that Li and Li \cite{LiY2013Optimal} studied a class of mean-variance reinsurance and investment problems under state dependent risk aversion, in which when the risk aversion function exhibits inverse proportionality to wealth, they obtained the semi-analytical equilibrium reinsurance and investment strategy. The expression of their strategy shows that when the financial market and the insurance market are independent, both reinsurance and investment strategies are influenced by parameters of risky assets and insurance market parameters. Although there are similarities between our conclusions and those of \cite{LiY2013Optimal}, we assumed that the risk aversion is a random variable satisfying a certain distribution and derived them under the power utility function by the approach of using expected certainty equivalent.
\end{remark}

\subsubsection{Special case with one possible risk aversion}
If the distribution of the risk aversion simplifies to a single outcome $\gamma$, then one can derive that $\frac{\partial g^{\gamma}(t)}{\partial t}=\frac{1-\gamma}{\gamma}\left[-A-\beta-\frac{0.5(\mu- r)^{2}}
{\sigma^{2}\gamma}-\frac{0.5a^{2} \eta^{2}_{2}}{b^{2}\gamma}\right] g^{\gamma}(t),$
which yields
$
  g^{\gamma}(t)=\exp\left\{\frac{1-\gamma}{\gamma}\left(A+\beta+\frac{0.5(\mu- r)^{2}}
{\sigma^{2}\gamma}+\frac{0.5a^{2} \eta^{2}_{2}}{b^{2}\gamma}\right)(T-t)\right\}.
$
Moreover, the equation \eqref{eq66} can be rewritten as
\begin{eqnarray}\label{eq62}
\hat{q}(t)=\frac{a \eta_{2} (x+\beta m_{1})}{b^{2}\gamma},\quad
   \hat{\pi}(t)=\frac{(\mu- r) (x+\beta m_{1})}{\sigma^{2}x\gamma}
\end{eqnarray}
and the candidate equilibrium value function is captured by
\begin{equation}\label{eq63}
U(t,x,m_{1})=\exp\left\{\left(A+\beta+\frac{0.5(\mu- r)^{2}}
{\sigma^{2}\gamma}+\frac{0.5a^{2} \eta^{2}_{2}}{b^{2}\gamma}\right)(T-t)\right\}(x+\beta m_{1}).
\end{equation}

By the fact that $g^{\gamma}(t)$ possesses a unique global solution, one can entirely adhere to the approach used for $n$ possible risk aversions to validate the admissibility requirements and the assumptions in Theorem \ref{theorem3.1}, yet in a more straightforward manner. Consequently, the following results can be obtained.
\begin{theorem}\label{theorem4.3}
In the situation where there is only one possible risk aversion $\gamma\in(0,\epsilon_{2}]$, assume that the utility function $\varphi^{\gamma}$ is given by the equation \eqref{eq51}, $\eta=0$ and the condition $(A2)$ holds. Then the equilibrium reinsurance and investment strategy and the corresponding equilibrium value function are given by equations \eqref{eq62} and \eqref{eq63}, respectively.
\end{theorem}

\begin{corollary}
If $\alpha=h=B=C=\beta=0$, which suggests that the model does not incorporate any delay, then the equilibrium reinsurance and investment strategy becomes
$
\left(\hat{q}(t),\hat{\pi}(t)\right)=\left(\frac{a \eta_{2}x}{b^{2}\gamma},\frac{(\mu- r)}{\sigma^{2}\gamma}\right)
$
and the corresponding equilibrium value function is $V(t,x)=x\exp\left\{\left(r+\frac{0.5(\mu- r)^{2}}
{\sigma^{2}\gamma}+\frac{0.5a^{2} \eta^{2}_{2}}{b^{2}\gamma}\right)(T-t)\right\}$.
\end{corollary}

\section{Numerical illustrations}
In this section, we will present numerical experiments to analyze the impacts of delay parameters and risk aversion parameters on the equilibrium reinsurance and investment strategies under the exponential utility function and power utility function, respectively. Meanwhile, under the Black-Scholes model, we will analyze the influences of parameters in the insurance and financial markets on the equilibrium reinsurance and investment strategy under the power utility function. In the sequel, the sensitivity of the equilibrium reinsurance and investment strategy will be investigated by varying one parameter at a time in each figure and without loss of generality, all investigations will be performed at the initial time $t=0$ for convenience.

In the numerical experiments, unless stated otherwise, the essential parameters of the insurance and financial markets and the delay parameters are sourced from \cite{Bi2019} and \cite{Yuan2023} separately, as presented in Table \ref{tab1}. Furthermore, we follow \cite{Bi2019} to set $x_{0}=0.6$ and $T=2(years)$ and then we have $m_{10}=\frac{x_{0}(1-e^{-\alpha h})}{\alpha}=1.2(1-e^{-1})$. In addition, analogous to \cite{Desmettre2023}, we characterize the random risk aversion by a two point-distribution, with parameters detailed for two specific cases: (I) $\gamma_{1}=0.5$, $\gamma_{2}=0.9$ and $p_{1}=p_{2}=0.5$; (II) $\gamma_{1}=0.5$, $\gamma_{2}=0.9$, $p_{1}=0.8$ and $p_{2}=0.2$.

\begin{table}[H]
 \caption{Values of insurance and financial market parameters and delay parameters.}
 \label{tab1}
 \small
  \centering
   \begin{tabular}{>{\hfil}p{1cm}<{\hfil} >{\hfil}p{1.1cm}<{\hfil} >{\hfil}p{1cm}<{\hfil} >{\hfil}p{1cm}<{\hfil} >{\hfil}p{1.2cm}<{\hfil}
   >{\hfil}p{1cm}<{\hfil} >{\hfil}p{1cm}<{\hfil} >{\hfil}p{1.1cm}<{\hfil} >{\hfil}p{1.1cm}<{\hfil} >{\hfil}p{1cm}<{\hfil}>{\hfil}p{0.6cm}<{\hfil}  }
     \hline
{$\eta_{1}$}& {$\eta_{2}$}& 	{$\lambda_{1}$} & {$\mu_{1}$} & {$\mu_{2}$} & {$r$} & {$\mu$}&{$\sigma$} &{$\alpha$}&{$\beta$} & {$h$}   \\\hline
0.3 & 0.5 &   1& 0.1 &0.2 & 0.1 & $0.2$ & 0.6 &$0.5$& 0.05 & 2 \\
     \hline
   \end{tabular}
\end{table}

\subsection{The case of the exponential utility}
\begin{figure}
    \centering
    \begin{minipage}{0.32\textwidth}
     \centering
      \includegraphics[width=\linewidth]{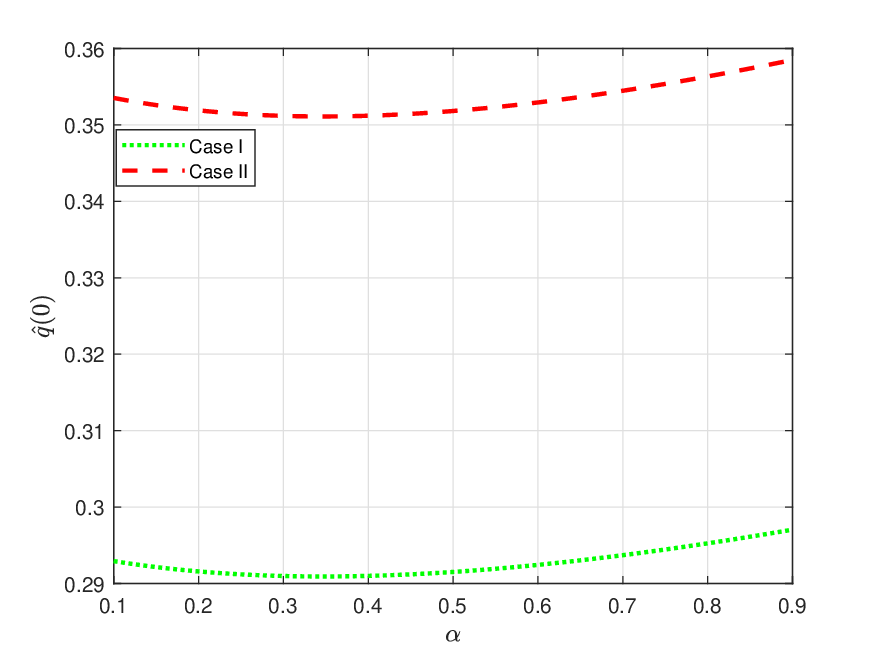}
      \caption*{(a) The impact of $\alpha$ on $\hat{q}(0)$ }
    \end{minipage}
     \begin{minipage}{0.32\textwidth}
     \centering
      \includegraphics[width=\linewidth]{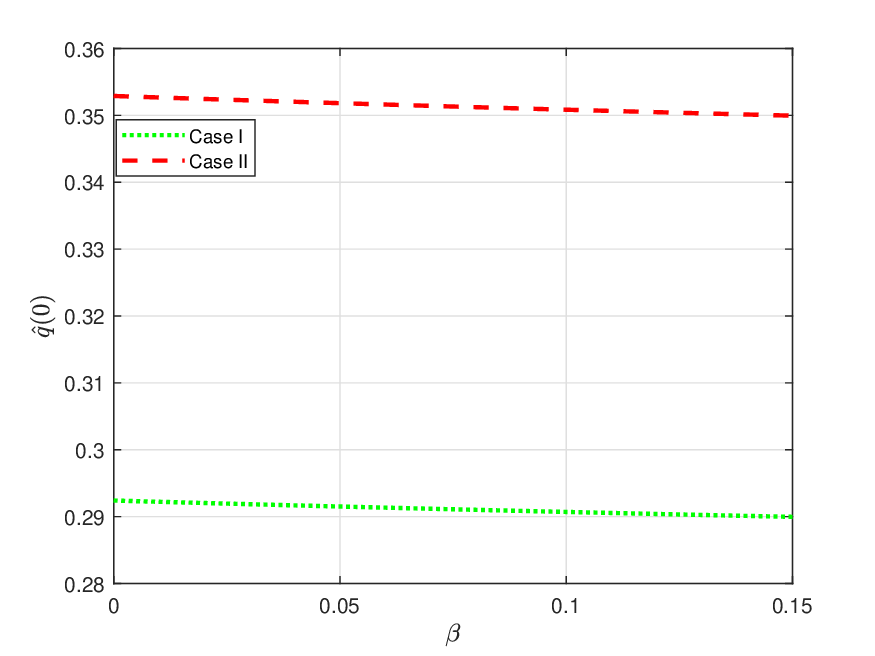}
      \caption*{(b) The impact of $\beta$ on $\hat{q}(0)$ }
    \end{minipage}
    \begin{minipage}{0.32\textwidth}
    \centering
      \includegraphics[width=\linewidth]{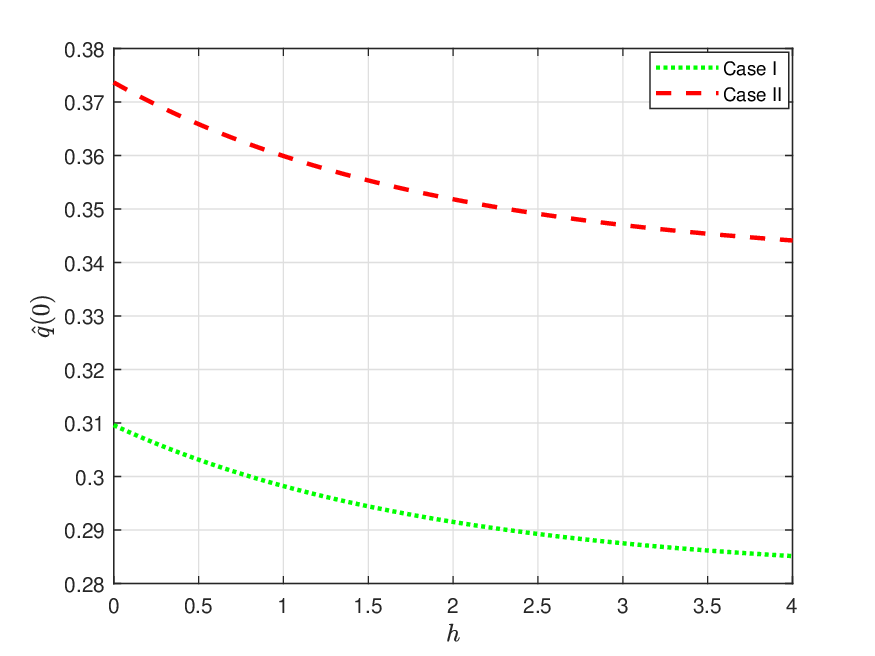}
      \caption*{(c) The impact of $h$ on $\hat{q}(0)$}
    \end{minipage}
    \caption{The impacts of $\alpha$, $\beta$ and $h$ on the equilibrium reinsurance strategy.}
\label{fig1}
\end{figure}

\begin{figure}
    \centering
    \begin{minipage}{0.32\textwidth}
     \centering
      \includegraphics[width=\linewidth]{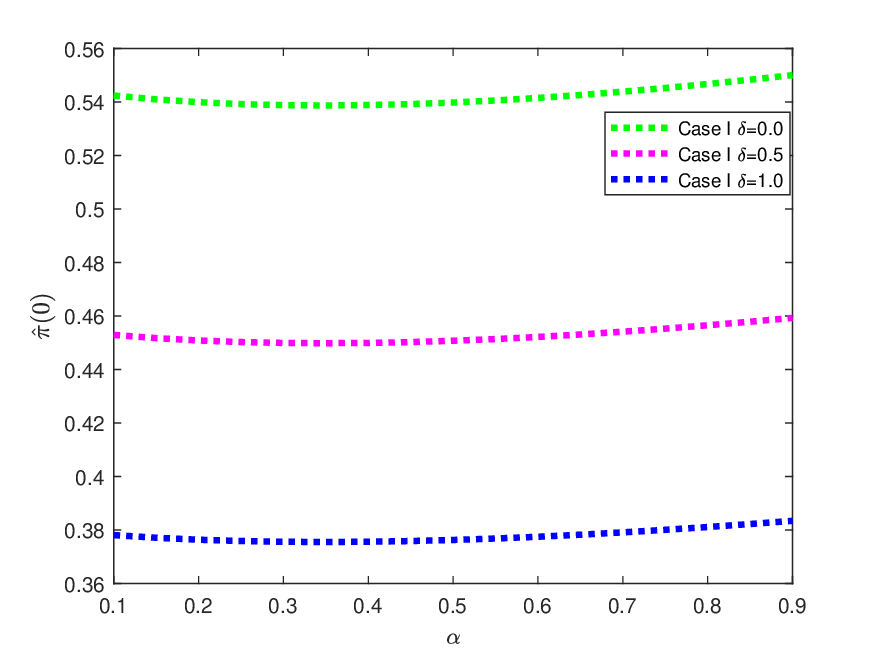}
      \caption*{(a) The impact of $\alpha$ on $\hat{\pi}(0)$ }
    \end{minipage}
     \begin{minipage}{0.32\textwidth}
     \centering
      \includegraphics[width=\linewidth]{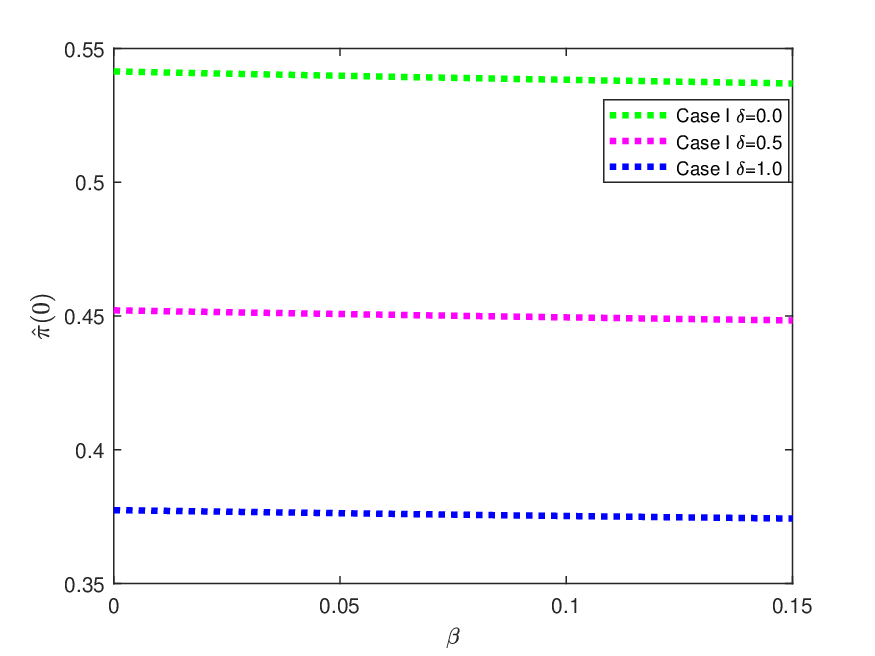}
      \caption*{(b) The impact of $\beta$ on $\hat{\pi}(0)$ }
    \end{minipage}
    \begin{minipage}{0.32\textwidth}
    \centering
      \includegraphics[width=\linewidth]{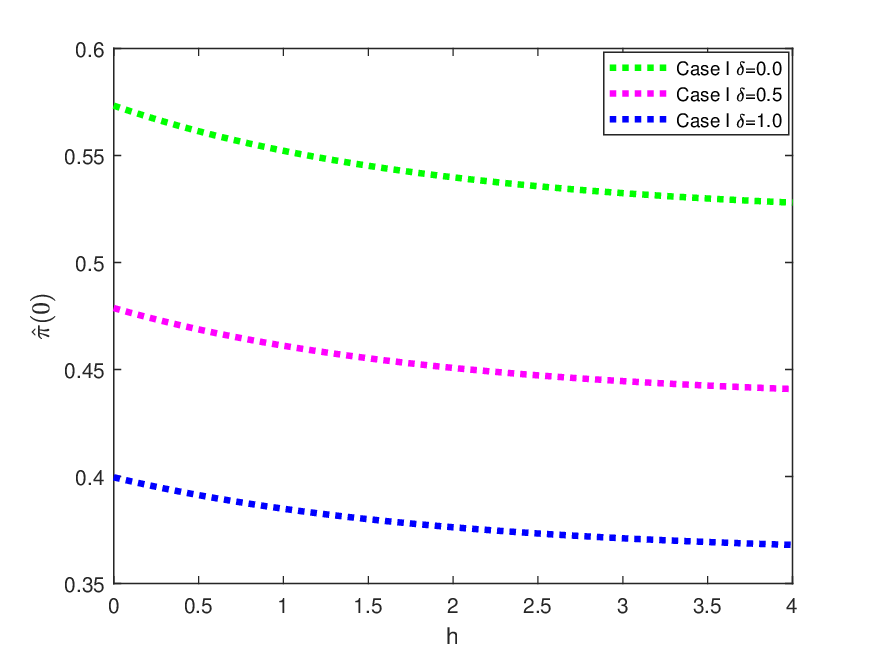}
      \caption*{(c) The impact of $h$ on $\hat{\pi}(0)$}
    \end{minipage}
        \begin{minipage}{0.32\textwidth}
     \centering
      \includegraphics[width=\linewidth]{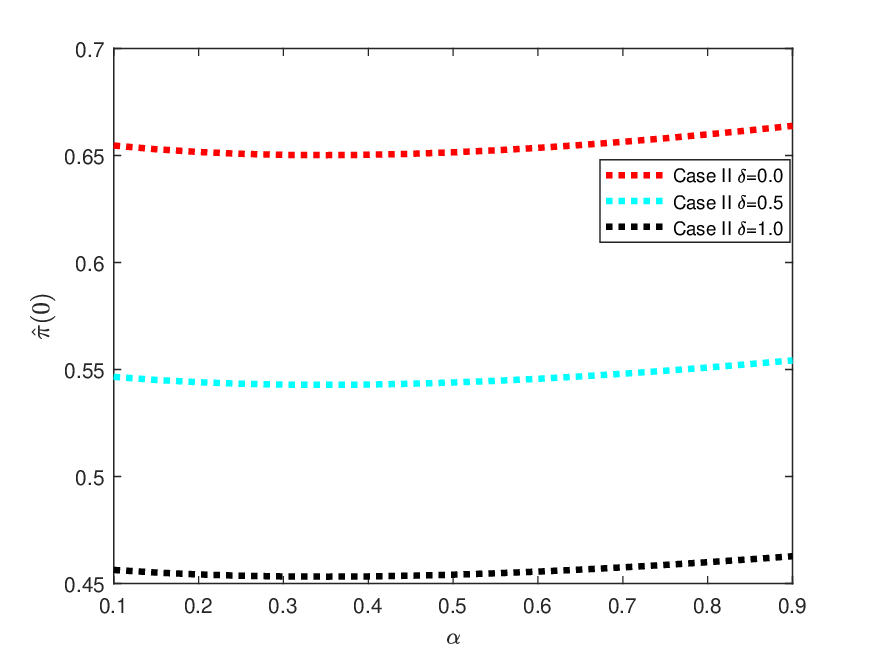}
      \caption*{(d) The impact of $\alpha$ on $\hat{\pi}(0)$ }
    \end{minipage}
     \begin{minipage}{0.32\textwidth}
     \centering
      \includegraphics[width=\linewidth]{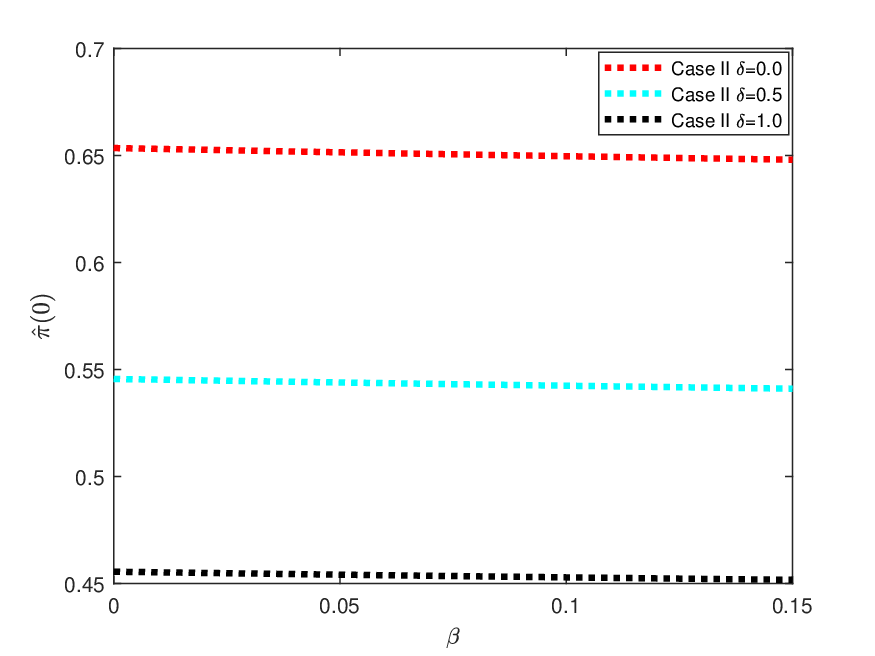}
      \caption*{(e) The impact of $\beta$ on $\hat{\pi}(0)$ }
    \end{minipage}
    \begin{minipage}{0.32\textwidth}
    \centering
      \includegraphics[width=\linewidth]{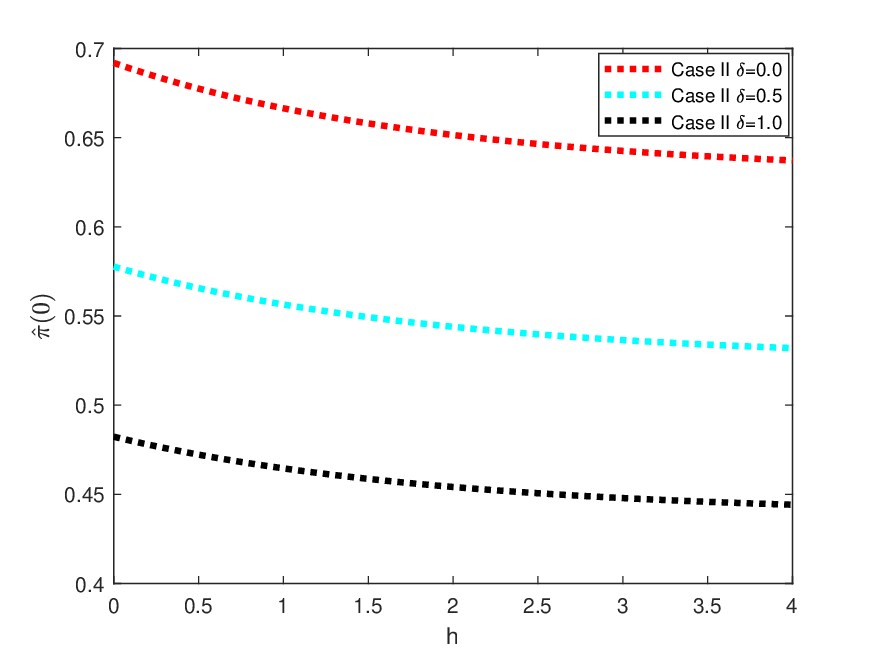}
      \caption*{(f) The impact of $h$ on $\hat{\pi}(0)$}
    \end{minipage}
    \caption{The impacts of $\alpha$, $\beta$ and $h$ on the equilibrium investment strategy.}
\label{fig7}
\end{figure}

In this subsection, we adopt $\bar{s}_{1}=1.2$ and analyze the impacts of $\alpha$, $\beta$ and $h$ on the equilibrium strategy for reinsurance and investment under the exponential utility by Figures \ref{fig1} and \ref{fig7}. As seen in subgraphs \ref{fig1} $(a)$ and subgraphs \ref{fig7} $(a)$ and $(d)$,  both $\hat{q}(0)$ and $\hat{\pi}(0)$ first decrease monotonically as $\alpha$ increases and then increase monotonically. This shows that prior to the extreme point, a higher average parameter corresponds to a smaller reinsurance strategy and a smaller investment proportion in the risky asset, which means the insurer adopts a more conservative reinsurance and investment strategy. Subsequent to the extreme point, a higher average parameter is associated with a greater reinsurance strategy and a greater investment proportion in the risky asset, which implies the insurer uses a more aggressive reinsurance and investment strategy. On the whole, the influence exerted by the average parameter $\alpha$ on the equilibrium reinsurance and investment strategy undergoes a transformation contingent upon the magnitude of the parameter. Moreover, subgraphs \ref{fig1} $(b)$ and subgraphs \ref{fig7} $(b)$ and $(e)$ disclose that both $\hat{q}(0)$ and $\hat{\pi}(0)$ exhibit a decreasing trend concomitant with the increase of $\beta$. With the increment of $\beta$, the cumulative delayed information related to the insurer's wealth assumes a greater weight whereas the current wealth carries a smaller proportion. This suggests that the more emphasis the insurer places on the cumulative delayed information of the insurer's wealth under the exponential utility, the smaller the reinsurance and the proportion to the risky asset will be. Furthermore, subgraphs \ref{fig1} $(c)$ and subgraphs \ref{fig7} $(c)$ and $(f)$ present that the quantities $\hat{q}(0)$ and $\hat{\pi}(0)$ undergo a reduction in tandem with the growth of $h$. As $h$ increases, the averaging horizon extends, which suggests the insurer adopting the exponential utility accepts to reduce the proportion allocated to the risky asset and lower the reinsurance ratio. All aforementioned behaviors correspond to the results in equations \eqref{eq70} and \eqref{eq71}. What is more, Figure \ref{fig7} shows that $\hat{\pi}(0)$ decreases with the increase of $\delta$. Additionally, it is readily apparent that a higher expected risk aversion degree for the insurer corresponds to a more cautious reinsurance and investment strategy. Specifically, the insurer with greater expected risk aversion is inclined to acquire additional reinsurance and allocate a smaller proportion to the risky asset. It is also intriguing to note that this phenomenon persists in the subsequent Figures \ref{fig2}-\ref{fig6} under varying parameter configurations.

\subsection{The case of the power utility}
Next, we take $\delta=0$ and  $\eta_{1}=\eta_{2}=0.3$ and carry out numerical test to exhibit how delay parameters and insurance-financial market parameters influence $(\hat{q}(0), \hat{\pi}(0))$ under the power utility through Figures \ref{fig2}-\ref{fig6}.

\begin{figure}
    \centering
    \begin{minipage}{6.5cm}
      \includegraphics[width=6cm]{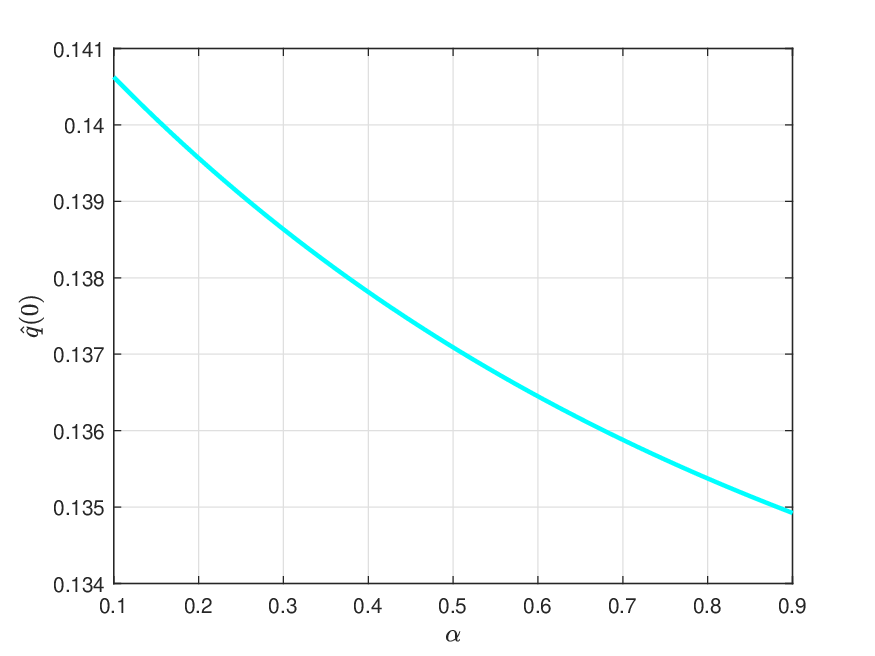}
      \caption*{(a) The impact of $\alpha$ on $\hat{q}(0)$ in case (I)}
    \end{minipage}
     \begin{minipage}{6.5cm}
      \includegraphics[width=6cm]{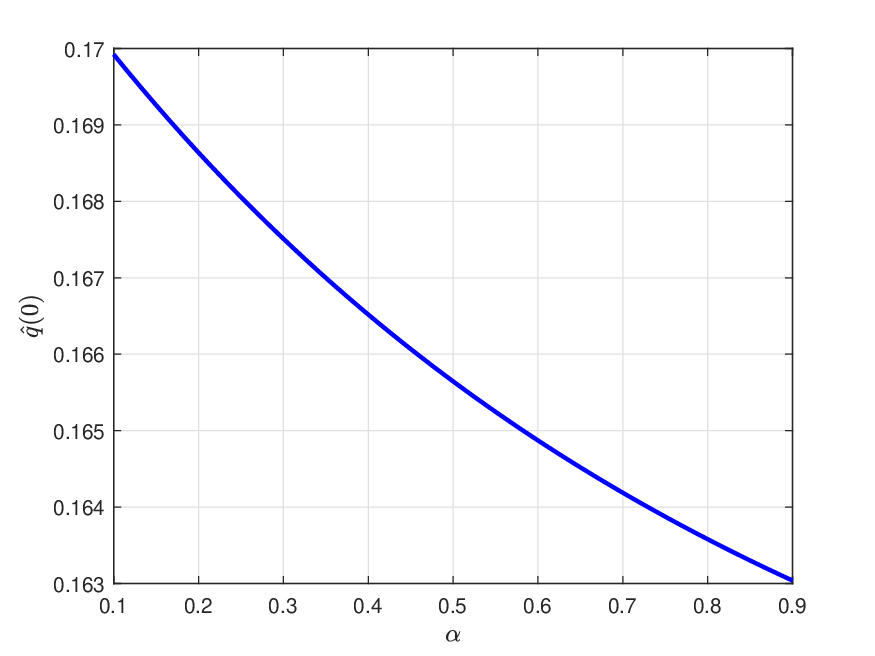}
      \caption*{(b) The impact of $\alpha$ on $\hat{q}(0)$ in case (II)}
    \end{minipage}
    \begin{minipage}{6.5cm}
      \includegraphics[width=6cm]{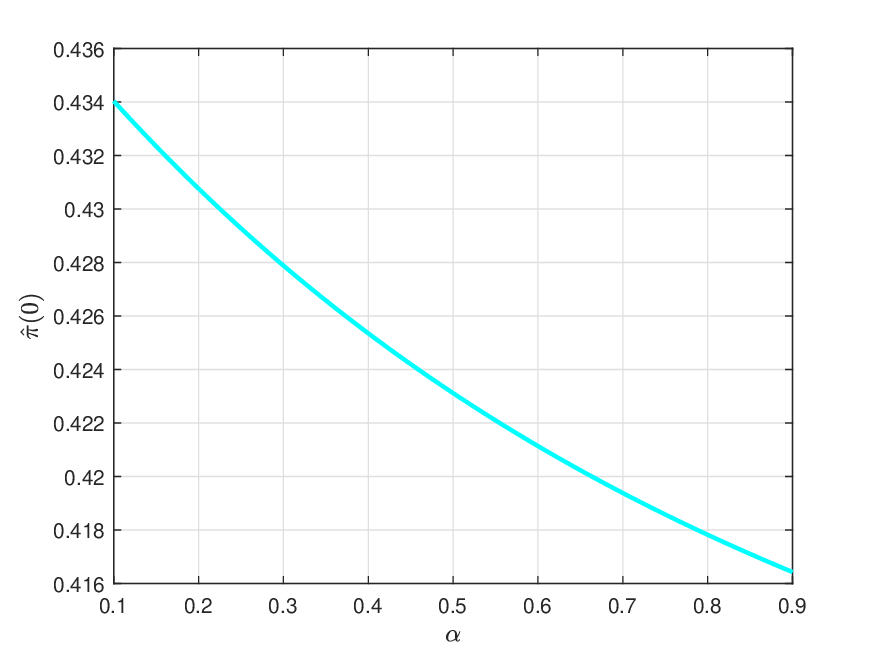}
      \caption*{(c) The impact of $\alpha$ on $\hat{\pi}(0)$ in case (I)}
    \end{minipage}
    \begin{minipage}{6.5cm}
      \includegraphics[width=6cm]{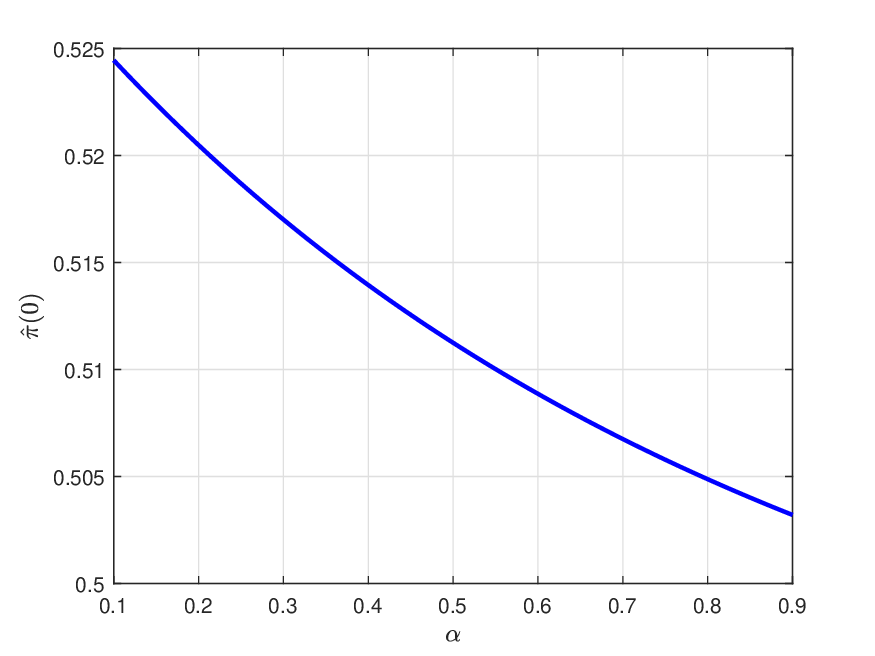}
      \caption*{(d) The impact of $\alpha$ on $\hat{\pi}(0)$ in case (II)}
    \end{minipage}
    \caption{The impacts of $\alpha$ on the equilibrium reinsurance and investment strategy.}
\label{fig2}
\end{figure}

Figures \ref{fig2} and \ref{fig3} demonstrate the impacts of delay parameters $\alpha$, $\beta$ and $h$ on $(\hat{q}(0), \hat{\pi}(0))$. Specifically, both $\hat{q}(0)$ and $\hat{\pi}(0)$ decrease as $\alpha$ increases, and increase as $\beta$ or $h$ increases. First of all, we can see from the expression of $M_{1}^{u}(t)$ that with an increase in $\alpha$, the weight of earlier wealth in the insurer's average delayed wealth diminishes, which reflects the insurer attach more importance to wealth nearer to the present time as $\alpha$ grows. Under the power utility function, the insurer would like to procure additional reinsurance and decrease the share of the risky asset in her/his portfolios. What is more, in accordance with the definition of $X^{u}(T)+\beta M_{1}^{u}(T)$, as the delay weight $\beta$ rises, the cumulative delayed information of the insurer's wealth carries more weight, while the current wealth's weight is reduced proportionally. Then, the insurer employing the power utility function is more inclined to reduce her/his demand for reinsurance and increase the proportion allocated to the risky asset in portfolios. Additionally, since the delay time $h$ grows, the span of the averaging period lengthens, which in turn stabilizes the cumulative delayed information regarding the insurer's wealth. This may motivate the insurer who uses the power utility function to decrease the demand for reinsurance and raise the share of the risky asset in investments.

\begin{figure}
    \centering
    \begin{minipage}{6.5cm}
      \includegraphics[width=6cm]{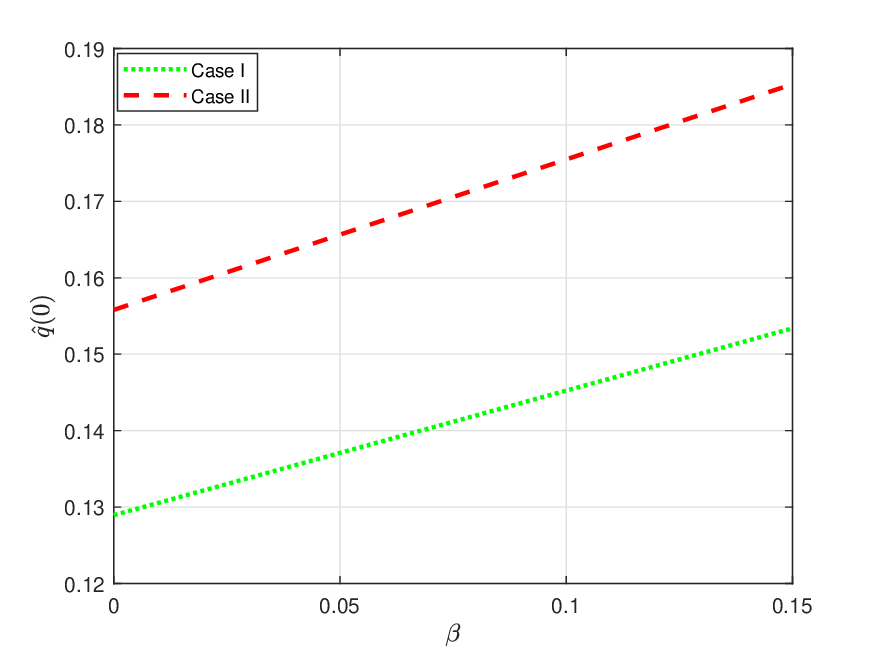}
      \caption*{(a) The impact of $\beta$ on $\hat{q}(0)$ }
    \end{minipage}
     \begin{minipage}{6.5cm}
      \includegraphics[width=6cm]{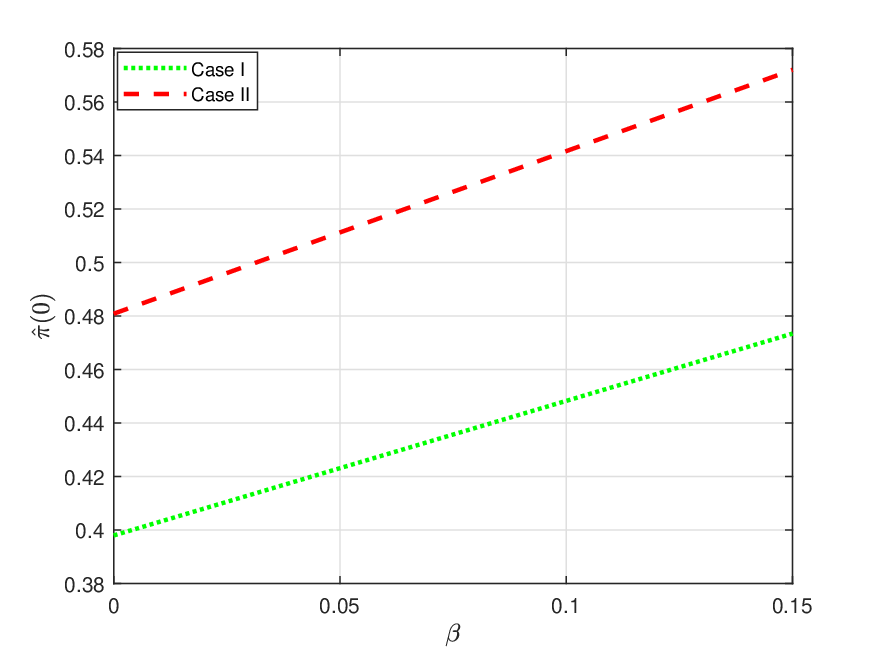}
      \caption*{(b) The impact of $\beta$ on $\hat{\pi}(0)$ }
    \end{minipage}
    \begin{minipage}{6.5cm}
      \includegraphics[width=6cm]{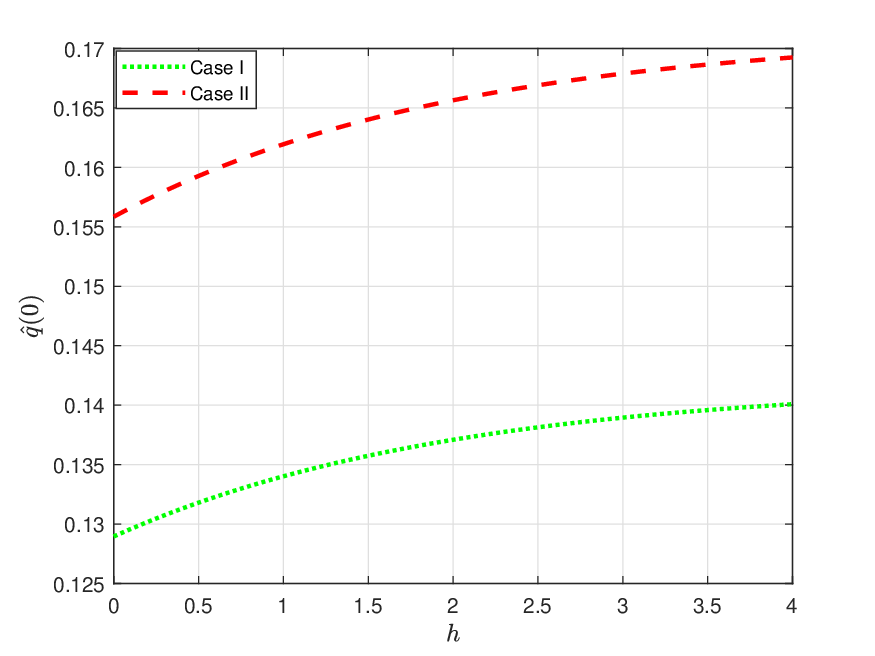}
      \caption*{(c) The impact of $h$ on $\hat{q}(0)$}
    \end{minipage}
    \begin{minipage}{6.5cm}
      \includegraphics[width=6cm]{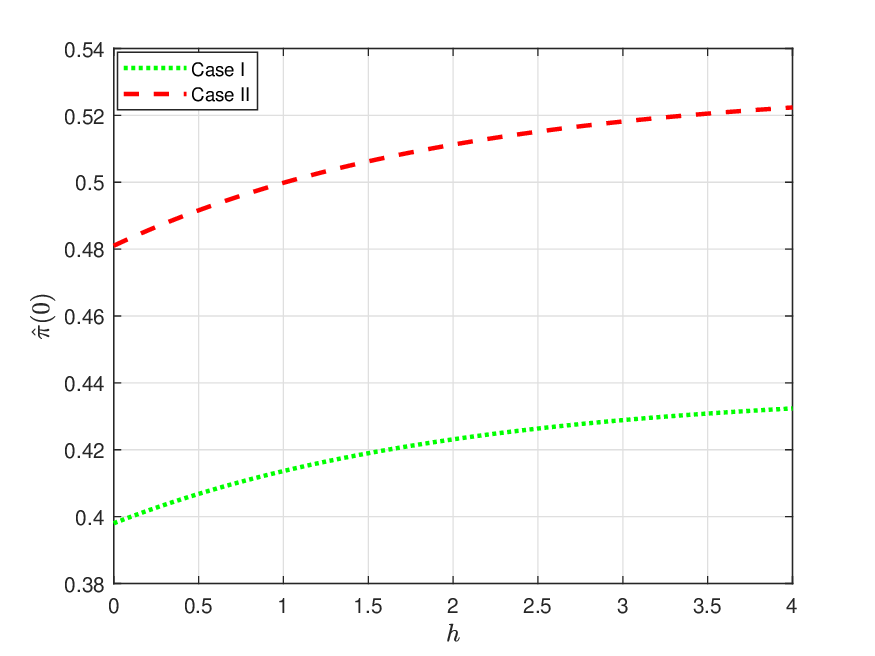}
      \caption*{(d) The impact of $h$ on $\hat{\pi}(0)$}
    \end{minipage}
    \caption{The impacts of $\beta$ and $h$ on the equilibrium reinsurance and investment strategy.}
\label{fig3}
\end{figure}

\begin{figure}
    \centering
    \begin{minipage}{0.32\textwidth}
     \centering
      \includegraphics[width=\linewidth]{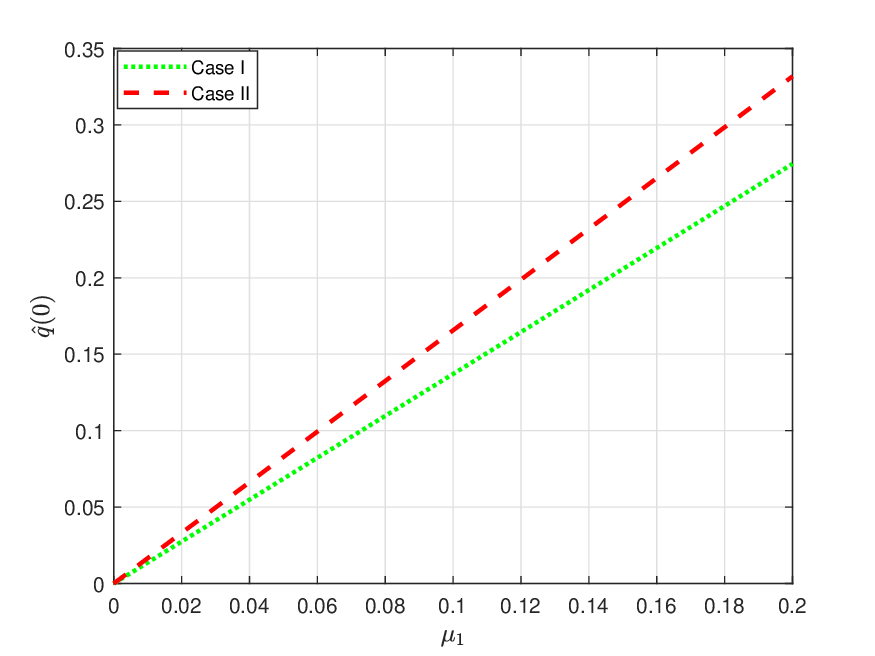}
      \caption*{(a) The impact of $\mu_{1}$ on $\hat{q}(0)$ }
    \end{minipage}
     \begin{minipage}{0.32\textwidth}
     \centering
      \includegraphics[width=\linewidth]{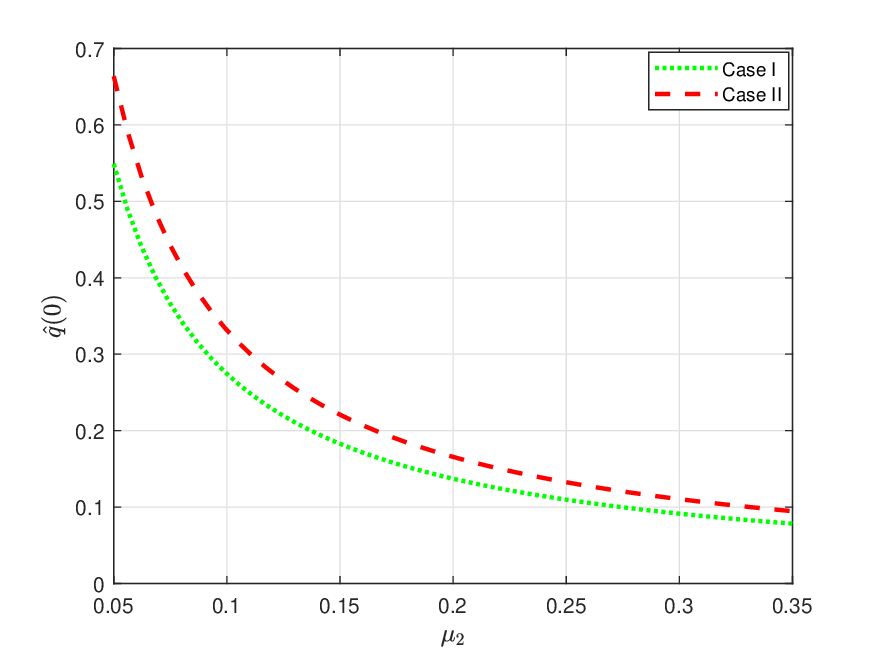}
      \caption*{(b) The impact of $\mu_{2}$ on $\hat{q}(0)$ }
    \end{minipage}
    \begin{minipage}{0.32\textwidth}
    \centering
      \includegraphics[width=\linewidth]{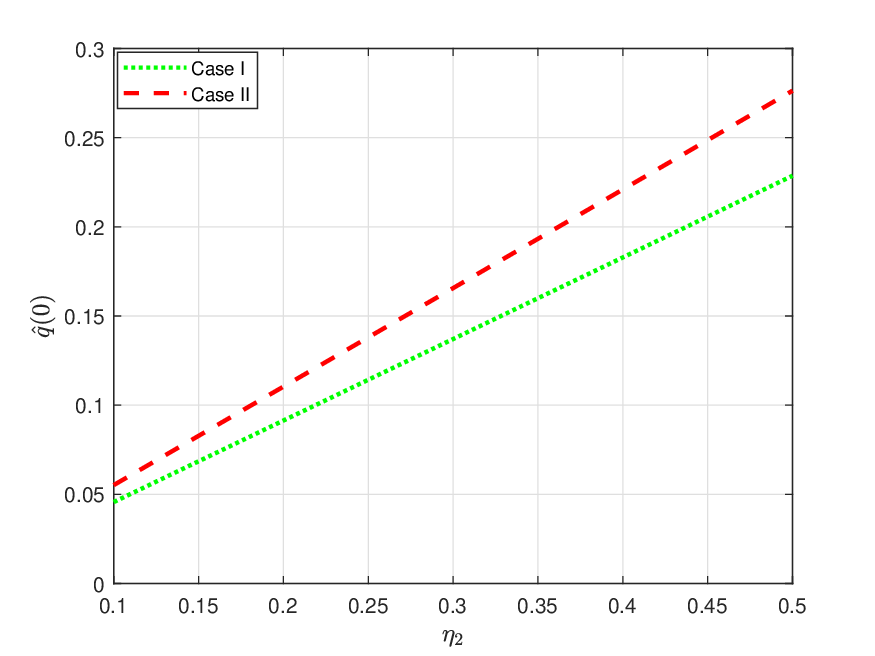}
      \caption*{(c) The impact of $\eta_{2}$ on $\hat{q}(0)$}
    \end{minipage}
    \caption{The impacts of $\mu_{1}$, $\mu_{2}$ and $\eta_{2}$ on the equilibrium reinsurance strategy.}
\label{fig4}
\end{figure}

A comparison of Figures \ref{fig1}-\ref{fig3} reveals that under the exponential utility and power utility, respectively, the same delay parameter changes affect the equilibrium reinsurance and investment strategy differently. This implies that the effects of delay on reinsurance and investment decisions are diverse, thereby reflecting the importance of considering delays.

Figures \ref{fig4} and \ref{fig5} display the effects of insurance market parameters $\mu_{1}$, $\mu_{2}$ and $\eta_{2}$ on $(\hat{q}(0), \hat{\pi}(0))$. Subgraph \ref{fig4} $(a)$ and subgraphs \ref{fig5} $(a)$ and $(d)$ illustrate that both $\hat{q}(0)$ and $\hat{\pi}(0)$ increase monotonically with $\mu_{1}$. We know that $\mu_{1}$ represents the first moment of the claim size. If $\mu_{1}$ increases, then it suggests that the insurer is required to pay a higher premium to the reinsurer, thus encouraging she/he to assume more insurance risks independently and correspondingly reduce the purchase of reinsurance. Under the premise that the insurer is willing to assume more risks, she/he also tends to direct a larger share of investments toward the risky asset to pursue higher returns. Moreover, subgraph \ref{fig4} $(b)$ and subgraphs \ref{fig5} $(b)$ and $(e)$ show that both $\hat{q}(0)$ and $\hat{\pi}(0)$ decrease monotonically with $\mu_{2}$. Since $\mu_{2}$ represents the second moment of the claim size, an increase in $\mu_{2}$ implies higher uncertainty within the insurance market. Consequently, the insurer faces higher risks in the insurance market, leading her/him to prefer purchasing more reinsurance to mitigate her/his own risks. Under the premise that the insurer seeks to avoid more risks, she/he is also willing to allocate a smaller proportion of investments to the risky asset. In addition, it can be seen from subgraph \ref{fig4} $(c)$ and subgraphs \ref{fig5} $(c)$ and $(f)$ that both $\hat{q}(0)$ and $\hat{\pi}(0)$ increase alongside a rise in $\eta_{2}$. If the reinsurer's safety loading $\eta_{2}$ increases, then it implies that the insurer needs to bear higher reinsurance costs. In this case, the insurer becomes more disposed to assume risks by her/his own accord and then reduces the purchase of reinsurance. When the insurer opts to shoulder more risks, she/he also exhibits a tendency to allocate a larger share of investment to the risky asset proportionally.

\begin{figure}
    \centering
    \begin{minipage}{0.32\textwidth}
     \centering
      \includegraphics[width=\linewidth]{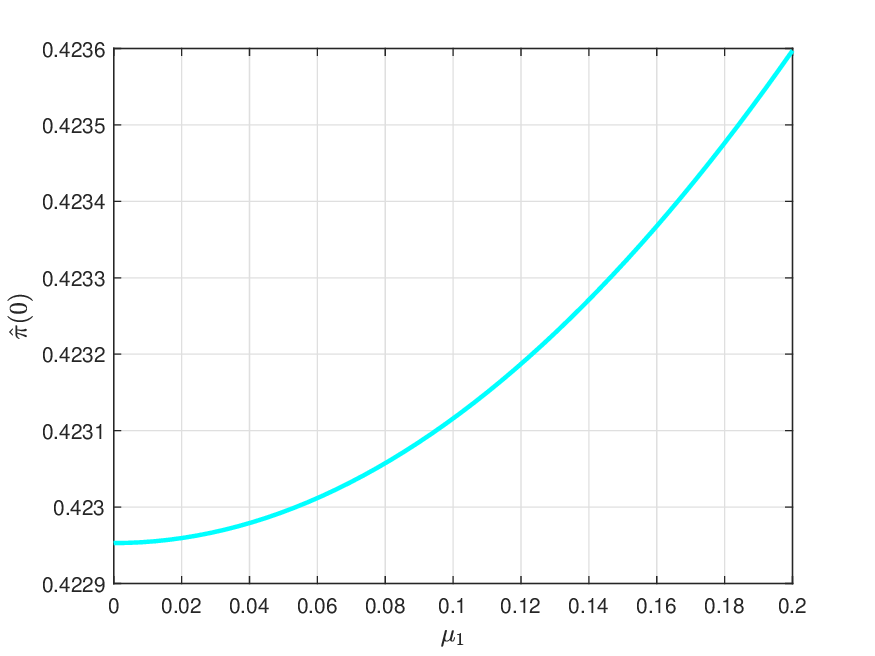}
      \caption*{(a) The impact of $\mu_{1}$ on $\hat{\pi}(0)$ in case (I)}
    \end{minipage}
     \begin{minipage}{0.32\textwidth}
     \centering
      \includegraphics[width=\linewidth]{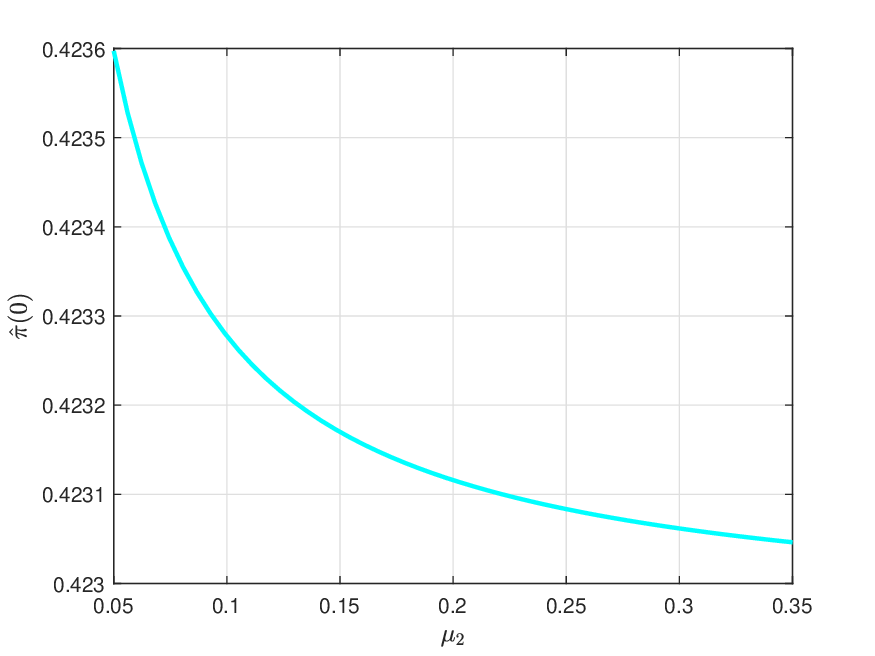}
      \caption*{(b) The impact of $\mu_{2}$ on $\hat{\pi}(0)$ in case (I)}
    \end{minipage}
    \begin{minipage}{0.32\textwidth}
    \centering
      \includegraphics[width=\linewidth]{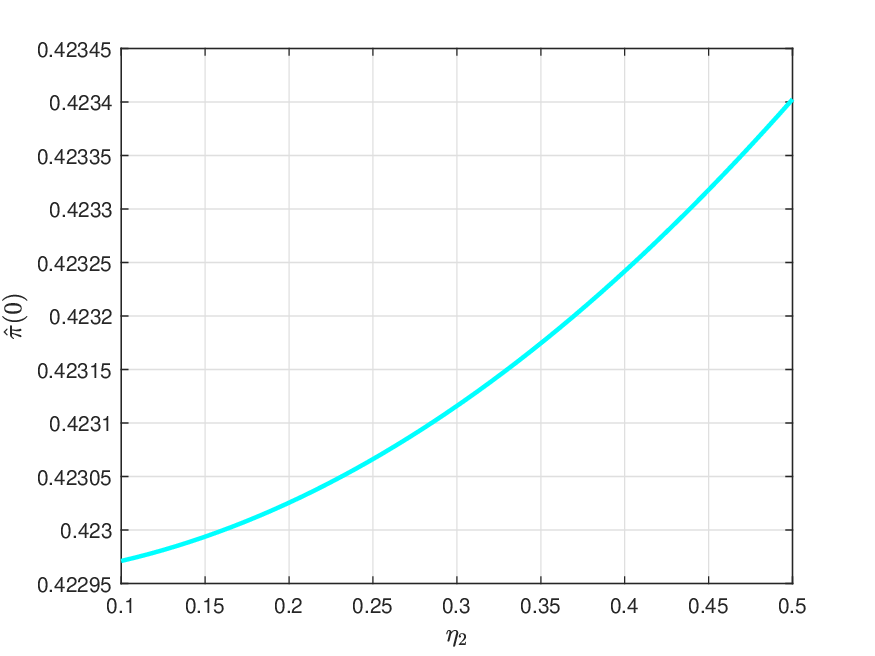}
      \caption*{(c) The impact of $\eta_{2}$ on $\hat{\pi}(0)$ in case (I)}
    \end{minipage}
    \begin{minipage}{0.32\textwidth}
     \centering
      \includegraphics[width=\linewidth]{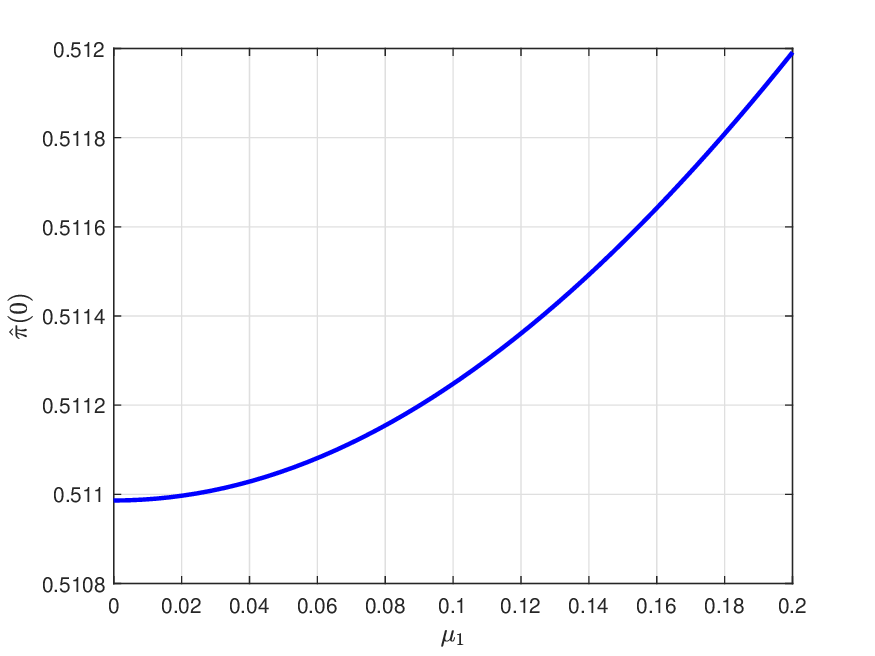}
      \caption*{(d) The impact of $\mu_{1}$ on $\hat{\pi}(0)$ in case (II)}
    \end{minipage}
     \begin{minipage}{0.32\textwidth}
     \centering
      \includegraphics[width=\linewidth]{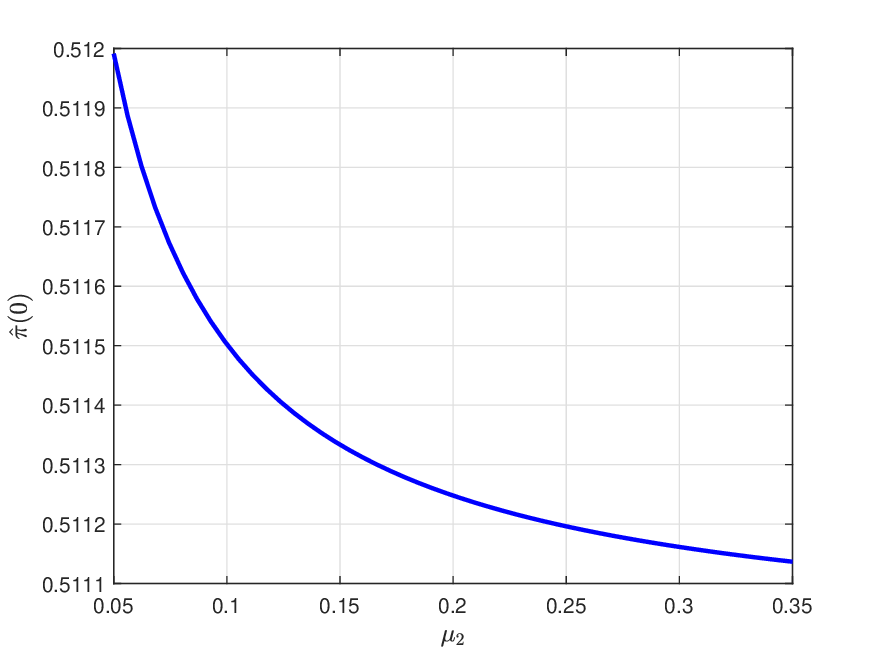}
      \caption*{(e) The impact of $\mu_{2}$ on $\hat{\pi}(0)$ in case (II)}
    \end{minipage}
    \begin{minipage}{0.32\textwidth}
    \centering
      \includegraphics[width=\linewidth]{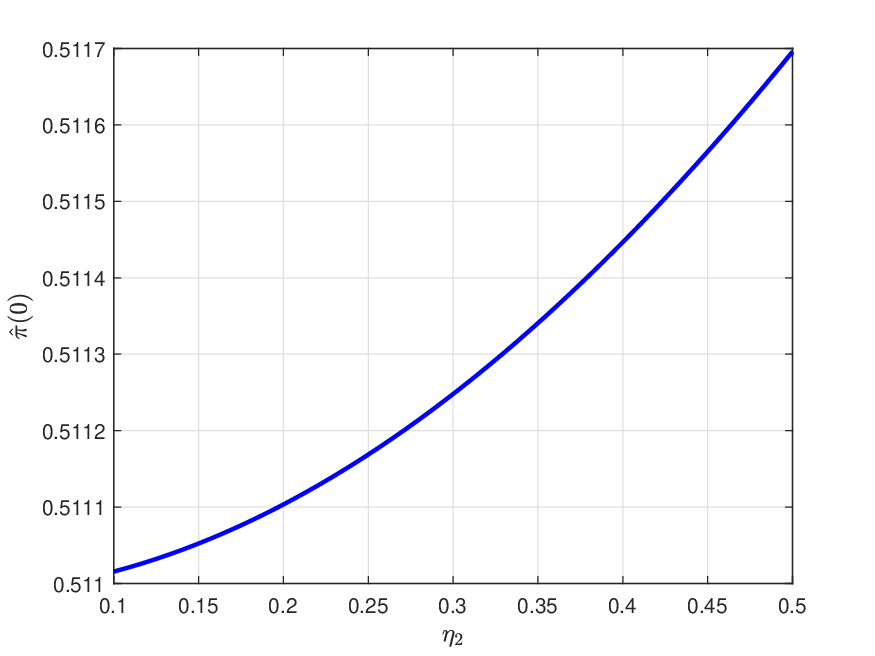}
      \caption*{(f) The impact of $\eta_{2}$ on $\hat{\pi}(0)$ in case (II)}
    \end{minipage}
    \caption{The impacts of $\mu_{1}$, $\mu_{2}$ and $\eta_{2}$ on the equilibrium investment strategy.}
\label{fig5}
\end{figure}

\begin{figure}
    \centering
    \begin{minipage}{0.32\textwidth}
     \centering
      \includegraphics[width=\linewidth]{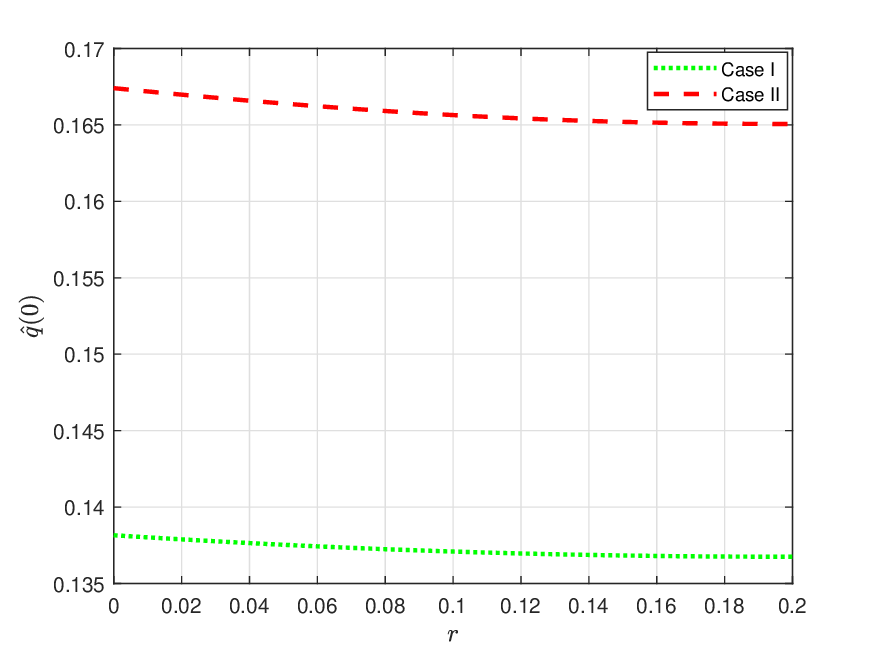}
      \caption*{(a) The impact of $r$ on $\hat{q}(0)$ }
    \end{minipage}
     \begin{minipage}{0.32\textwidth}
     \centering
      \includegraphics[width=\linewidth]{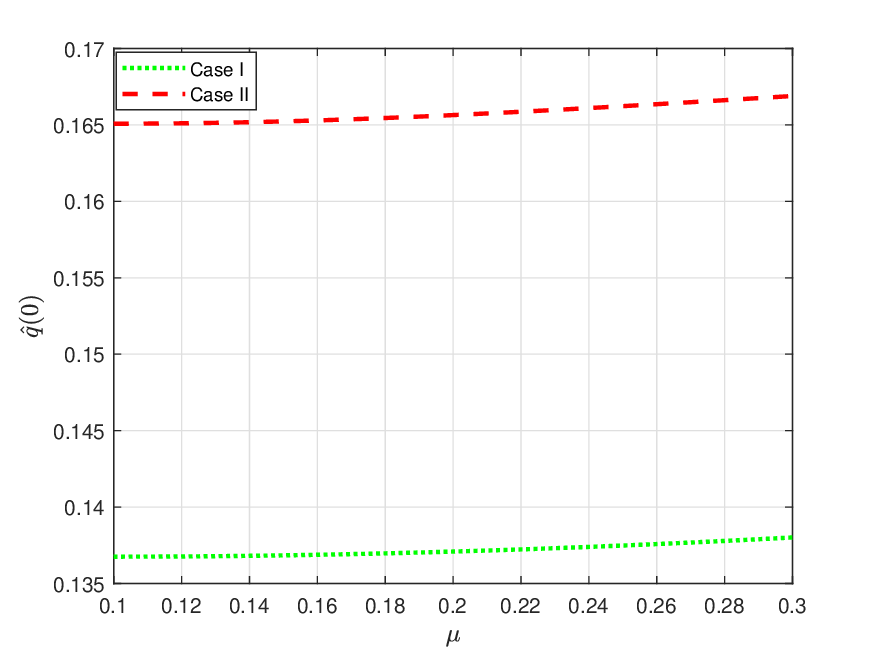}
      \caption*{(b) The impact of $\mu$ on $\hat{q}(0)$ }
    \end{minipage}
    \begin{minipage}{0.32\textwidth}
    \centering
      \includegraphics[width=\linewidth]{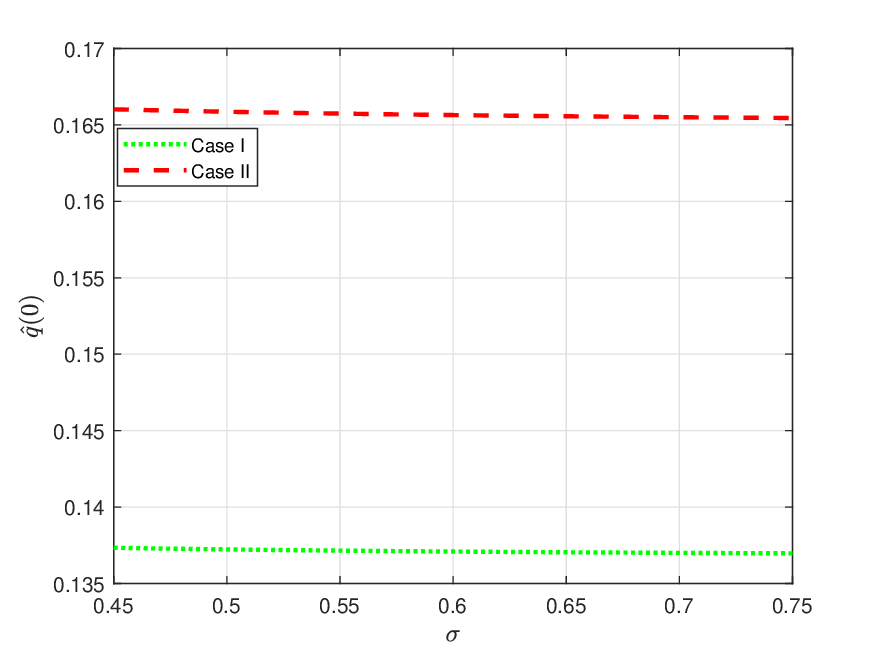}
      \caption*{(c) The impact of $\sigma$ on $\hat{q}(0)$}
    \end{minipage}
    \begin{minipage}{0.32\textwidth}
     \centering
      \includegraphics[width=\linewidth]{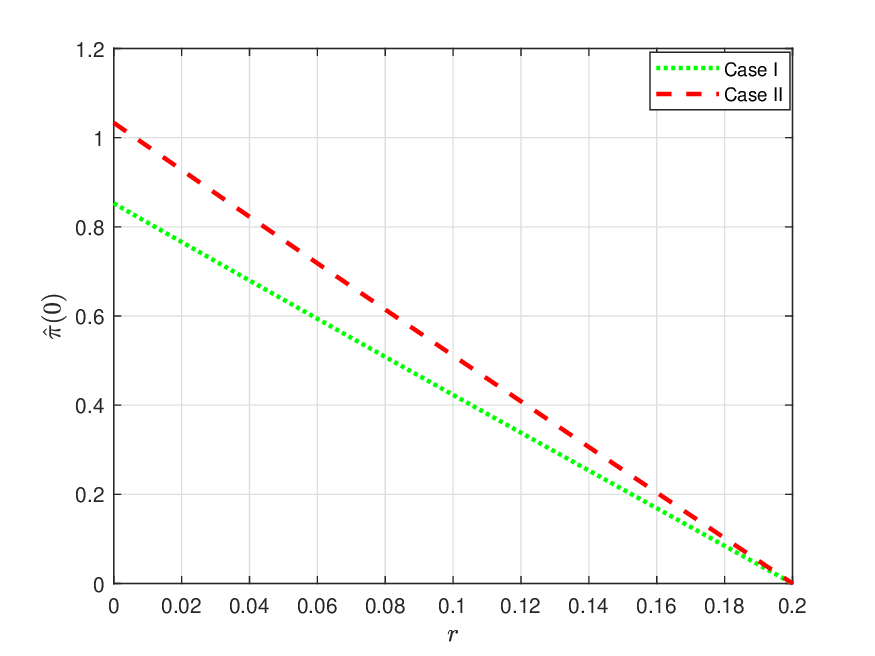}
      \caption*{(d) The impact of $r$ on $\hat{\pi}(0)$ }
    \end{minipage}
     \begin{minipage}{0.32\textwidth}
     \centering
      \includegraphics[width=\linewidth]{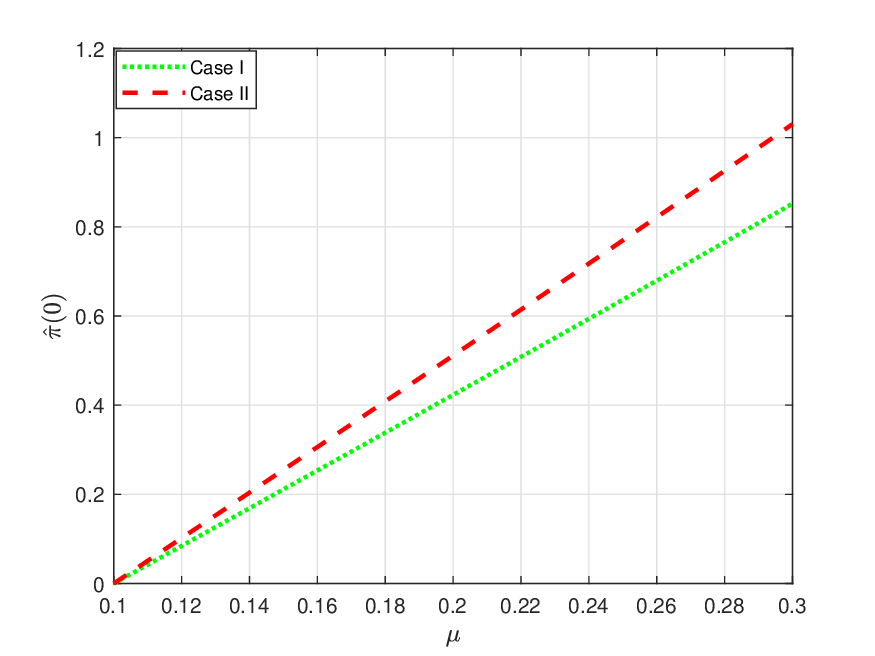}
      \caption*{(e) The impact of $\mu$ on $\hat{\pi}(0)$ }
    \end{minipage}
    \begin{minipage}{0.32\textwidth}
    \centering
      \includegraphics[width=\linewidth]{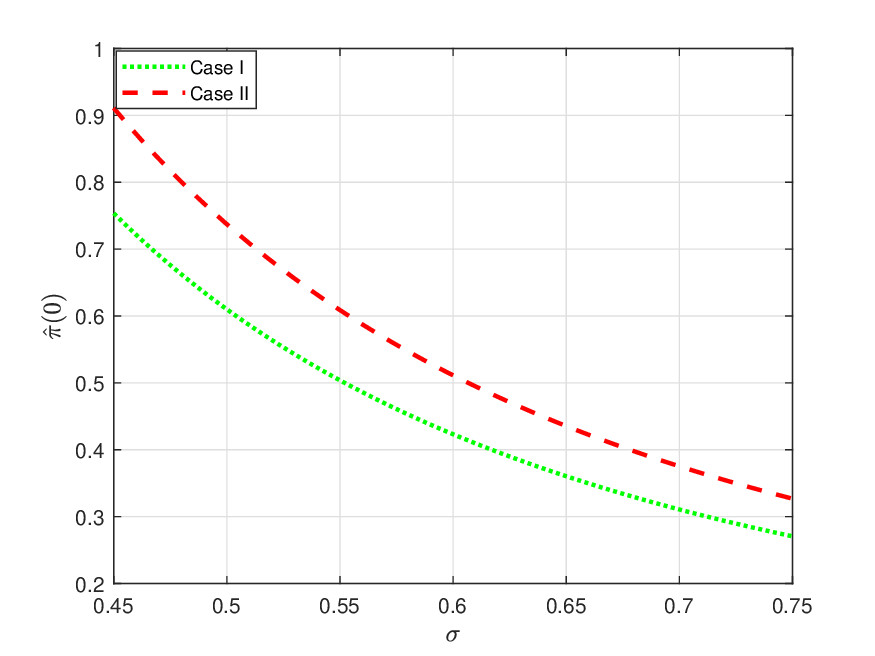}
      \caption*{(f) The impact of $\sigma$ on $\hat{\pi}(0)$}
    \end{minipage}
    \caption{The impacts of $r$, $\mu$ and $\sigma$ on the equilibrium reinsurance and investment strategy.}
\label{fig6}
\end{figure}

Figure \ref{fig6} characterizes how financial market parameters $r$, $\mu$ and $\sigma$ affect $(\hat{q}(0), \hat{\pi}(0))$. It follows from subgraph \ref{fig6} $(a)$ that an increase in $r$ leads to a reduction in $\hat{q}(0)$, suggesting that the insurer needs to secure more reinsurance when $r$ rises. Furthermore, subgraph \ref{fig6} $(d)$ demonstrates that $r$ has a similar influence pattern on $\hat{\pi}(0)$. When the risk-free interest rate $r$ increases, returns from risk-free asset investments increase, thereby attracting the insurer to increase the investment proportion in such a asset. Correspondingly, the proportion directed toward the risky asset will decrease. In addition, subgraphs \ref{fig6} $(b)$ and $(e)$ reveal that both $\hat{q}(0)$ and $\hat{\pi}(0)$ rise as $\mu$ increases. As $\mu$ signifies the expected return of the risky asset, an upward movement in $\mu$ indicates higher expected payoffs. Therefore, the insurer is prepared to accept greater risks to expand the share of the risky asset in her/his investments, seeking to gain higher returns. In such a case where the insurer opts to take on more risks, it tends to reduce the reinsurance purchases. What is more, subgraphs \ref{fig6} $(c)$ and $(f)$ show that both $\hat{q}(0)$ and $\hat{\pi}(0)$ decrease monotonically with $\sigma$. A rise in the volatility $\sigma$ of the risky asset leads to significant price fluctuations and instability. To avoid risks, insurers will cut the proportion of her/his investments in such a asset. Then, given that the insurer wants to avoid risks, she/he will also increase the purchase of reinsurance.

By comparing Figures \ref{fig4}, \ref{fig5} and \ref{fig6}, we find that parameters $\mu_{1}$, $\mu_{2}$ and $\eta_{2}$ have more significant impacts on $\hat{q}(0)$, while their impacts on $\hat{\pi}(0)$ are relatively smaller. On the other hand, it can also be seen that parameters $r$, $\mu$ and $\sigma$ have more pronounced effects on $\hat{\pi}(0)$, whereas their effects on $\hat{q}(0)$ are comparatively minor.

\section{Conclusions}
This paper focused on the reinsurance and investment optimization issue with delay for an insurer adopting random risk aversion. Specifically, based on a diffusion approximation model for the insurer's surplus evolution, the insurer can transfer risks via proportional reinsurance, grow the business through new policy acquisitions, and engage in investments within the financial market captured by the CEV model. To begin with, we posed the delayed optimal reinsurance and investment problem with random risk aversion by expected certainty equivalent and then formulated the verification theorem to describe the equilibrium reinsurance and investment strategy and the associated equilibrium value function under the game theory framework. Next, we obtained (semi-)analytical equilibrium reinsurance and investment strategies and equilibrium value functions under the CEV model for the exponential utility and under the Black-Scholes model for both exponential and power utilities. Finally, we performed numerical simulations to examine the effects of delay parameters and insurance-financial market parameters on the equilibrium reinsurance and investment strategy. The primary findings are outlined below. (I) A higher expected risk aversion level for the insurer leads to a tendency to buy more reinsurance and assign a smaller investment proportion to the risky asset. (II) The same magnitude of delay parameter changes yields different effects on  equilibrium reinsurance and investment strategies for exponential utility versus power utility. To be specific, in the context of exponential utility, both $\hat{q}(0)$ and $\hat{\pi}(0)$ first decrease and then increase when $\alpha$ increases, and decrease when $\beta$ or $h$ increases. In the case of power utility, both $\hat{q}(0)$ and $\hat{\pi}(0)$ decrease when $\alpha$ increases, and increase when $\beta$ or $h$ increases. (III) Under the exponential utility, the proportion of investment in the risky asset will decrease with the increase of the elasticity parameter $\delta$ if the initial value of the stock price is greater than $1$. (IV) Under the power utility, parameters in insurance and financial markets demonstrate analogous influence patterns on the variations of equilibrium reinsurance and investment strategies. In particular, both $\hat{q}(0)$ and $\hat{\pi}(0)$ decrease when $\mu_{2}$, $r$ or $\sigma$ increases, and increase when $\mu_{1}$, $\eta_{2}$ or $\mu$ increases.

Several compelling research directions are worthy of additional study. The contract optimizing an insurer's objective function generally fails to simultaneously maximize the reinsurer's value functional. Hence, the bilateral optimization of reinsurance and investment strategies has evolved into a canonical application of non-cooperative game theory in actuarial science \cite{Bai2022}. Moreover, for the purpose of protecting insurers against highly risky market positions, implementing risk measures, particularly Value-at-Risk (VaR) and Conditional VaR (CVaR), to control market risk exposures is of assistance \cite{Bi2019}. Therefore, it would be meaningful and practical to consider the delayed optimal reinsurance and investment game with VaR (or CVaR) constraints under random risk aversion. We intend to address these issues in subsequent research.

\section*{Appendices}

\appendix
\renewcommand{\appendixname}{Appendix~\Alph{section}}

\section{Proof of Lemma \ref{lemma3.2}}
\begin{proof}
If $X^{u}(\cdot)\in\mathbb{S}^{2}_{\mathcal{F}}(0,T;\mathbb{R})$, then one has
$$
\mathbb{E}\left[\sup_{t\in[0,T]}|X^{u}(t)|\right]\leq\left(\mathbb{E}\left[\left(\sup_{t\in[0,T]}|X^{u}(t)|\right)^{2}\right]\right)^{\frac{1}{2}}
=\left(\mathbb{E}\left[\sup_{t\in[0,T]}|X^{u}(t)|^{2}\right]\right)^{\frac{1}{2}}<\infty,
$$
which implies that $X^{u}(\cdot)\in\mathbb{S}^{1}_{\mathcal{F}}(0,T;\mathbb{R})$. Since $\mathbb{E}\left[\int^{T}_{0}| X^{u}(t)|\mathrm{d}t\right]\leq T\mathbb{E}\left[\sup\limits_{t\in[0,T]}|X^{u}(t)|\right]<\infty$, one can derive that $X^{u}(\cdot)\in\mathbb{L}^{1}_{\mathcal{F}}(0,T;\mathbb{R})$. In addition, we have
\begin{align*}
\mathbb{E}\left[\sup_{t\in[0,T]}|M_{2}^{u}(t)|\right]=\mathbb{E}\left[\sup_{t\in[0,T]}|X^{u}(t-h)|\right]
\leq\mathbb{E}\left[\sup_{t\in[0,T]}|X^{u}(t)|\right]+x_{0}<\infty
\end{align*}
and
\begin{align*}
\mathbb{E}\left[\sup_{t\in[0,T]}|M_{1}^{u}(t)|\right]&=\mathbb{E}\left[\sup_{t\in[0,T]}\left|\int^{0}_{-h}e^{\alpha s}X^{u}(t+s)\mathrm{d}s\right|\right]
\leq K\mathbb{E}\left[\sup_{t\in[0,T]}\sup_{s\in[-h,0]}|X^{u}(t+s)|\right]\\
&\leq K\mathbb{E}\left[\sup_{t\in[0,T]}|X^{u}(t)|\right]+Kx_{0}<\infty,
\end{align*}
where $K$ is a positive constant that will be used thereafter when there is no ambiguity, which can take different values in different lines. Thus, $M_{1}^{u}(\cdot),M_{2}^{u}(\cdot)\in\mathbb{S}^{1}_{\mathcal{F}}(0,T;\mathbb{R})$ and then $M_{1}^{u}(\cdot),M_{2}^{u}(\cdot)\in\mathbb{L}^{1}_{\mathcal{F}}(0,T;\mathbb{R})$.
\end{proof}

\section{Proof of Lemma \ref{lemma3.1}}
\begin{proof}
The result can be obtained similarly to the proof of Lemma 2.1 in \cite{Elsanosi2000}, and so we omit details here.
\end{proof}

\section{Proof of Theorem \ref{theorem3.1}}
\begin{proof}

Our proof adapts the techniques developed in Theorem 15.1 of \cite{Bjork2021}, Theorem 3.5 of \cite{Desmettre2023} and Theorem 3.1 of \cite{Kang2026}. To maintain the paper's completeness, we present the argument in three sequential steps.

Step 1. Given the assumed condition $U, Y^{\gamma}, H\in C^{1,2,2,1}(\mathcal{D})$ for all $\gamma$, it follows that $\hat{u}$ is continuous, which consequently ensures the continuity of both the drift and diffusion coefficients in the SDE governing $X^{\hat{u}}$. Furthermore, given that $Y^{\gamma}$ meets the condition $(A1)$ for all $\gamma$, one can use the Feynman-Kac therorem \cite{Yong1999} and equations \eqref{eq15} and \eqref{eq21} to demonstrate that
\begin{equation}\label{eq24}
Y^{\gamma}(t,x,s_{1},m_{1})=\mathbb{E}_{t}\left[\varphi^{\gamma}\left(X^{\hat{u}}(T)+\beta M_{1}^{\hat{u}}(T)\right)\right]= y^{\hat{u},\gamma}(t,x,s_{1},m_{1},m_{2})\; \text{for all}\;\gamma.
\end{equation}

Step 2. The substitution of $u=\hat{u}$ in the pseudo HJB equation \eqref{eq20} leads to
\begin{align*}
\mathcal{A}^{\hat{u}}U(t,x,s_{1},m_{1})-\mathcal{A}^{\hat{u}}H(t,x,s_{1},m_{1})
+\int\iota^{\gamma}(Y^{\gamma}(t,x,s_{1},m_{1}))\mathcal{A}^{\hat{u}}Y^{\gamma}(t,x,s_{1},m_{1})\mathrm{d}\Gamma(\gamma)=0.
\end{align*}
Plugging in the equation \eqref{eq21} then gives
\begin{align}\label{eq25}
\mathcal{A}^{\hat{u}}U(t,x,s_{1},m_{1})-\mathcal{A}^{\hat{u}}H(t,x,s_{1},m_{1})=0.
\end{align}
When $U\in C^{1,2,2,1}(\mathcal{D})$ and $U(t,X^{u}(t),S_{1}(t),M_{1}^{u}(t))$ meets the condition $(A1)$, one applies Lemma \ref{lemma3.1} to get
\begin{align}\label{eq26}
\mathbb{E}_{t}\left[U(T,X^{\hat{u}}(T),S_{1}(T),M_{1}^{\hat{u}}(T))\right]=&U(t,x,s_{1},m_{1})+\mathbb{E}_{t}\left[\int^{T}_{t}\bigg(U_{s}(s,X^{\hat{u}}(s),S_{1}(s),M_{1}^{\hat{u}}(s))
+U_{x}(s,X^{\hat{u}}(s),S_{1}(s),M_{1}^{\hat{u}}(s))\right.\nonumber\\
&\times \left[X^{\hat{u}}(s)\left(A+\hat{\pi}(s)(\mu-r)\right)+BM_{1}^{\hat{u}}(s)+CM_{2}^{\hat{u}}(s)+a\eta+a\eta_{2}\hat{q}(s)\right]\nonumber\\
&+0.5\left(b^{2}\hat{q}^{2}(s)+\hat{\pi}^{2}(s)(X^{\hat{u}}(s))^{2}\sigma^{2}S^{2 \delta}_{1}(s)\right)U_{xx}(s,X^{\hat{u}}(s),S_{1}(s),M_{1}^{\hat{u}}(s))\nonumber\\
&+\mu S_{1}(s)U_{s_{1}}(s,X^{\hat{u}}(s),S_{1}(s),M_{1}^{\hat{u}}(s))+0.5\sigma^{2} S^{2\delta+2}_{1}(s)U_{s_{1}s_{1}}(s,X^{\hat{u}}(s),S_{1}(s),M_{1}^{\hat{u}}(s))\nonumber\\
&+\left(X^{\hat{u}}(s)-\alpha M_{1}^{\hat{u}}(s)-e^{-\alpha h }M_{2}^{\hat{u}}(s)\right)U_{m_{1}}(s,X^{\hat{u}}(s),S_{1}(s),M_{1}^{\hat{u}}(s))\nonumber\\
&+\hat{\pi}(s)X^{\hat{u}}(s)\sigma^{2}S^{2 \delta+1}_{1}(s)U_{xs_{1}}(s,X^{\hat{u}}(s),S_{1}(s),M_{1}^{\hat{u}}(s))\bigg)\mathrm{d}s\bigg].
\end{align}
Similarly, for $H\in C^{1,2,2,1}(\mathcal{D})$ with $H(t,X^{u}(t),S_{1}(t),M_{1}^{u}(t))$ satisfying the condition $(A1)$, we derive
\begin{align}\label{eq27}
\mathbb{E}_{t}\left[H(T,X^{\hat{u}}(T),S_{1}(T),M_{1}^{\hat{u}}(T))\right]=&H(t,x,s_{1},m_{1})+\mathbb{E}_{t}\left[\int^{T}_{t}\bigg(H_{s}(s,X^{\hat{u}}(s),S_{1}(s),M_{1}^{\hat{u}}(s))
+H_{x}(s,X^{\hat{u}}(s),S_{1}(s),M_{1}^{\hat{u}}(s))\right.\nonumber\\
&\times \left[X^{\hat{u}}(s)\left(A+\hat{\pi}(s)(\mu-r)\right)+BM_{1}^{\hat{u}}(s)+CM_{2}^{\hat{u}}(s)+a\eta+a\eta_{2}\hat{q}(s)\right]\nonumber\\
&+0.5\left(b^{2}\hat{q}^{2}(s)+\hat{\pi}^{2}(s)(X^{\hat{u}}(s))^{2}\sigma^{2}S^{2 \delta}_{1}(s)\right)H_{xx}(s,X^{\hat{u}}(s),S_{1}(s),M_{1}^{\hat{u}}(s))\nonumber\\
&+\mu S_{1}(s)H_{s_{1}}(s,X^{\hat{u}}(s),S_{1}(s),M_{1}^{\hat{u}}(s))+0.5\sigma^{2} S^{2\delta+2}_{1}(s)H_{s_{1}s_{1}}(s,X^{\hat{u}}(s),S_{1}(s),M_{1}^{\hat{u}}(s))\nonumber\\
&+\left(X^{\hat{u}}(s)-\alpha M_{1}^{\hat{u}}(s)-e^{-\alpha h }M_{2}^{\hat{u}}(s)\right)H_{m_{1}}(s,X^{\hat{u}}(s),S_{1}(s),M_{1}^{\hat{u}}(s))\nonumber\\
&+\hat{\pi}(s)X^{\hat{u}}(s)\sigma^{2}S^{2 \delta+1}_{1}(s)H_{xs_{1}}(s,X^{\hat{u}}(s),S_{1}(s),M_{1}^{\hat{u}}(s))\bigg)\mathrm{d}s\bigg].
\end{align}
We emphasize that equations \eqref{eq26} and \eqref{eq27} represent general results valid for all $U, H\in C^{1,2,2,1}(\mathcal{D})$ satisfying the condition $(A1)$, independent of both equations \eqref{eq20} and \eqref{eq21}. The terminal condition $U(T,x,s_{1},m_{1})=H(T,x,s_{1},m_{1})$ together with equations \eqref{eq17}, \eqref{eq22} and \eqref{eq24}-\eqref{eq27} collectively yield
\begin{align}\label{eq28}
U(t,x,s_{1},m_{1})&=H(t,x,s_{1},m_{1})=\int(\varphi^{\gamma})^{-1}\left(Y^{\gamma}(t,x,s_{1},m_{1})\right)\mathrm{d}\Gamma(\gamma)\nonumber\\
&=\int(\varphi^{\gamma})^{-1}\left(y^{\hat{u},\gamma}(t,x,s_{1},m_{1},m_{2})\right)\mathrm{d}\Gamma(\gamma)
=J^{\hat{u}}(t,x,s_{1},m_{1},m_{2}).
\end{align}

Step 3. Analogous to Lemma 3.8 of \cite{Bjork2014}, one establishes that
\begin{align*}
J^{u_{\varepsilon}}(t,x,s_{1},m_{1},m_{2})=&\mathbb{E}_{t}\left[J^{u_{\varepsilon}}\left(t+\varepsilon,X^{u_{\varepsilon}}(t+\varepsilon),S_{1}(t+\varepsilon),M_{1}^{u_{\varepsilon}}(t+\varepsilon),M_{2}^{u_{\varepsilon}}(t+\varepsilon)\right)\right]\nonumber\\
&-\mathbb{E}_{t}\left[\int(\varphi^{\gamma})^{-1}\left(y^{u_{\varepsilon},\gamma}(t+\varepsilon,X^{u_{\varepsilon}}(t+\varepsilon),S_{1}(t+\varepsilon),M_{1}^{u_{\varepsilon}}(t+\varepsilon),M_{2}^{u_{\varepsilon}}(t+\varepsilon))\right)\mathrm{d}\Gamma(\gamma)\right]\nonumber\\
&+\int(\varphi^{\gamma})^{-1}\left(\mathbb{E}_{t}\left[y^{u_{\varepsilon},\gamma}(t+\varepsilon,X^{u_{\varepsilon}}(t+\varepsilon),S_{1}(t+\varepsilon),M_{1}^{u_{\varepsilon}}(t+\varepsilon),M_{2}^{u_{\varepsilon}}(t+\varepsilon))\right]\right)\mathrm{d}\Gamma(\gamma).
\end{align*}
Given that $u_{\varepsilon}=u$ on the interval $[t,t+\varepsilon)$, the continuity of $X^{u_{\varepsilon}}(s)$ for $s\leq t$ implies $X^{u_{\varepsilon}}(t+\varepsilon)=X^{u}(t+\varepsilon)$, $M_{1}^{u_{\varepsilon}}(t+\varepsilon)=M_{1}^{u}(t+\varepsilon)$ and $M_{2}^{u_{\varepsilon}}(t+\varepsilon)=M_{2}^{u}(t+\varepsilon)$. Furthermore, given the equality $u_{\varepsilon}=\hat{u}$ on $[t+\varepsilon, T]$, the equation \eqref{eq28} consequently implies
$$
J^{u_{\varepsilon}}\left(t+\varepsilon,X^{u_{\varepsilon}}(t+\varepsilon),S_{1}(t+\varepsilon),M_{1}^{u_{\varepsilon}}(t+\varepsilon),M_{2}^{u_{\varepsilon}}(t+\varepsilon)\right)
=U\left(t+\varepsilon,X^{u}(t+\varepsilon),S_{1}(t+\varepsilon),M_{1}^{u}(t+\varepsilon)\right).
$$
Moreover, one has
$$
y^{u_{\varepsilon},\gamma}\left(t+\varepsilon,X^{u_{\varepsilon}}(t+\varepsilon),S_{1}(t+\varepsilon),M_{1}^{u_{\varepsilon}}(t+\varepsilon),M_{2}^{u_{\varepsilon}}(t+\varepsilon)\right)
=y^{\hat{u},\gamma}\left(t+\varepsilon,X^{u}(t+\varepsilon),S_{1}(t+\varepsilon),M_{1}^{u}(t+\varepsilon),M_{2}^{u}(t+\varepsilon)\right).
$$
Consequently, we obtain the following result
\begin{align}\label{eq29}
J^{u_{\varepsilon}}(t,x,s_{1},m_{1},m_{2})=&\mathbb{E}_{t}\left[U(t+\varepsilon,X^{u}(t+\varepsilon),S_{1}(t+\varepsilon),M_{1}^{u}(t+\varepsilon))\right]\nonumber\\
&-\mathbb{E}_{t}\left[\int(\varphi^{\gamma})^{-1}\left(y^{\hat{u},\gamma}(t+\varepsilon,X^{u}(t+\varepsilon),S_{1}(t+\varepsilon),M_{1}^{u}(t+\varepsilon),M_{2}^{u}(t+\varepsilon))\right)\mathrm{d}\Gamma(\gamma)\right]\nonumber\\
&+\int(\varphi^{\gamma})^{-1}\left(\mathbb{E}_{t}\left[y^{\hat{u},\gamma}(t+\varepsilon,X^{u}(t+\varepsilon),S_{1}(t+\varepsilon),M_{1}^{u}(t+\varepsilon),M_{2}^{u}(t+\varepsilon))\right]\right)\mathrm{d}\Gamma(\gamma).
\end{align}
In addition, combining the pseudo HJB \eqref{eq20} and the equation \eqref{eq24} yields
\begin{align}\label{eq30}
\mathcal{A}^{u}U(t,x,s_{1},m_{1})-\mathcal{A}^{u}H(t,x,s_{1},m_{1})
+\int\iota^{\gamma}\left(y^{\hat{u},\gamma}(t,x,s_{1},m_{1},m_{2})\right)\mathcal{A}^{u}y^{\hat{u},\gamma}(t,x,s_{1},m_{1},m_{2})\mathrm{d}\Gamma(\gamma)\leq0.
\end{align}
Given the continuity of the admissible reinsurance and investment strategy $u(t,x,s_{1},m_{1},m_{2})$, we first apply Lemma \ref{lemma3.1} for arbitrary $\varepsilon>0$ and then take the limit $\varepsilon\rightarrow0$ to establish the following identities
\begin{align}\label{eq31}
&\mathbb{E}_{t}\left[U\left(t+\varepsilon,X^{u}(t+\varepsilon),S_{1}(t+\varepsilon),M_{1}^{u}(t+\varepsilon)\right)\right]-U(t,x,s_{1},m_{1})\nonumber\\
=&\varepsilon\left\{\left[(A+\pi(t,x,s_{1},m_{1},m_{2})(\mu-r))x+Bm_{1}+Cm_{2}+a\eta+a\eta_{2}q(t,x,s_{1},m_{1},m_{2})\right]
U_{x}(t,x,s_{1},m_{1})\right.\nonumber\\
&+U_{t}(t,x,s_{1},m_{1})+0.5\left(b^{2}q^{2}(t,x,s_{1},m_{1},m_{2})+\pi^{2}(t,x,s_{1},m_{1},m_{2})x^{2}\sigma^{2}s^{2\delta}_{1}\right)U_{xx}(t,x,s_{1},m_{1})\nonumber\\
&+\mu s_{1}U_{s_{1}}(t,x,s_{1},m_{1})+0.5\sigma^{2}s^{2\delta+2}_{1}U_{s_{1}s_{1}}(t,x,s_{1},m_{1})+\pi(t,x,s_{1},m_{1},m_{2})x \sigma^{2}s^{2\delta+1}_{1}\nonumber\\
&\left.\times U_{xs_{1}}(t,x,s_{1},m_{1})+\left(x-\alpha m_{1}-e^{-\alpha h}m_{2}\right)U_{m_{1}}(t,x,s_{1},m_{1})\right\}+o(\varepsilon)
\end{align}
and
\begin{align}\label{eq32}
&\mathbb{E}_{t}\left[H\left(t+\varepsilon,X^{u}(t+\varepsilon),S_{1}(t+\varepsilon),M_{1}^{u}(t+\varepsilon)\right)\right]-H(t,x,s_{1},m_{1})\nonumber\\
=&\varepsilon\left\{\left[(A+\pi(t,x,s_{1},m_{1},m_{2})(\mu-r))x+Bm_{1}+Cm_{2}+a\eta+a\eta_{2}q(t,x,s_{1},m_{1},m_{2})\right]
H_{x}(t,x,s_{1},m_{1})\right.\nonumber\\
&+H_{t}(t,x,s_{1},m_{1})+0.5\left(b^{2}q^{2}(t,x,s_{1},m_{1},m_{2})+\pi^{2}(t,x,s_{1},m_{1},m_{2})x^{2}\sigma^{2}s^{2\delta}_{1}\right)H_{xx}(t,x,s_{1},m_{1})\nonumber\\
&+\mu s_{1}H_{s_{1}}(t,x,s_{1},m_{1})+0.5\sigma^{2}s^{2\delta+2}_{1}H_{s_{1}s_{1}}(t,x,s_{1},m_{1})+\pi(t,x,s_{1},m_{1},m_{2})x \sigma^{2}s^{2\delta+1}_{1}\nonumber\\
&\left.\times H_{xs_{1}}(t,x,s_{1},m_{1})+\left(x-\alpha m_{1}-e^{-\alpha h}m_{2}\right)H_{m_{1}}(t,x,s_{1},m_{1})\right\}+o(\varepsilon)
\end{align}
as well as
\begin{align}\label{eq33}
&\mathbb{E}_{t}\left[y^{\hat{u},\gamma}\left(t+\varepsilon,X^{u}(t+\varepsilon),S_{1}(t+\varepsilon),M_{1}^{u}(t+\varepsilon),M_{2}^{u}(t+\varepsilon)\right)\right]
-y^{\hat{u},\gamma}(t,x,s_{1},m_{1},m_{2})\nonumber\\
=&\varepsilon\left\{\left[(A+\pi(t,x,s_{1},m_{1},m_{2})(\mu-r))x+Bm_{1}+Cm_{2}+a\eta+a\eta_{2}q(t,x,s_{1},m_{1},m_{2})\right]
y^{\hat{u},\gamma}_{x}(t,x,s_{1},m_{1},m_{2})\right.\nonumber\\
&+y^{\hat{u},\gamma}_{t}(t,x,s_{1},m_{1},m_{2})+0.5\left(b^{2}q^{2}(t,x,s_{1},m_{1},m_{2})+\pi^{2}(t,x,s_{1},m_{1},m_{2})x^{2}\sigma^{2}s^{2\delta}_{1}\right)
y^{\hat{u},\gamma}_{xx}(t,x,s_{1},m_{1},m_{2})\nonumber\\
&+\mu s_{1}y^{\hat{u},\gamma}_{s_{1}}(t,x,s_{1},m_{1},m_{2})+0.5\sigma^{2}s^{2\delta+2}_{1}y^{\hat{u},\gamma}_{s_{1}s_{1}}(t,x,s_{1},m_{1},m_{2})+\pi(t,x,s_{1},m_{1},m_{2})x \sigma^{2}s^{2\delta+1}_{1}\nonumber\\
&\left.\times y^{\hat{u},\gamma}_{xs_{1}}(t,x,s_{1},m_{1},m_{2})+\left(x-\alpha m_{1}-e^{-\alpha h}m_{2}\right)y^{\hat{u},\gamma}_{m_{1}}(t,x,s_{1},m_{1},m_{2})\right\}+o(\varepsilon),
\end{align}
where $o(\varepsilon)$ represents a higher-order infinitesimal and we use the fact derived from the equation \eqref{eq24} that $y^{\hat{u},\gamma}(t,x,s_{1},m_{1},m_{2})$ is independent of $m_{2}$. Then, from the equation \eqref{eq33}, one has
\begin{align}\label{eq34}
&\int(\varphi^{\gamma})^{-1}\left(\mathbb{E}_{t}\left[y^{\hat{u},\gamma}\left(t+\varepsilon,X^{u}(t+\varepsilon),S_{1}(t+\varepsilon),M_{1}^{u}(t+\varepsilon),M_{2}^{u}(t+\varepsilon)\right)\right]\right)\mathrm{d}\Gamma(\gamma)\nonumber\\
&\quad-\int(\varphi^{\gamma})^{-1}\left(y^{\hat{u},\gamma}(t,x,s_{1},m_{1},m_{2})\right)\mathrm{d}\Gamma(\gamma)
\nonumber\\
=&\varepsilon\left\{\int\iota^{\gamma}\left(y^{\hat{u},\gamma}(t,x,s_{1},m_{1},m_{2})\right)\bigg(y^{\hat{u},\gamma}_{t}(t,x,s_{1},m_{1},m_{2})+\big[(A+\pi(t,x,s_{1},m_{1},m_{2})(\mu-r))x\right.\nonumber\\
&+Bm_{1}+Cm_{2}+a\eta+a\eta_{2}q(t,x,s_{1},m_{1},m_{2})\big]
y^{\hat{u},\gamma}_{x}(t,x,s_{1},m_{1},m_{2})\nonumber\\
&+0.5\left(b^{2}q^{2}(t,x,s_{1},m_{1},m_{2})+\pi^{2}(t,x,s_{1},m_{1},m_{2})x^{2}\sigma^{2}s^{2\delta}_{1}\right)y^{\hat{u},\gamma}_{xx}(t,x,s_{1},m_{1},m_{2})\nonumber\\
&+\mu s_{1}y^{\hat{u},\gamma}_{s_{1}}(t,x,s_{1},m_{1},m_{2})+0.5\sigma^{2}s^{2\delta+2}_{1}y^{\hat{u},\gamma}_{s_{1}s_{1}}(t,x,s_{1},m_{1},m_{2})+\pi(t,x,s_{1},m_{1},m_{2})x \sigma^{2}s^{2\delta+1}_{1}\nonumber\\
&\times y^{\hat{u},\gamma}_{xs_{1}}(t,x,s_{1},m_{1},m_{2})+\left(x-\alpha m_{1}-e^{-\alpha h}m_{2}\right)y^{\hat{u},\gamma}_{m_{1}}(t,x,s_{1},m_{1},m_{2})\bigg)\mathrm{d}\Gamma(\gamma)\bigg\}+o(\varepsilon).
\end{align}
Combining equations \eqref{eq31}, \eqref{eq32} and \eqref{eq34} with the equation \eqref{eq30} yields
\begin{align*}
&\mathbb{E}_{t}\left[U(t+\varepsilon,X^{u}(t+\varepsilon),S_{1}(t+\varepsilon),M_{1}^{u}(t+\varepsilon))\right]-U(t,x,s_{1},m_{1})-\mathbb{E}_{t}\left[H(t+\varepsilon,X^{u}(t+\varepsilon),S_{1}(t+\varepsilon),M_{1}^{u}(t+\varepsilon))\right]\\
&\quad+H(t,x,s_{1},m_{1})+\int(\varphi^{\gamma})^{-1}\left(\mathbb{E}_{t}\left[y^{\hat{u},\gamma}(t+\varepsilon,X^{u}(t+\varepsilon),S_{1}(t+\varepsilon),M_{1}^{u}(t+\varepsilon),M_{2}^{u}(t+\varepsilon))\right]\right)\mathrm{d}\Gamma(\gamma)\\
&\quad-\int(\varphi^{\gamma})^{-1}\left(y^{\hat{u},\gamma}(t,x,s_{1},m_{1},m_{1})\right)\mathrm{d}\Gamma(\gamma)\leq  o(\varepsilon).
\end{align*}
Consequently, equations \eqref{eq22}, \eqref{eq28} and \eqref{eq29} collectively imply that
\begin{align*}
U(t,x,s_{1},m_{1})\geq&\mathbb{E}_{t}\left[U(t+\varepsilon,X^{u}(t+\varepsilon),S_{1}(t+\varepsilon),M_{1}^{u}(t+\varepsilon))\right]-\mathbb{E}_{t}\left[H(t+\varepsilon,X^{u}(t+\varepsilon),S_{1}(t+\varepsilon),M_{1}^{u}(t+\varepsilon))\right]\\
&+\int(\varphi^{\gamma})^{-1}\left(\mathbb{E}_{t}\left[y^{\hat{u},\gamma}(t+\varepsilon,X^{u}(t+\varepsilon),S_{1}(t+\varepsilon),M_{1}^{u}(t+\varepsilon),M_{2}^{u}(t+\varepsilon))\right]\right)\mathrm{d}\Gamma(\gamma)
+o(\varepsilon)\\
=&J^{u_{\varepsilon}}(t,x,s_{1},m_{1},m_{2})+o(\varepsilon)
\end{align*}
and so $J^{\hat{u}}(t,x,s_{1},m_{1},m_{2})\geq J^{u_{\varepsilon}}(t,x,s_{1},m_{1},m_{2})+o(\varepsilon)$.
Thus,
$
\underset{\varepsilon\rightarrow0}\liminf\frac{J^{\hat{u}}(t,x,s_{1},m_{1},m_{2})-J^{u_{\varepsilon}}(t,x,s_{1},m_{1},m_{2})}{\varepsilon}\geq0,
$
which establishes $\hat{u}$ as an equilibrium reinsurance and investment strategy. Therefore, the identity $V(t,x,s_{1},m_{1},m_{2})=U(t,x,s_{1},m_{1})$ is established, thereby completing the proof.
\end{proof}

\section{Proof of Theorem \ref{theorem4.4}}
\begin{proof}
We aim to confirm that the candidate equilibrium reinsurance and investment strategy \eqref{eq80} is admissible and that the assumptions of Theorem \ref{theorem3.1} are met.

Setting $\widetilde{Z}(t)=\frac{1}{4\delta^{2}}S^{-2\delta}_{1}(t)$, we derive that $\mathrm{d}\widetilde{Z}(t)=\widetilde{\kappa}(\widetilde{\theta}-\widetilde{Z}(t))\mathrm{d}t+\widetilde{\sigma}\sqrt{\widetilde{Z}(t)}\mathrm{d}W_{2}(t)$, where $\widetilde{\kappa}=2\mu \delta$, $\widetilde{\theta}=\frac{2\delta+1}{8\mu \delta^{2}}\sigma^{2}$ and $\widetilde{\sigma}=-\sigma$. It is easy to see that $2\widetilde{\kappa}\widetilde{\theta}>\widetilde{\sigma}^{2}$, which implies that $\widetilde{Z}(t)$ is almost surely non-negative. Moreover, by the proof process of Lemma 1 in \cite{Li2012}, one has $\mathbb{E}_{t}\big[\sup\limits_{s\in[t,T]}|\widetilde{Z}(s)|^{p}\big]<\infty$ for $p\in[1,\infty)$.

Now, when the expression \eqref{eq80} is substituted into the SDDE \eqref{eq11}, it leads to
\begin{equation}\label{eq83}
   \left\{ \begin{aligned}
   \mathrm{d}X^{\hat{u}}(t)=&\left[AX^{\hat{u}}(t)+\frac{e^{-(A+\beta)(T-t)}}{\sum\limits_{i=1}^{n}\gamma_{i}p_{i}}\left(\left(\frac{\mu-r}{\sigma^{2}}-2\delta\sum\limits_{i=1}^{n}\widehat{g}_{6}^{\gamma_{i}}(t)p_{i}\right)
   \frac{\mu-r}{S^{2\delta}_{1}(t)}+\frac{a^{2} \eta^{2}_{2}}{b^{2}}\right)+BM_{1}^{\hat{u}}(t)+CM_{2}^{\hat{u}}(t)\right.\\
  &+a\eta\bigg]\mathrm{d}t+\frac{e^{-(A+\beta)(T-t)}}{\sum\limits_{i=1}^{n}\gamma_{i}p_{i}}\left[\frac{a \eta_{2}}{b}\mathrm{d}W_{1}(t)+\left(\frac{\mu-r}{\sigma^{2}}-2\delta\sum\limits_{i=1}^{n}\widehat{g}_{6}^{\gamma_{i}}(t)p_{i}\right)
   \frac{\sigma}{S^{\delta}_{1}(t)}\mathrm{d}W_{2}(t)\right],\; t\in[\tau,T],\\
   X^{\hat{u}}(t)=&\psi(t-\tau), \;t\in[\tau-h,\tau].
  \end{aligned}\right.
\end{equation}
Since $S^{-2\delta}_{1}(\cdot)\in\mathbb{S}^{1}_{\mathcal{F}}(0,T;\mathbb{R})$, it follows from Proposition 2.1 in \cite{Meng2025} that there exists a unique strong solution $X^{\hat{u}}(\cdot)\in\mathbb{S}^{p}_{\mathcal{F}}(0,T;\mathbb{R})$ with $p\geq2$ for the above SDDE. This result confirms that the condition (i) in Definition \ref{definition3.1} holds true. Furthermore, taking into account that $a=\lambda_{1}\mu_{1}$ and $\eta_{2}$ are positive constants and $\gamma>0$, we derive from the equation \eqref{eq80} the results that $\hat{q}(t)>0$ and $\{(\hat{q}(t),\hat{\pi}(t))\}_{t\in[0,T]}$ is $\mathcal{F}_{t}$-progressively measurable and continuous. By using the expression \eqref{eq80} and the fact that $S^{-2\delta}_{1}(\cdot)\in\mathbb{S}^{1}_{\mathcal{F}}(0,T;\mathbb{R})$, we have
\begin{align*}
\mathbb{E}\left[\int^{T}_{0}\left(\hat{q}^{2}(t)+\hat{\pi}^{2}(t)(X^{\hat{u}}(t))^{2}(S^{\delta}_{1}(t))^{2}\right)\mathrm{d}t\right]
\leq K\mathbb{E}\left[\sup_{t\in[0,T]}S^{-2\delta}_{1}(t)\right]+K
<\infty.
\end{align*}
Thus, the condition (ii) in Definition \ref{definition3.1} is satisfied. In addition, because $\widehat{g}_{1}^{\gamma_{i}}(t)$, $\widehat{g}_{2}^{\gamma_{i}}(t)$, $\widehat{g}_{5}^{\gamma_{i}}(t)$ and $\widehat{g}_{6}^{\gamma_{i}}(t)$ are deterministic and continuous on $[0,T]$, by the equation \eqref{eq24}, one has
\begin{align}\label{eq86}
&\int\left|(\varphi^{\gamma})^{-1}\left(\mathbb{E}_{t}\left[\varphi^{\gamma}(X^{\hat{u}}(T)+\beta M_{1}^{\hat{u}}(T))\right]\right)\right|\mathrm{d}\Gamma(\gamma)
=
\int\left|(\varphi^{\gamma})^{-1}\left(Y^{\gamma}(t,x,s_{1},m_{1})\right)\right|\mathrm{d}\Gamma(\gamma)\nonumber\\
=&\sum\limits_{i=1}^{n}\frac{1}{\gamma_{i}}\left|\widehat{g}_{1}^{\gamma_{i}}(t)(x+\beta m_{1})+\widehat{g}_{2}^{\gamma_{i}}(t)+\widehat{g}_{5}^{\gamma_{i}}(t)+\widehat{g}_{6}^{\gamma_{i}}(t)s^{-2\delta}_{1}\right|p_{i}
<\infty,
\end{align}
which shows that the condition (iii) in Definition \ref{definition3.1} is met.

Next, since $Y^{\gamma_{i}}(t,x,s_{1},m_{1})=-\frac{1}{\gamma_{i}}e^{\widehat{g}_{1}^{\gamma_{i}}(t)(x+\beta m_{1})+\widehat{g}_{2}^{\gamma_{i}}(t)+\widehat{g}_{5}^{\gamma_{i}}(t)+\widehat{g}_{6}^{\gamma_{i}}(t)s_{1}^{-2 \delta}}$, where $\gamma_{i} \in[\epsilon_{1},\infty)$, $\widehat{g}_{1}^{\gamma_{i}}(t)=-\gamma_{i} e^{(A+\beta)(T-t)}$, $\widehat{g}_{2}^{\gamma_{i}}(t)=\frac{\gamma_{i} a \eta}{A+\beta}\left[1-e^{(A+\beta)(T-t)}\right]+\frac{a^{2}\eta^{2}_{2}}{b^{2}}(0.5a^{2}_{3i}-a_{3i})(T-t)$, $\widehat{g}_{5}^{\gamma_{i}}(t)$ satisfies $\frac{\partial \widehat{g}_{5}^{\gamma_{i}}(t)}{\partial t}+\delta(2\delta+1)\sigma^{2}\widehat{g}_{6}^{\gamma_{i}}(t)=0$ and $\widehat{g}_{6}^{\gamma_{i}}(t)$ meets the equation \eqref{eq79}, we can obtain that $Y^{\gamma_{i}}\in C^{1,2,2,1}(\mathcal{D})$. By using the condition $(A2)$ and equations \eqref{eq9} and \eqref{eq83}, we have
\begin{align*}
   \mathrm{d}\left[X^{\hat{u}}(t)+\beta M_{1}^{\hat{u}}(t)\right]=&\left[(A+\beta)(X^{\hat{u}}(t)+\beta M_{1}^{\hat{u}}(t))+
   X^{\hat{u}}(t)\hat{\pi}(t)(\mu-r)+a\eta+a\eta_{2}\hat{q}(t)\right]\mathrm{d}t\\
   &+b\hat{q}(t)\mathrm{d}W_{1}(t)+\hat{\pi}(t)X^{\hat{u}}(t)\sigma S^{\delta}_{1}(t)\mathrm{d}W_{2}(t)
\end{align*}
for $t\in[\tau,T]$. Then, one has
\begin{align*}
\left|Y^{\gamma_{i}}(t,X^{\hat{u}}(t),S_{1}(t),M_{1}^{\hat{u}}(t))\right|^{4}=&\left|\frac{1}{\gamma_{i}^{4}}e^{4\widehat{g}_{1}^{\gamma_{i}}(t)\left(X^{\hat{u}}(t)+\beta M_{1}^{\hat{u}}(t)\right)+4\widehat{g}_{2}^{\gamma_{i}}(t)+4\widehat{g}_{5}^{\gamma_{i}}(t)+4\widehat{g}_{6}^{\gamma_{i}}(t)S^{-2\delta}_{1}(t)}\right|\nonumber\\
\leq&Ke^{4\widehat{g}_{1}^{\gamma_{i}}(t)\left(X^{\hat{u}}(t)+\beta M_{1}^{\hat{u}}(t)\right)}\nonumber\\
=&Ke^{4\widehat{g}_{1}^{\gamma_{i}}(t)\left((x_{0}+\beta m_{10})e^{(A+\beta)t}+\int^{t}_{0}e^{(A+\beta)(t-s)}\left[X^{\hat{u}}(s)\hat{\pi}(s)(\mu-r)+a\eta+a\eta_{2}\hat{q}(s)\right]\mathrm{d}s
   \right)}\nonumber\\
&\times e^{4\widehat{g}_{1}^{\gamma_{i}}(t)\int^{t}_{0}e^{(A+\beta)(t-s)}\left[b\hat{q}(s)\mathrm{d}W_{1}(s)+\hat{\pi}(s)X^{\hat{u}}(s)\sigma S^{\delta}_{1}(s)\mathrm{d}W_{2}(s)\right]}\nonumber\\
   \leq&Ke^{
   -4\gamma_{i}\int^{t}_{0}\left[(\mu-r)\widetilde{\pi}(s)S^{-2\delta}_{1}(s)\mathrm{d}s+b\widetilde{q}\mathrm{d}W_{1}(s)+\widetilde{\pi}(s)\sigma S^{-\delta}_{1}(s)\mathrm{d}W_{2}(s)\right]},
\end{align*}
where the first inequality holds because $\widehat{g}_{2}^{\gamma_{i}}(t)$, $\widehat{g}_{5}^{\gamma_{i}}(t)$ and $\widehat{g}_{6}^{\gamma_{i}}(t)\leq 0$ are bounded and $S^{-2\delta}_{1}(t)$ is almost surely non-negative, $m_{10}=\frac{x_{0}(1-e^{-\alpha h})}{\alpha}$ and
\begin{equation}\label{eq84}
  \widetilde{q}=\frac{a \eta_{2}}{b^{2} \sum\limits_{i=1}^{n}\gamma_{i}p_{i}},\;\;
  \widetilde{\pi}(s)=\frac{1}{\sum\limits_{i=1}^{n}\gamma_{i}p_{i}}(\frac{\mu-r}{\sigma^{2}}-2\delta\sum\limits_{i=1}^{n}\widehat{g}_{6}^{\gamma_{i}}(s)p_{i}).
\end{equation}

Letting $\widetilde{M}_{1}(t)=e^{-4\gamma_{i}\int^{t}_{0}b\widetilde{q}\mathrm{d}W_{1}(s)}$ and noting that $\widetilde{q}$ is a constant, we can deduce that
$$
\widetilde{M}_{1}(T)=\underbrace{e^{\int^{T}_{0}8\gamma_{i}^{2}b^{2}\widetilde{q}^{2}\mathrm{d}s}}_{constant}\times
\underbrace{e^{-\int^{T}_{0}8\gamma_{i}^{2}b^{2}\widetilde{q}^{2}\mathrm{d}s-4\gamma_{i}\int^{T}_{0}b\widetilde{q}\mathrm{d}W_{1}(s)}}_{martingale}.
$$
Therefore, $\mathbb{E}\left[\widetilde{M}_{1}(T)\right]<\infty$. Then, we seek to derive an estimate for
$e^{-4\gamma_{i}\int^{T}_{0}\left[(\mu-r)\widetilde{\pi}(s)S^{-2\delta}_{1}(s)\mathrm{d}s+\widetilde{\pi}(s)\sigma S^{-\delta}_{1}(s)\mathrm{d}W_{2}(s)\right]}$.
Since $\widetilde{Z}(t)=\frac{1}{4\delta^{2}}S^{-2\delta}_{1}(t)$, we note that
\begin{align*}
e^{-4\gamma_{i}\int^{t}_{0}\left[(\mu-r)\widetilde{\pi}(s)S^{-2\delta}_{1}(s)\mathrm{d}s+\widetilde{\pi}(s)\sigma S^{-\delta}_{1}(s)\mathrm{d}W_{2}(s)\right]}
=&\underbrace{e^{\int^{t}_{0}\left(-16\gamma_{i}(\mu-r)\delta^{2}\widetilde{\pi}(s)+64\gamma_{i}^{2}\delta^{2}\sigma^{2}\widetilde{\pi}^{2}(s)\right)\widetilde{Z}(s)\mathrm{d}s}}_{\widetilde{M}_{2}(t)}\\
&\times\underbrace{e^{-\int^{t}_{0}64\gamma_{i}^{2}\delta^{2}\sigma^{2}\widetilde{\pi}^{2}(s)\widetilde{Z}(s)\mathrm{d}s
  -\int^{t}_{0}8\gamma_{i}\delta\sigma\widetilde{\pi}(s)\sqrt{\widetilde{Z}(s)}\mathrm{d}W_{2}(s)}}_{\widetilde{M}_{3}(t)}.
\end{align*}
For the term $\widetilde{M}_{3}(t)$, because $\gamma_{i}\delta\sigma\widetilde{\pi}(s)$ is deterministic and bounded on $[0,T]$ and $\widetilde{Z}(s)$ satisfies the Cox-Ingersoll-Ross Model, one can use Lemma 4.3 in \cite{ZengX2013} to derive that for $t\in[0,T]$,
\begin{equation*}
  \mathbb{E}[(\widetilde{M}_{3}(t))^{2}]=\mathbb{E}\left[e^{-\int^{t}_{0}128\gamma_{i}^{2}\delta^{2}\sigma^{2}\widetilde{\pi}^{2}(s)\widetilde{Z}(s)\mathrm{d}s
  -\int^{t}_{0}16\gamma_{i}\delta\sigma\widetilde{\pi}(s)\sqrt{\widetilde{Z}(s)}\mathrm{d}W_{2}(s)}\right]<\infty.
\end{equation*}
For the term $\widetilde{M}_{2}(t)$, we know that
$
  \mathbb{E}[(\widetilde{M}_{2}(T))^{2}]=\mathbb{E}\left[e^{\int^{T}_{0}\left(-32\gamma_{i}(\mu-r)\delta^{2}\widetilde{\pi}(s)
  +128\gamma_{i}^{2}\delta^{2}\sigma^{2}\widetilde{\pi}^{2}(s)\right)\widetilde{Z}(s)\mathrm{d}s}\right]
$. It follows from Theorem 5.1 in \cite{ZengX2013} that if the condition $-32\gamma_{i}(\mu-r)\delta^{2}\widetilde{\pi}(s)
  +128\gamma_{i}^{2}\delta^{2}\sigma^{2}\widetilde{\pi}^{2}(s)\leq\frac{\kappa^{2}}{2\sigma^{2}}$ is satisfied for $s\in[0,T]$, then $\mathbb{E}[(\widetilde{M}_{2}(T))^{2}]<\infty$. Thus, based on the above analysis, by using Burkh\"{o}lder-Davis-Gundy inequality (see, for example, Theorem 1.1.6 in \cite{Pham2009}) and Cauchy-Schwarz inequality, one can obtain that
\begin{align}\label{eq85}
  \mathbb{E}\left[\sup_{t\in[0,T]}\left|Y^{\gamma_{i}}(t,X^{\hat{u}}(t),S_{1}(t),M_{1}^{\hat{u}}(t))\right|^{4}\right]&\leq K\mathbb{E}\left[\left|Y^{\gamma_{i}}(T,X^{\hat{u}}(T),S_{1}(T),M_{1}^{\hat{u}}(T))\right|^{4}\right]\nonumber\\
  &\leq Ke^{-4\gamma_{i}\int^{T}_{0}\left[(\mu-r)\widetilde{\pi}(s)S^{-2\delta}_{1}(s)\mathrm{d}s+b\widetilde{q}\mathrm{d}W_{1}(s)+\widetilde{\pi}(s)\sigma S^{-\delta}_{1}(s)\mathrm{d}W_{2}(s)\right]}\\
  &\leq K \mathbb{E}\left[\widetilde{M}_{2}(T)\widetilde{M}_{3}(T)\right] \leq K \left(\mathbb{E}\left[(\widetilde{M}_{2}(T))^{2}\right] \mathbb{E}\left[(\widetilde{M}_{3}(T))^{2}\right]\right)^{\frac{1}{2}}<\infty,\nonumber
\end{align}
where the independence of $W_{1}$ and $W_{2}$ is applied in the third inequality. Then, similar to the proof of Lemma \ref{lemma3.2}, one has $Y^{\gamma_{i}}(\cdot,X^{\hat{u}}(\cdot),S_{1}(\cdot),M_{1}^{\hat{u}}(\cdot))\in\mathbb{S}^{2}_{\mathcal{F}}(0,T;\mathbb{R})$.

Subsequently, the assumption that $Y^{\gamma_{i}}(\cdot,X^{\hat{u}}(\cdot),S_{1}(\cdot),M_{1}^{\hat{u}}(\cdot))$ satisfies the condition $(A1)$ will be verified. In view of the expression of $Y^{\gamma_{i}}$, we have
\begin{align*}
&\mathbb{E}\left[\int^{T}_{0}\left|Y^{\gamma_{i}}_{t}(t,X^{\hat{u}}(t),S_{1}(t),M_{1}^{\hat{u}}(t))\right|\mathrm{d}t\right]\\
=&\mathbb{E}\left[\int^{T}_{0}\left|Y^{\gamma_{i}}(t,X^{\hat{u}}(t),S_{1}(t),M_{1}^{\hat{u}}(t))\left[\frac{\partial \widehat{g}_{1}^{\gamma_{i}}(t)}{\partial t}(X^{\hat{u}}(t)+\beta M_{1}^{\hat{u}}(t))+\frac{\partial \widehat{g}_{2}^{\gamma_{i}}(t)}{\partial t}+\frac{\partial \widehat{g}_{5}^{\gamma_{i}}(t)}{\partial t}+\frac{\partial \widehat{g}_{6}^{\gamma_{i}}(t)}{\partial t}S^{-2\delta}_{1}(t)\right]\right|\mathrm{d}t\right]\\
\leq &K\mathbb{E}\left[\sup_{t\in[0,T]}\left|Y^{\gamma_{i}}(t,X^{\hat{u}}(t),S_{1}(t),M_{1}^{\hat{u}}(t))\right|\times \left(\sup_{t\in[0,T]}\left|X^{\hat{u}}(t)\right|+\sup_{t\in[0,T]}\left|M_{1}^{\hat{u}}(t)\right|+\sup_{t\in[0,T]}\left|S^{-2\delta}_{1}(t)\right|+K\right)\right],
\end{align*}
where we use the fact that $\frac{\partial \widehat{g}_{1}^{\gamma_{i}}(t)}{\partial t}$, $\frac{\partial \widehat{g}_{2}^{\gamma_{i}}(t)}{\partial t}$, $\frac{\partial \widehat{g}_{5}^{\gamma_{i}}(t)}{\partial t}$ and $\frac{\partial \widehat{g}_{6}^{\gamma_{i}}(t)}{\partial t}$ are deterministic and bounded in the inequality. According to $Y^{\gamma_{i}}(\cdot,X^{\hat{u}}(\cdot),S_{1}(\cdot),M_{1}^{\hat{u}}(\cdot))\in\mathbb{S}^{2}_{\mathcal{F}}(0,T;\mathbb{R})$ and $X^{\hat{u}}(\cdot)\in\mathbb{S}^{p}_{\mathcal{F}}(0,T;\mathbb{R})$ with $p\geq2$, we can deduce
\begin{align*}
&\mathbb{E}\left[\sup_{t\in[0,T]}\left|Y^{\gamma_{i}}(t,X^{\hat{u}}(t),S_{1}(t),M_{1}^{\hat{u}}(t))\right|\times \sup_{t\in[0,T]}\left|X^{\hat{u}}(t)\right|\right]\\
\leq&
\left(\mathbb{E}\left[\left(\sup_{t\in[0,T]}\left|Y^{\gamma_{i}}(t,X^{\hat{u}}(t),S_{1}(t),M_{1}^{\hat{u}}(t))\right|\right)^{2}\right]\times
\mathbb{E}\left[\left(\sup_{t\in[0,T]}\left|X^{\hat{u}}(t)\right|\right)^{2}\right]\right)^{\frac{1}{2}}\\
\leq&
\left(\mathbb{E}\left[\sup_{t\in[0,T]}\left|Y^{\gamma_{i}}(t,X^{\hat{u}}(t),S_{1}(t),M_{1}^{\hat{u}}(t))\right|^{2}\right]\times
\mathbb{E}\left[\sup_{t\in[0,T]}\left|X^{\hat{u}}(t)\right|^{2}\right]\right)^{\frac{1}{2}}
<\infty.
\end{align*}
Because
$
|M_{1}^{\hat{u}}(t)|=\left|\int^{0}_{-h}e^{\alpha s}X^{\hat{u}}(t+s)\mathrm{d}s\right|
\leq K\max\limits_{s\in[-h,0]}|X^{\hat{u}}(t+s)|
\leq K \|X^{\hat{u}}_{t}\|
$, one has $$
\mathbb{E}\left[\sup_{t\in[0,T]}\left|M_{1}^{\hat{u}}(t)\right|^{2}\right]\leq K\mathbb{E}\left[\sup_{t\in[0,T]}\|X^{\hat{u}}_{t}\|^{2}\right]
\leq K\mathbb{E}\left[\sup_{t\in[0,T]}\left|X^{\hat{u}}(t)\right|^{2}+x^{2}_{0}\right]
<\infty,
$$
which implies $M_{1}^{\hat{u}}(\cdot)\in\mathbb{S}^{2}_{\mathcal{F}}(0,T;\mathbb{R})$. Moreover, since $S^{-2\delta}_{1}(\cdot)\in\mathbb{S}^{2}_{\mathcal{F}}(0,T;\mathbb{R})$, one can similarly derive
\begin{align*}
&\mathbb{E}\left[\sup_{t\in[0,T]}\left|Y^{\gamma_{i}}(t,X^{\hat{u}}(t),S_{1}(t),M_{1}^{\hat{u}}(t))\right|\times \sup_{t\in[0,T]}\left|M_{1}^{\hat{u}}(t)\right|\right]<\infty,\\
&\mathbb{E}\left[\sup_{t\in[0,T]}\left|Y^{\gamma_{i}}(t,X^{\hat{u}}(t),S_{1}(t),M_{1}^{\hat{u}}(t))\right|\times \sup_{t\in[0,T]}\left|S^{-2\delta}_{1}(t)\right|\right]<\infty.
\end{align*}
Similar to the proof of Lemma \ref{lemma3.2}, one can also show $\mathbb{E}\left[\sup\limits_{t\in[0,T]}\left|Y^{\gamma_{i}}(t,X^{\hat{u}}(t),S_{1}(t),M_{1}^{\hat{u}}(t))\right|\right]
<\infty$. Thus, $\mathbb{E}\left[\int^{T}_{0}\left|Y^{\gamma_{i}}_{t}(t,X^{\hat{u}}(t),S_{1}(t),M_{1}^{\hat{u}}(t))\right|\mathrm{d}t\right]<\infty$, which means $Y^{\gamma_{i}}_{t}(t,X^{\hat{u}}(t),S_{1}(t),M_{1}^{\hat{u}}(t))\in\mathbb{L}^{1}_{\mathcal{F}}(0,T;\mathbb{R})$. Proceeding analogously, we can fully verify that
$\mathcal{A}^{u}Y^{\gamma_{i}}(t,X^{\hat{u}}(t),S_{1}(t),M_{1}^{\hat{u}}(t))\in\mathbb{L}^{1}_{\mathcal{F}}(0,T;\mathbb{R})$. In addition,
\begin{align*}
&\mathbb{E}\left[\int^{T}_{0}\left|\hat{\pi}(t)X^{\hat{u}}(t)\sigma S^{\delta}_{1}(t)Y^{\gamma_{i}}_{x}(t,X^{\hat{u}}(t),S_{1}(t),M_{1}^{\hat{u}}(t))\right|^{2}\mathrm{d}t\right]\\
\leq& K\mathbb{E}\left[\sup_{t\in[0,T]}\left|Y^{\gamma_{i}}(t,X^{\hat{u}}(t),S_{1}(t),M_{1}^{\hat{u}}(t))\right|^{2}\times \sup_{t\in[0,T]}\left|S^{-2\delta}_{1}(t)\right|\right]\\
\leq&K
\left(\mathbb{E}\left[\left(\sup_{t\in[0,T]}\left|Y^{\gamma_{i}}(t,X^{\hat{u}}(t),S_{1}(t),M_{1}^{\hat{u}}(t))\right|^{2}\right)^{2}\right]\times
\mathbb{E}\left[\left(\sup_{t\in[0,T]}\left|S^{-2\delta}_{1}(t)\right|\right)^{2}\right]\right)^{\frac{1}{2}}\\
\leq&K
\left(\mathbb{E}\left[\sup_{t\in[0,T]}\left|Y^{\gamma_{i}}(t,X^{\hat{u}}(t),S_{1}(t),M_{1}^{\hat{u}}(t))\right|^{4}\right]\times
\mathbb{E}\left[\sup_{t\in[0,T]}\left|S^{-2\delta}_{1}(t)\right|^{2}\right]\right)^{\frac{1}{2}}
<\infty,
\end{align*}
where the last inequality holds by virtue of the equation \eqref{eq85} and the fact that $S^{-2\delta}_{1}(\cdot)\in\mathbb{S}^{2}_{\mathcal{F}}(0,T;\mathbb{R})$. Thus, $\hat{\pi}(t)X^{\hat{u}}(t)\sigma S^{\delta}_{1}(t)Y^{\gamma_{i}}_{x}(t,X^{\hat{u}}(t),S_{1}(t),M_{1}^{\hat{u}}(t))\in\mathbb{L}^{2}_{\mathcal{F}}(0,T;\mathbb{R})$. Using a similar argument, we can demonstrate that $b\hat{q}(t)Y^{\gamma_{i}}_{x}(t,X^{\hat{u}}(t),S_{1}(t),M_{1}^{\hat{u}}(t)),\;\sigma S^{\delta+1}_{1}(t)Y^{\gamma_{i}}_{s_{1}}(t,X^{\hat{u}}(t),S_{1}(t),M_{1}^{\hat{u}}(t))\in\mathbb{L}^{2}_{\mathcal{F}}(0,T;\mathbb{R})$.

Finally, the assumption that $U(\cdot,X^{\hat{u}}(\cdot),S_{1}(\cdot),M_{1}^{\hat{u}}(\cdot))$ and $H(\cdot,X^{\hat{u}}(\cdot),S_{1}(\cdot),M_{1}^{\hat{u}}(\cdot))$ satisfy the condition $(A1)$ will be demonstrated. Due to the equations \eqref{eq24} and \eqref{eq86}, we can deduce that $U(t,x,s_{1},m_{1})=H(t,x,s_{1},m_{1})=-\sum\limits_{i=1}^{n}\frac{1}{\gamma_{i}}\left[\widehat{g}_{1}^{\gamma_{i}}(t)(x+\beta m_{1})+\widehat{g}_{2}^{\gamma_{i}}(t)+\widehat{g}_{5}^{\gamma_{i}}(t)+\widehat{g}_{6}^{\gamma_{i}}(t)s^{-2\delta}_{1}\right]p_{i}$, where $\widehat{g}_{1}^{\gamma_{i}}(t)$, $\widehat{g}_{2}^{\gamma_{i}}(t)$, $\widehat{g}_{5}^{\gamma_{i}}(t)$ and $\widehat{g}_{6}^{\gamma_{i}}(t)$ are deterministic and continuous functions on $[0,T]$, we know $U,\;H\in C^{1,2,2,1}(\mathcal{D})$. Furthermore, by using the fact $\frac{\partial \widehat{g}_{1}^{\gamma_{i}}(t)}{\partial t}$, $\frac{\partial \widehat{g}_{2}^{\gamma_{i}}(t)}{\partial t}$, $\frac{\partial \widehat{g}_{5}^{\gamma_{i}}(t)}{\partial t}$ and $\frac{\partial \widehat{g}_{6}^{\gamma_{i}}(t)}{\partial t}$ are deterministic and bounded, one can derive that
\begin{align*}
&\mathbb{E}\left[\int^{T}_{0}\left|U_{t}(t,X^{\hat{u}}(t),S_{1}(t),M_{1}^{\hat{u}}(t))\right|\mathrm{d}t\right]\\
=&\mathbb{E}\left[\int^{T}_{0}
\left|-\sum\limits_{i=1}^{n}\frac{1}{\gamma_{i}}\left[\frac{\partial \widehat{g}_{1}^{\gamma_{i}}(t)}{\partial t}(X^{\hat{u}}(t)+\beta M_{1}^{\hat{u}}(t))+\frac{\partial \widehat{g}_{2}^{\gamma_{i}}(t)}{\partial t}+\frac{\partial \widehat{g}_{5}^{\gamma_{i}}(t)}{\partial t}+\frac{\partial \widehat{g}_{6}^{\gamma_{i}}(t)}{\partial t}S^{-2\delta}_{1}(t)\right]\right|\mathrm{d}t\right]\\
\leq& K\mathbb{E}\left[\sup_{t\in[0,T]}\left|X^{\hat{u}}(t)\right|+\sup_{t\in[0,T]}\left|M_{1}^{\hat{u}}(t)\right|+\sup_{t\in[0,T]}\left|S^{-2\delta}_{1}(t)\right|\right]+K<\infty,
\end{align*}
which shows $U_{t}(t,X^{\hat{u}}(t),S_{1}(t),M_{1}^{\hat{u}}(t))\in\mathbb{L}^{1}_{\mathcal{F}}(0,T;\mathbb{R})$. Analogously, we are able to carry out a full verification that $\mathcal{A}^{u}U(t,X^{\hat{u}}(t),S_{1}(t),M_{1}^{\hat{u}}(t))\in\mathbb{L}^{1}_{\mathcal{F}}(0,T;\mathbb{R})$. Moreover, because $S^{-2\delta}_{1}(\cdot)\in\mathbb{S}^{1}_{\mathcal{F}}(0,T;\mathbb{R})$, one has
\begin{align*}
&\mathbb{E}\left[\int^{T}_{0}\left|\hat{\pi}(t)X^{\hat{u}}(t)\sigma S^{\delta}_{1}(t)U_{x}(t,X^{\hat{u}}(t),S_{1}(t),M_{1}^{\hat{u}}(t))\right|^{2}\mathrm{d}t\right]\\ \leq& K\mathbb{E}\left[\sup_{t\in[0,T]}\left|U_{x}(t,X^{\hat{u}}(t),S_{1}(t),M_{1}^{\hat{u}}(t))\right|^{2}\times \sup_{t\in[0,T]}\left|S^{-2\delta}_{1}(t)\right|\right]\\
\leq&K\mathbb{E}\left[\sup_{t\in[0,T]}\left|\sum\limits_{i=1}^{n}\frac{1}{\gamma_{i}}\widehat{g}_{1}^{\gamma_{i}}(t)p_{i}\right|^{2}\times \sup_{t\in[0,T]}\left|S^{-2\delta}_{1}(t)\right|\right]
\leq K\mathbb{E}\left[ \sup_{t\in[0,T]}\left|S^{-2\delta}_{1}(t)\right|\right]
<\infty,
\end{align*}
which indicates that $\hat{\pi}(t)X^{\hat{u}}(t)\sigma S^{\delta}_{1}(t)U_{x}(t,X^{\hat{u}}(t),S_{1}(t),M_{1}^{\hat{u}}(t))\in\mathbb{L}^{2}_{\mathcal{F}}(0,T;\mathbb{R})$. Likewise, one can verify that $b\hat{q}(t)U_{x}(t,X^{\hat{u}}(t),S_{1}(t),M_{1}^{\hat{u}}(t)),\;\sigma S^{\delta+1}_{1}(t)U_{s_{1}}(t,X^{\hat{u}}(t),S_{1}(t),M_{1}^{\hat{u}}(t))\in\mathbb{L}^{2}_{\mathcal{F}}(0,T;\mathbb{R})$. Hence, $U(\cdot,X^{\hat{u}}(\cdot),S_{1}(\cdot),M_{1}^{\hat{u}}(\cdot))$ satisfies the condition $(A1)$. Since $U=H$, $H(\cdot,X^{\hat{u}}(\cdot),S_{1}(\cdot),M_{1}^{\hat{u}}(\cdot))$ also meets the condition $(A1)$.
\end{proof}

\section{Proof of Proposition \ref{proposition4.1}}
\begin{proof}
Since $A=r-B-C$, $C=\beta e^{-\alpha h}$ and $B e^{-\alpha h}=(\alpha+A+\beta)C$, we can rewrite $\hat{q}(t)$ as
$
\hat{q}(t)=\frac{a \eta_{2}}{b^{2}\mathbb{E}[\gamma]}\exp\left\{\left[-r+\frac{\beta}{1+\beta}\left(r+\alpha+e^{-\alpha h}-1\right)\right](T-t)\right\}.
$
Therefore, we have $\frac{\partial \hat{q}(t)}{\partial \alpha}=\frac{\beta\hat{q}(t)}{1+\beta}\left(1-he^{-\alpha h}\right)(T-t)$, $\frac{\partial \hat{q}(t)}{\partial \beta}=\frac{\hat{q}(t)}{(1+\beta)^{2}}\left(r+\alpha+e^{-\alpha h}-1\right)(T-t)$ and $\frac{\partial \hat{q}(t)}{\partial h}=-\frac{\alpha\beta\hat{q}(t)}{1+\beta}e^{-\alpha h}(T-t)$. Because $\hat{q}(t)>0$, by the assumption $r+\alpha<1$, we can get equations \eqref{eq70} and \eqref{eq71} and $\frac{\partial \hat{q}(t)}{\partial h}<0$. Similarly, we can obtain the sensitivity analysis on the effects of $\alpha$, $\beta$ and $h$ on $\hat{\pi}(t)x$.
\end{proof}

\section{Proof of Theorem \ref{theorem4.1}}
\begin{proof}
We need to verify that the candidate equilibrium reinsurance and investment strategy \eqref{eq47} is admissible and the assumptions in Theorem \ref{theorem3.1} are satisfied.

Substituting the expression \eqref{eq47} into the SDDE \eqref{eq11} arrives at
\begin{equation}\label{eq86}
   \left\{ \begin{aligned}
   \mathrm{d}X^{\hat{u}}(t)=&\left[AX^{\hat{u}}(t)+\frac{(\mu-r)^{2}}{\sigma^{2}\mathbb{E}[\gamma]}e^{-(A+\beta)(T-t)}+BM_{1}^{\hat{u}}(t)+CM_{2}^{\hat{u}}(t)+a\eta+\frac{a^{2} \eta^{2}_{2}}{b^{2}\mathbb{E}[\gamma]}e^{-(A+\beta)(T-t)}\right]\mathrm{d}t\\
  &+\frac{a \eta_{2}}{b\mathbb{E}[\gamma]}e^{-(A+\beta)(T-t)}\mathrm{d}W_{1}(t)+\frac{\mu-r}{\sigma\mathbb{E}[\gamma]}e^{-(A+\beta)(T-t)}\mathrm{d}W_{2}(t),\; t\in[\tau,T],\\
   X^{\hat{u}}(t)=&\psi(t-\tau), \;t\in[\tau-h,\tau].
  \end{aligned}\right.
\end{equation}
Due to Proposition 2.1 in \cite{Meng2025}, the above SDDE possesses a unique strong solution $X^{\hat{u}}(\cdot)\in\mathbb{S}^{p}_{\mathcal{F}}(0,T;\mathbb{R})$ for $p\geq2$, which shows that the condition (i) in Definition \ref{definition3.1} is met. Since $a=\lambda_{1}\mu_{1}$ and $\eta_{2}$ are positive constants and $\gamma>0$, it follows from the equation \eqref{eq47} that $\hat{q}(t)>0$ and $\{(\hat{q}(t),\hat{\pi}(t))\}_{t\in[0,T]}$ is $\mathcal{F}_{t}$-progressively measurable and continuous. Moreover, it is easy to know that $\hat{q}(t)$ and $\hat{\pi}(t)X^{\hat{u}}(t)$ is deterministic and continuous on $[0,T]$. Therefore, one can obtain that $\mathbb{E}\left[\int^{T}_{0}\left(\hat{q}^{2}(t)+\hat{\pi}^{2}(t)(X^{\hat{u}}(t))^{2}\right)\mathrm{d}t\right]<\infty$, which implies the condition (ii) in Definition \ref{definition3.1} holds. In addition, by the equation \eqref{eq24}, one has
\begin{align*}
&\int\left|(\varphi^{\gamma})^{-1}\left(\mathbb{E}_{t}\left[\varphi^{\gamma}(X^{\hat{u}}(T)+\beta M_{1}^{\hat{u}}(T))\right]\right)\right|\mathrm{d}\Gamma(\gamma)
=\int\left|(\varphi^{\gamma})^{-1}\left(Y^{\gamma}(t,x,m_{1})\right)\right|\mathrm{d}\Gamma(\gamma)\\
=&\int\left|\frac{1}{\gamma}\left[g_{1}^{\gamma}(t)(x+\beta m_{1})+g_{2}^{\gamma}(t)\right]\right|\mathrm{d}\Gamma(\gamma)
\leq\left|x+\beta m_{1}\right|e^{(A+\beta)(T-t)}+\left|\frac{a \eta}{A+\beta}\left[1-e^{(A+\beta)(T-t)}\right]\right|\\
&\mbox{} +1.5\left[\frac{(\mu-r)^{2}}{\sigma^{2}\mathbb{E}[\gamma]}+\frac{a^{2} \eta^{2}_{2}}{b^{2}\mathbb{E}[\gamma]}\right](T-t)
<\infty,
\end{align*}
which means that the condition (iii) in Definition \ref{definition3.1} is satisfied.

Since $Y^{\gamma}(t,x,m_{1})=-\frac{1}{\gamma}e^{g_{1}^{\gamma}(t)(x+\beta m_{1})+g_{2}^{\gamma}(t)}$, where $\gamma \in[\epsilon_{1},\infty)$, $g_{1}^{\gamma}(t)=-\gamma e^{(A+\beta)(T-t)}$ and $g_{2}^{\gamma}(t)$ is given by the equation \eqref{eq48}, we can derive that $Y^{\gamma}\in C^{1,2,1}(\mathcal{\widetilde{D}})$. In view of equations \eqref{eq86} and \eqref{eq9}, we have
\begin{align*}
   \mathrm{d}\left[X^{\hat{u}}(t)+\beta M_{1}^{\hat{u}}(t)\right]=&\left[(A+\beta)(X^{\hat{u}}(t)+\beta M_{1}^{\hat{u}}(t))+\frac{(\mu-r)^{2}}{\sigma^{2}\mathbb{E}[\gamma]}e^{-(A+\beta)(T-t)}+a\eta+\frac{a^{2} \eta^{2}_{2}}{b^{2}\mathbb{E}[\gamma]}e^{-(A+\beta)(T-t)}\right]\mathrm{d}t\\
  &+\frac{a \eta_{2}}{b\mathbb{E}[\gamma]}e^{-(A+\beta)(T-t)}\mathrm{d}W_{1}(t)+\frac{\mu-r}{\sigma\mathbb{E}[\gamma]}e^{-(A+\beta)(T-t)}\mathrm{d}W_{2}(t).
\end{align*}
Then, based on the above SDDE, we can obtain that
\begin{align*}
\left|Y^{\gamma}(t,X^{\hat{u}}(t),M_{1}^{\hat{u}}(t))\right|^{4}=&\left|\frac{1}{\gamma^{4}}e^{4g_{1}^{\gamma}(t)\left(X^{\hat{u}}(t)+\beta M_{1}^{\hat{u}}(t)\right)+4g_{2}^{\gamma}(t)}\right|\nonumber\\
\leq&Ke^{4g_{1}^{\gamma}(t)\left(X^{\hat{u}}(t)+\beta M_{1}^{\hat{u}}(t)\right)}\nonumber\\
=&Ke^{4g_{1}^{\gamma}(t)\left((x_{0}+\beta m_{10})e^{(A+\beta)t}+\int^{t}_{0}e^{(A+\beta)(t-s)}\left[X^{\hat{u}}(s)\hat{\pi}(s)(\mu-r)+a\eta+a\eta_{2}\hat{q}(s)\right]\mathrm{d}s
   \right)}\nonumber\\
&\times e^{4g_{1}^{\gamma}(t)\int^{t}_{0}e^{(A+\beta)(t-s)}\left[b\hat{q}(s)\mathrm{d}W_{1}(s)+\hat{\pi}(s)X^{\hat{u}}(s)\sigma \mathrm{d}W_{2}(s)\right]}\nonumber\\
   \leq&Ke^{
   -4\gamma\int^{t}_{0}\left[b\bar{q}\mathrm{d}W_{1}(s)+\bar{\pi}\sigma\mathrm{d}W_{2}(s)\right]},
\end{align*}
where we employ the boundedness of $g_{2}^{\gamma}(t)$, $X^{\hat{u}}(t)\hat{\pi}(t)$ and $\hat{q}(t)$, $\bar{q}=\frac{a \eta_{2}}{b^{2}\mathbb{E}[\gamma]}$ and $\bar{\pi}=\frac{\mu-r}{\sigma^{2}\mathbb{E}[\gamma]}$. Define $\widetilde{M}_{4}(t)=e^{-4\gamma\int^{t}_{0}\left[b\bar{q}\mathrm{d}W_{1}(s)+\bar{\pi}\sigma\mathrm{d}W_{2}(s)\right]}$. Because $\bar{q}$ and $\bar{\pi}$ are constants and $W_{1}$ and $W_{2}$ are independent, we find that
$$
\widetilde{M}_{4}(T)=\underbrace{e^{8\gamma^{2}\int^{T}_{0}\left(b^{2}\bar{q}^{2}+\bar{\pi}^{2}\sigma^{2}\right)\mathrm{d}s}}_{constant}\times
\underbrace{e^{-8\gamma^{2}\int^{T}_{0}\left(b^{2}\bar{q}^{2}+\bar{\pi}^{2}\sigma^{2}\right)\mathrm{d}s-4\gamma\int^{T}_{0}\left[b\bar{q}\mathrm{d}W_{1}(s)+\bar{\pi}\sigma\mathrm{d}W_{2}(s)\right]}}_{martingale}.
$$
Hence, $\mathbb{E}\left[\widetilde{M}_{4}(T)\right]<\infty$. By Burkh\"{o}lder-Davis-Gundy inequality, we know that
\begin{align}\label{eq87}
  \mathbb{E}\left[\sup_{t\in[0,T]}\left|Y^{\gamma}(t,X^{\hat{u}}(t),M_{1}^{\hat{u}}(t))\right|^{4}\right]&\leq K\mathbb{E}\left[\left|Y^{\gamma}(T,X^{\hat{u}}(T),M_{1}^{\hat{u}}(T))\right|^{4}\right]\leq K\mathbb{E}\left[\widetilde{M}_{4}(T)\right]<\infty.
\end{align}
Proceeding as in the proof of Lemma \ref{lemma3.2}, one arrives at $Y^{\gamma}(\cdot,X^{\hat{u}}(\cdot),M_{1}^{\hat{u}}(\cdot))\in\mathbb{S}^{2}_{\mathcal{F}}(0,T;\mathbb{R})$.

We now proceed to verify that $Y^{\gamma}(\cdot,X^{\hat{u}}(\cdot),M_{1}^{\hat{u}}(\cdot))$ satisfies the condition $(A1)$. From the expression of $Y^{\gamma}$, it follows that
\begin{align*}
&\mathbb{E}\left[\int^{T}_{0}\left|Y^{\gamma}_{t}(t,X^{\hat{u}}(t),M_{1}^{\hat{u}}(t))\right|\mathrm{d}t\right]\\
=&\mathbb{E}\left[\int^{T}_{0}\left|Y^{\gamma}(t,X^{\hat{u}}(t),M_{1}^{\hat{u}}(t))\left[\frac{\partial g_{1}^{\gamma}(t)}{\partial t}\left(X^{\hat{u}}(t)+\beta M_{1}^{\hat{u}}(t)\right)+\frac{\partial g_{2}^{\gamma}(t)}{\partial t}\right]\right|\mathrm{d}t\right]\\
\leq& K\mathbb{E}\left[\sup_{t\in[0,T]}\left|Y^{\gamma}(t,X^{\hat{u}}(t),M_{1}^{\hat{u}}(t))\right|\times \left(\sup_{t\in[0,T]}\left|X^{\hat{u}}(t)+\beta M_{1}^{\hat{u}}(t)\right|+K\right)\right],
\end{align*}
in which the inequality holds because $\frac{\partial g_{1}^{\gamma}(t)}{\partial t}$ and $\frac{\partial g_{2}^{\gamma}(t)}{\partial t}$ are deterministic and bounded. Since $Y^{\gamma}(\cdot,X^{\hat{u}}(\cdot),M_{1}^{\hat{u}}(\cdot))\in\mathbb{S}^{2}_{\mathcal{F}}(0,T;\mathbb{R})$ and $X^{\hat{u}}(\cdot)\in\mathbb{S}^{2}_{\mathcal{F}}(0,T;\mathbb{R})$, we have
\begin{align*}
&\mathbb{E}\left[\sup_{t\in[0,T]}\left|Y^{\gamma}(t,X^{\hat{u}}(t),M_{1}^{\hat{u}}(t))\right|\times \sup_{t\in[0,T]}\left|X^{\hat{u}}(t)\right|\right]\\
\leq&
\left(\mathbb{E}\left[\left(\sup_{t\in[0,T]}\left|Y^{\gamma}(t,X^{\hat{u}}(t),M_{1}^{\hat{u}}(t))\right|\right)^{2}\right]\times
\mathbb{E}\left[\left(\sup_{t\in[0,T]}\left|X^{\hat{u}}(t)\right|\right)^{2}\right]\right)^{\frac{1}{2}}\\
=&
\left(\mathbb{E}\left[\sup_{t\in[0,T]}\left|Y^{\gamma}(t,X^{\hat{u}}(t),M_{1}^{\hat{u}}(t))\right|^{2}\right]\times
\mathbb{E}\left[\sup_{t\in[0,T]}\left|X^{\hat{u}}(t)\right|^{2}\right]\right)^{\frac{1}{2}}<\infty.
\end{align*}
Because $X^{\hat{u}}(\cdot)\in\mathbb{S}^{2}_{\mathcal{F}}(0,T;\mathbb{R})$, one has $M_{1}^{\hat{u}}(\cdot)\in\mathbb{S}^{2}_{\mathcal{F}}(0,T;\mathbb{R})$ by the proof process of Theorem \ref{theorem4.4}. Similarly, we can also obtain
$$\mathbb{E}\left[\sup_{t\in[0,T]}\left|Y^{\gamma}(t,X^{\hat{u}}(t),M_{1}^{\hat{u}}(t))\right|\times \sup_{t\in[0,T]}\left|M_{1}^{\hat{u}}(t)\right|\right]
<\infty, \quad \mathbb{E}\left[\sup_{t\in[0,T]}\left|Y^{\gamma}(t,X^{\hat{u}}(t),M_{1}^{\hat{u}}(t))\right|\right]
<\infty.$$
Thus, the following holds $\mathbb{E}\left[\int^{T}_{0}\left|Y^{\gamma}_{t}(t,X^{\hat{u}}(t),M_{1}^{\hat{u}}(t))\right|\mathrm{d}t\right]<\infty$, which shows $Y^{\gamma}_{t}(t,X^{\hat{u}}(t),M_{1}^{\hat{u}}(t))\in\mathbb{L}^{1}_{\mathcal{F}}(0,T;\mathbb{R})$. Mirroring this procedure, $\left[(A+\hat{\pi}(t)(\mu-r))X^{\hat{u}}(t)+BM_{1}^{\hat{u}}(t)+CM_{2}^{\hat{u}}(t)+a\eta+a\eta_{2}\hat{q}(t)\right]Y^{\gamma}_{x}(t,X^{\hat{u}}(t),M_{1}^{\hat{u}}(t))$, $(b^{2}\hat{q}^{2}(t)+\hat{\pi}^{2}(t)(X^{\hat{u}}(t))^{2}\sigma^{2})Y^{\gamma}_{xx}(t,X^{\hat{u}}(t),M_{1}^{\hat{u}}(t))$, $\left(X^{\hat{u}}(t)-\alpha M_{1}^{\hat{u}}(t)-e^{-\alpha h}M_{2}^{\hat{u}}(t)\right)Y^{\gamma}_{m_{1}}(t,X^{\hat{u}}(t),M_{1}^{\hat{u}}(t))\in\mathbb{L}^{1}_{\mathcal{F}}(0,T;\mathbb{R})$ can be also verified. Moreover, given that $\hat{q}(t)$ and $g_{1}^{\gamma}(t)$ are both deterministic and bounded, it follows that
\begin{align*}
\mathbb{E}\left[\int^{T}_{0}\left|b\hat{q}(t)Y^{\gamma}_{x}(t,X^{\hat{u}}(t),M_{1}^{\hat{u}}(t))\right|^{2}\mathrm{d}t\right]\leq& K\mathbb{E}\left[\sup_{t\in[0,T]}\left|Y^{\gamma}(t,X^{\hat{u}}(t),M_{1}^{\hat{u}}(t))\right|^{2}\right]
<\infty,
\end{align*}
which means that $b\hat{q}(t)Y^{\gamma}_{x}(t,X^{\hat{u}}(t),M_{1}^{\hat{u}}(t))\in\mathbb{L}^{2}_{\mathcal{F}}(0,T;\mathbb{R})$. Analogously, $\hat{\pi}(t)X^{\hat{u}}(t)\sigma Y^{\gamma}_{x}(t,X^{\hat{u}}(t),M_{1}^{\hat{u}}(t))\in\mathbb{L}^{2}_{\mathcal{F}}(0,T;\mathbb{R})$. Overall, $Y^{\gamma}(t,X^{\hat{u}}(t),M_{1}^{\hat{u}}(t))$ satisfies the condition $(A1)$.

By equations \eqref{eq28} and \eqref{eq50} and $X^{\hat{u}}(\cdot),M_{1}^{\hat{u}}(\cdot)\in\mathbb{S}^{1}_{\mathcal{F}}(0,T;\mathbb{R})$, we have $U, H\in C^{1,2,1}(\mathcal{\widetilde{D}})$ and
\begin{align*}
\mathbb{E}\left[\int^{T}_{0}\left|U_{t}(t,X^{\hat{u}}(t),M_{1}^{\hat{u}}(t))\right|\mathrm{d}t\right]
\leq K\mathbb{E}\left[\sup_{t\in[0,T]}\left|X^{\hat{u}}(t)\right|+\beta\sup_{t\in[0,T]}\left|M_{1}^{\hat{u}}(t)\right|\right]+K
<\infty.
\end{align*}
Then, we know that $U_{t}(t,X^{\hat{u}}(t),M_{1}^{\hat{u}}(t))\in\mathbb{L}^{1}_{\mathcal{F}}(0,T;\mathbb{R})$. Likewise, one can draw the following conclusion $\left[(A+\hat{\pi}(t)(\mu-r))X^{\hat{u}}(t)+BM_{1}^{\hat{u}}(t)+CM_{2}^{\hat{u}}(t)+a\eta+a\eta_{2}\hat{q}(t)\right]U_{x}(t,X^{\hat{u}}(t),M_{1}^{\hat{u}}(t))$, $(b^{2}\hat{q}^{2}(t)+\hat{\pi}^{2}(t)(X^{\hat{u}}(t))^{2}\sigma^{2})$ $U_{xx}(t,X^{\hat{u}}(t),M_{1}^{\hat{u}}(t))$, $\left(X^{\hat{u}}(t)-\alpha M_{1}^{\hat{u}}(t)-e^{-\alpha h}M_{2}^{\hat{u}}(t)\right)U_{m_{1}}(t,X^{\hat{u}}(t),M_{1}^{\hat{u}}(t))\in\mathbb{L}^{1}_{\mathcal{F}}(0,T;\mathbb{R})$. By using the equation \eqref{eq50}, one has
\begin{align*}
\mathbb{E}\left[\int^{T}_{0}\left|b\hat{q}(t)U_{x}(t,X^{\hat{u}}(t),M_{1}^{\hat{u}}(t))\right|^{2}\mathrm{d}t\right]=\mathbb{E}\left[\int^{T}_{0}\left|b\hat{q}(t)e^{(A+\beta)(T-t)}\right|^{2}\mathrm{d}t\right]
<\infty,
\end{align*}
which shows $b\hat{q}(t)U_{x}(t,X^{\hat{u}}(t),M_{1}^{\hat{u}}(t))\in\mathbb{L}^{2}_{\mathcal{F}}(0,T;\mathbb{R})$. Following analogous step, one can demonstrate that $\hat{\pi}(t)X^{\hat{u}}(t)\sigma U_{x}(t,X^{\hat{u}}(t),M_{1}^{\hat{u}}(t))\in\mathbb{L}^{2}_{\mathcal{F}}(0,T;\mathbb{R})$. Therefore, $U(t,X^{\hat{u}}(t),M_{1}^{\hat{u}}(t))$ meets the condition $(A1)$. For $H=U$, $H(t,X^{\hat{u}}(t),M_{1}^{\hat{u}}(t))$ also fulfills the condition $(A1)$.
\end{proof}

\section{Proof of Theorem \ref{theorem4.2}}
\begin{proof}
We need to verify the admissibility of candidate strategy \eqref{eq66} for reinsurance and investment, as well as the satisfaction of Theorem \ref{theorem3.1}'s assumptions, to ensure the validity of the equilibrium strategy.

If there exists at least a local solution to $g^{\gamma_{i}}(t)$ in the equation \eqref{eq60} for $i=1,\cdot\cdot\cdot,n$, then it follows from $g^{\gamma_{i}}(T)=1$ that $g^{\gamma_{i}}(t)$ is positive and continuous on $[0,T]$. Hence, $\varpi(t,\Gamma)$ is also positive and continuous. Substituting the expression \eqref{eq66} into the SDDE \eqref{eq11} yields
\begin{equation}\label{eq67}
   \left\{ \begin{aligned}
   \mathrm{d}X^{\hat{u}}(t)=&\left[AX^{\hat{u}}(t)+BM_{1}^{\hat{u}}(t)+CM_{2}^{\hat{u}}(t)+\left(\frac{(\mu- r)^{2} }
{\sigma^{2}\varpi(t,\Gamma)}+\frac{a^{2} \eta^{2}_{2}}
{b^{2}\varpi(t,\Gamma)}
  \right)(X^{\hat{u}}(t)+\beta M_{1}^{\hat{u}}(t))\right]\mathrm{d}t\\
  &+\frac{a \eta_{2}}{b\varpi(t,\Gamma)}(X^{\hat{u}}(t)+\beta M_{1}^{\hat{u}}(t))\mathrm{d}W_{1}(t)+\frac{\mu-r}{\sigma\varpi(t,\Gamma)}(X^{\hat{u}}(t)+\beta M_{1}^{\hat{u}}(t))\mathrm{d}W_{2}(t),\; t\in[\tau,T],\\
   X^{\hat{u}}(t)=&\psi(t-\tau), \;t\in[\tau-h,\tau].
  \end{aligned}\right.
\end{equation}
By Proposition 2.1 in \cite{Meng2025}, one can obtain that the SDDE \eqref{eq67} admits a unique strong solution $X^{\hat{u}}(\cdot)\in\mathbb{S}^{p}_{\mathcal{F}}(0,T;\mathbb{R})$ for $p\geq2$. This indicates that the condition (i) in Definition \ref{definition3.1} holds. Moreover, by using equations \eqref{eq9} and \eqref{eq67}, one has
\begin{align}\label{eq68}
\mathrm{d}\left[X^{\hat{u}}(t)+\beta M_{1}^{\hat{u}}(t)\right]=&\left[(A+\beta)X^{\hat{u}}(t)+(B-\alpha\beta)M_{1}^{\hat{u}}(t)+(C-e^{-\alpha h}\beta)M_{2}^{\hat{u}}(t)\right]\mathrm{d}t \nonumber\\
&+\left(X^{\hat{u}}(t)+\beta M_{1}^{\hat{u}}(t)\right)\left[\widetilde{M}_{5}(t)\mathrm{d}t+\widetilde{M}_{6}(t)\mathrm{d}W_{1}(t)+\widetilde{M}_{7}(t)\mathrm{d}W_{2}(t)\right],
\end{align}
where $\widetilde{M}_{5}(t)=\frac{(\mu- r)^{2} }{\sigma^{2}\varpi(t,\Gamma)}+\frac{a^{2} \eta^{2}_{2}}{b^{2}\varpi(t,\Gamma)}$, $\widetilde{M}_{6}(t)=\frac{a \eta_{2}}{b\varpi(t,\Gamma)}$ and $\widetilde{M}_{7}(t)=\frac{\mu-r}{\sigma\varpi(t,\Gamma)}$. Applying the condition $(A2)$, the equation \eqref{eq68} can be rewritten as
\begin{align*}
\mathrm{d}\left[X^{\hat{u}}(t)+\beta M_{1}^{\hat{u}}(t)\right]=\left(X^{\hat{u}}(t)+\beta M_{1}^{\hat{u}}(t)\right)\left[\left(A+\beta+\widetilde{M}_{5}(t)\right)\mathrm{d}t+\widetilde{M}_{6}(t)\mathrm{d}W_{1}(t)+\widetilde{M}_{7}(t)\mathrm{d}W_{2}(t)\right].
\end{align*}
Combining the initial conditions of $X^{\hat{u}}(t)$ and $M_{1}^{\hat{u}}(t)$ at $t=0$, we have
\begin{align*}
X^{\hat{u}}(t)+\beta M_{1}^{\hat{u}}(t)=&\left(x_{0}+\frac{\beta x_{0}(1-e^{-\alpha h})}{\alpha}\right)\exp\left\{\int^{t}_{0}\left(A+\beta+\widetilde{M}_{5}(s)-\frac{(\widetilde{M}_{6}(s))^{2}+(\widetilde{M}_{7}(s))^{2}}{2}\right)\mathrm{d}s\right\}\\
&\times\exp\left\{\int^{t}_{0}\left(\widetilde{M}_{6}(s)\mathrm{d}W_{1}(s)+\widetilde{M}_{7}(s)\mathrm{d}W_{2}(s)\right)\right\},
\end{align*}
which shows that $X^{\hat{u}}(t)+\beta M_{1}^{\hat{u}}(t)>0$ for $t\in[0,T]$ since $x_{0}$, $\beta$, $\alpha$ and $h$ are positive. Thus, it follows from the equation \eqref{eq66} that $\hat{q}(t)>0$ and $\{(\hat{q}(t),\hat{\pi}(t))\}_{t\in[0,T]}$ is $\mathcal{F}_{t}$-progressively measurable and continuous. In addition, since $X^{\hat{u}}(\cdot)\in\mathbb{S}^{2}_{\mathcal{F}}(0,T;\mathbb{R})$, one can derive that $M_{1}^{\hat{u}}(\cdot)\in\mathbb{S}^{2}_{\mathcal{F}}(0,T;\mathbb{R})$ by the proof process of Theorem \ref{theorem4.4}. Therefore, one has
\begin{align*}
\mathbb{E}\left[\int^{T}_{0}\left(\hat{q}^{2}(t)+\hat{\pi}^{2}(t)(X^{\hat{u}}(t))^{2}\right)\mathrm{d}t\right]
\leq &\mathbb{E}\left[\int^{T}_{0}\left(\frac{(\mu- r)^{2} }{\sigma^{4}\varpi^{2}(t,\Gamma)}+\frac{a^{2} \eta^{2}_{2}}
{b^{4}\varpi^{2}(t,\Gamma)}\right)(X^{\hat{u}}(t)+\beta M_{1}^{\hat{u}}(t))^{2}\mathrm{d}t\right]\\
\leq &K\mathbb{E}\left[\sup_{t\in[0,T]}\left(\left|X^{\hat{u}}(t)\right|^{2}
+\beta^{2}\left|M_{1}^{\hat{u}}(t)\right|^{2}\right)\right]
<\infty.
\end{align*}
Hence, the condition (ii) in Definition \ref{definition3.1} is satisfied. Furthermore, by using the equation \eqref{eq24}, we have
\begin{align*}
&\int\left|(\varphi^{\gamma})^{-1}\left(\mathbb{E}_{t}\left[\varphi^{\gamma}(X^{\hat{u}}(T)+\beta M_{1}^{\hat{u}}(T))\right]\right)\right|\mathrm{d}\Gamma(\gamma)
\\=&
\int\left|(\varphi^{\gamma})^{-1}\left(Y^{\gamma}(t,x,m_{1})\right)\right|\mathrm{d}\Gamma(\gamma)
=\left(x+\beta m_{1}\right)\sum\limits_{i=1}^{n}(g^{\gamma_{i}}(t))^{\frac{\gamma_{i}}{1-\gamma_{i}}}p_{i}
<\infty,
\end{align*}
which implies that the condition (iii) in Definition \ref{definition3.1} is met.

On the other hand, by the proof process of Theorem \ref{theorem4.4}, we find that $|M_{1}^{\hat{u}}(t)|\leq K \|X^{\hat{u}}_{t}\|$. Clearly, $|X^{\hat{u}}(t)|\leq \|X^{\hat{u}}_{t}\|$. Because $Y^{\gamma_{i}}(t,x,m_{1})=\frac{1}{1-\gamma_{i}}(g^{\gamma_{i}}(t))^{\gamma_{i}}(x+\beta m_{1})^{1-\gamma_{i}}$, where $g^{\gamma_{i}}(t)$ is captured by the equation \eqref{eq60} and bounded on $[0,T]$, we can derive that $Y^{\gamma_{i}}\in C^{1,2,1}(\mathcal{\widetilde{D}})$ and
\begin{align*}
\mathbb{E}\left[\int^{T}_{0}\left|Y^{\gamma_{i}}_{t}(t,X^{\hat{u}}(t),M_{1}^{\hat{u}}(t))\right|\mathrm{d}t\right]&=\mathbb{E}\left[\int^{T}_{0}\left|\frac{\gamma_{i}}{1-\gamma_{i}}(g^{\gamma_{i}}(t))^{\gamma_{i}-1}
(X^{\hat{u}}(t)+\beta M_{1}^{\hat{u}}(t))^{1-\gamma_{i}}\frac{\partial g^{\gamma_{i}}(t)}{\partial t}\right|\mathrm{d}t\right]\\
&\leq K\mathbb{E}\left[\sup_{t\in[0,T]}\left(X^{\hat{u}}(t)+\beta M_{1}^{\hat{u}}(t)\right)^{1-\gamma_{i}}\right]
\leq K\mathbb{E}\left[\|X^{\hat{u}}_{t}\|^{1-\gamma_{i}}\right],
\end{align*}
where the inequality follows from the fact that $\frac{\partial g^{\gamma_{i}}(t)}{\partial t}$ is deterministic and bounded as well as $\gamma_{i}\in(0,\epsilon_{2}]$ with $0<\epsilon_{2}<1$ for $i=1,\cdot\cdot\cdot,n$. Then, since $X^{\hat{u}}(\cdot)\in\mathbb{S}^{2}_{\mathcal{F}}(0,T;\mathbb{R})$, we can further obtain that
\begin{align*}
\left(\mathbb{E}\left[\|X^{\hat{u}}_{t}\|^{1-\gamma_{i}}\right]\right)^{2}\leq &K\mathbb{E}\left[\|X^{\hat{u}}_{t}\|^{2(1-\gamma_{i})}\right]\leq K\mathbb{E}\left[\left(1+\|X^{\hat{u}}_{t}\|\right)^{2(1-\gamma_{i})}\right]
\leq K\mathbb{E}\left[\left(1+\|X^{\hat{u}}_{t}\|\right)^{2}\right]\\
\leq& K\mathbb{E}\left[1+\|X^{\hat{u}}_{t}\|^{2}\right]
\leq K\mathbb{E}\left[1+\sup_{t\in[0,T]}(X^{\hat{u}}(t))^{2}+x^{2}_{0}\right]
<\infty.
\end{align*}
Hence, $Y^{\gamma_{i}}_{t}(t,X^{\hat{u}}(t),M_{1}^{\hat{u}}(t))\in\mathbb{L}^{1}_{\mathcal{F}}(0,T;\mathbb{R})$. Additionally, letting $\widetilde{M}_{8}(t)=AX^{\hat{u}}(t)+BM_{1}^{\hat{u}}(t)+CM_{2}^{\hat{u}}(t)$ and $\widetilde{M}_{9}(t)=X^{\hat{u}}(t)-\alpha M_{1}^{\hat{u}}(t)-e^{-\alpha h}M_{2}^{\hat{u}}(t)$, by the condition $(A2)$ and applying a similar proof procedure as for $Y^{\gamma_{i}}_{t}(t,X^{\hat{u}}(t),M_{1}^{\hat{u}}(t))\in\mathbb{L}^{1}_{\mathcal{F}}(0,T;\mathbb{R})$, one can get that
\begin{align*}
&\mathbb{E}\left[\int^{T}_{0}\left| \widetilde{M}_{8}(t)  Y^{\gamma_{i}}_{x}(t,X^{\hat{u}}(t),M_{1}^{\hat{u}}(t))
+ \widetilde{M}_{9}(t) Y^{\gamma_{i}}_{m_{1}}(t,X^{\hat{u}}(t),M_{1}^{\hat{u}}(t))\right|\mathrm{d}t\right]\\
=&\mathbb{E}\left[\int^{T}_{0}\left|\widetilde{M}_{8}(t)(g^{\gamma_{i}}(t))^{\gamma_{i}}\left(X^{\hat{u}}(t)+\beta M_{1}^{\hat{u}}(t)\right)^{-\gamma_{i}}
+\widetilde{M}_{9}(t)\beta(g^{\gamma_{i}}(t))^{\gamma_{i}}\left(X^{\hat{u}}(t)+\beta M_{1}^{\hat{u}}(t)\right)^{-\gamma_{i}}\right|\mathrm{d}t\right]\\
\leq&K\mathbb{E}\left[\int^{T}_{0}\left(X^{\hat{u}}(t)+\beta M_{1}^{\hat{u}}(t)\right)^{1-\gamma_{i}}\mathrm{d}t\right]
\leq K\mathbb{E}\left[\sup_{t\in[0,T]}\left(X^{\hat{u}}(t)+\beta M_{1}^{\hat{u}}(t)\right)^{1-\gamma_{i}}\right]
<\infty,
\end{align*}
which means that $\widetilde{M}_{8}(t)Y^{\gamma_{i}}_{x}(t,X^{\hat{u}}(t),M_{1}^{\hat{u}}(t))+\widetilde{M}_{9}(t)Y^{\gamma_{i}}_{m_{1}}(t,X^{\hat{u}}(t),M_{1}^{\hat{u}}(t)) \in\mathbb{L}^{1}_{\mathcal{F}}(0,T;\mathbb{R})$.  Analogously, it can be shown that $\widetilde{M}_{10}(t)Y^{\gamma_{i}}_{x}(t,X^{\hat{u}}(t),M_{1}^{\hat{u}}(t)) + \widetilde{M}_{11}(t)Y^{\gamma_{i}}_{xx}(t,X^{\hat{u}}(t),M_{1}^{\hat{u}}(t)) \in\mathbb{L}^{1}_{\mathcal{F}}(0,T;\mathbb{R})$ with $\widetilde{M}_{10}(t)=\hat{\pi}(t)(\mu-r)X^{\hat{u}}(t)+a\eta_{2}\hat{q}(t)
$ and $\widetilde{M}_{11}(t)=0.5\left(b^{2}\hat{q}^{2}(t)
+\hat{\pi}^{2}(t)(X^{\hat{u}}(t))^{2}\sigma^{2}\right)$. Moreover, by using a similar proof process as for $Y^{\gamma_{i}}_{t}(t,X^{\hat{u}}(t),M_{1}^{\hat{u}}(t))\in\mathbb{L}^{1}_{\mathcal{F}}(0,T;\mathbb{R})$, one can arrive at
\begin{align*}
&\mathbb{E}\left[\int^{T}_{0}\left|b\hat{q}(t)Y^{\gamma_{i}}_{x}(t,X^{\hat{u}}(t),M_{1}^{\hat{u}}(t))\right|^{2}\mathrm{d}t\right]=
\mathbb{E}\left[\int^{T}_{0}\left|\frac{a \eta_{2}}{b\varpi(t,\Gamma)}(g^{\gamma_{i}}(t))^{\gamma_{i}}
\left(X^{\hat{u}}(t)+\beta M_{1}^{\hat{u}}(t)\right)^{1-\gamma_{i}}\right|^{2}\mathrm{d}t\right]\\
\leq& K\mathbb{E}\left[\sup_{t\in[0,T]}\left(X^{\hat{u}}(t)+\beta M_{1}^{\hat{u}}(t)\right)^{2(1-\gamma_{i})}\right]\leq K\mathbb{E}\left[\|X^{\hat{u}}_{t}\|^{2(1-\gamma_{i})}\right]
<\infty,
\end{align*}
which implies that $b\hat{q}(t)Y^{\gamma_{i}}_{x}(t,X^{\hat{u}}(t),M_{1}^{\hat{u}}(t))\in\mathbb{L}^{2}_{\mathcal{F}}(0,T;\mathbb{R})$. Similarly, $\hat{\pi}(t)X^{\hat{u}}(t)\sigma Y^{\gamma_{i}}_{x}(t,X^{\hat{u}}(t),M_{1}^{\hat{u}}(t))\in\mathbb{L}^{2}_{\mathcal{F}}(0,T;\mathbb{R})$ can be also verified. Ultimately, $Y^{\gamma_{i}}(t,X^{\hat{u}}(t),M_{1}^{\hat{u}}(t))$  fulfils the condition $(A1)$.

What is more, by the equation \eqref{eq28}, one has
\begin{align*}
U(t,x,m_{1})=H(t,x,m_{1})=\int(\varphi^{\gamma})^{-1}\left(Y^{\gamma}(t,x,m_{1})\right)\mathrm{d}\Gamma(\gamma)=\left(x+\beta m_{1}\right)\sum\limits_{i=1}^{n}(g^{\gamma_{i}}(t))^{\frac{\gamma_{i}}{1-\gamma_{i}}}p_{i}.
\end{align*}
In light of $X^{\hat{u}}(\cdot)\in\mathbb{S}^{2}_{\mathcal{F}}(0,T;\mathbb{R})$ and Lemma \ref{lemma3.2}, one can also get that
$X^{\hat{u}}(\cdot),\;M_{1}^{\hat{u}}(\cdot)\in\mathbb{S}^{1}_{\mathcal{F}}(0,T;\mathbb{R})$. Since $g^{\gamma_{i}}(t)$ is deterministic and continuous on $[0,T]$ for $i=1,\cdot\cdot\cdot,n$, we have $U, H\in C^{1,2,1}(\mathcal{\widetilde{D}})$ and
\begin{align*}
\mathbb{E}\left[\int^{T}_{0}\left|U_{t}(t,X^{\hat{u}}(t),M_{1}^{\hat{u}}(t))\right|\mathrm{d}t\right]
\leq K\mathbb{E}\left[\sup_{t\in[0,T]}\left|X^{\hat{u}}(t)\right|+\beta\sup_{t\in[0,T]}\left|M_{1}^{\hat{u}}(t)\right|\right]
<\infty.
\end{align*}
In a similar manner, $\left[(A+\hat{\pi}(t)(\mu-r))X^{\hat{u}}(t)+BM_{1}^{\hat{u}}(t)+CM_{2}^{\hat{u}}(t)+a\eta_{2}\hat{q}(t)\right]U_{x}(t,X^{\hat{u}}(t),M_{1}^{\hat{u}}(t))$ $+0.5(b^{2}\hat{q}^{2}(t)+\hat{\pi}^{2}(t)(X^{\hat{u}}(t))^{2}\sigma^{2})U_{xx}(t,X^{\hat{u}}(t),M_{1}^{\hat{u}}(t))+\left(X^{\hat{u}}(t)-\alpha M_{1}^{\hat{u}}(t)-e^{-\alpha h}M_{2}^{\hat{u}}(t)\right)U_{m_{1}}(t,X^{\hat{u}}(t),M_{1}^{\hat{u}}(t))\in\mathbb{L}^{1}_{\mathcal{F}}(0,T;\mathbb{R})$ can be also proved. Furthermore, because $X^{\hat{u}}(\cdot),M_{1}^{\hat{u}}(\cdot)\in\mathbb{S}^{2}_{\mathcal{F}}(0,T;\mathbb{R})$, we have
\begin{align*}
&\mathbb{E}\left[\int^{T}_{0}\left|b\hat{q}(t)U_{x}(t,X^{\hat{u}}(t),M_{1}^{\hat{u}}(t))\right|^{2}\mathrm{d}t\right]=\mathbb{E}\left[\int^{T}_{0}\left|\frac{a \eta_{2} \left(X^{\hat{u}}(t)+\beta M_{1}^{\hat{u}}(t)\right)}{b\varpi(t,\Gamma)}\sum\limits_{i=1}^{n}(g^{\gamma_{i}}(t))^{\frac{\gamma_{i}}{1-\gamma_{i}}}p_{i}\right|^{2}\mathrm{d}t\right]\\
\leq& K\mathbb{E}\left[\sup_{t\in[0,T]}\left(X^{\hat{u}}(t)+\beta M_{1}^{\hat{u}}(t)\right)^{2}\right]
\leq K\mathbb{E}\left[\sup_{t\in[0,T]}\left|X^{\hat{u}}(t)\right|^{2}
+\beta^{2}\sup_{t\in[0,T]}\left|M_{1}^{\hat{u}}(t)\right|^{2}\right]
<\infty,
\end{align*}
which means that $b\hat{q}(t)U_{x}(t,X^{\hat{u}}(t),M_{1}^{\hat{u}}(t))\in\mathbb{L}^{2}_{\mathcal{F}}(0,T;\mathbb{R})$. Likewise, $\hat{\pi}(t)X^{\hat{u}}(t)\sigma U_{x}(t,X^{\hat{u}}(t),M_{1}^{\hat{u}}(t))\in\mathbb{L}^{2}_{\mathcal{F}}(0,T;\mathbb{R})$. Overall, $U(t,X^{\hat{u}}(t),M_{1}^{\hat{u}}(t))$ meets the condition $(A1)$. Since $U=H$, $H(t,X^{\hat{u}}(t),M_{1}^{\hat{u}}(t))$ also satisfies the condition $(A1)$.
\end{proof}


\begin{thebibliography}{99}

\bibitem{A2018}
C.X. A, Y.Z. Lai and Y. Shao,
Optimal excess-of-loss reinsurance and investment problem with delay and jump-diffusion risk process under the CEV model,
{\em Journal of Computational and Applied Mathematics}, {\bf342}:317-336, 2018.

\bibitem{A2020}
C.X. A and Y. Shao,
Optimal investment and reinsurance problem with delay under the CEV model,
{\em Operations Research Transactions}, {\bf24(1)}:73-87, 2020.

\bibitem{A2022}
C.X. A, Y. Shen and Y. Zeng,
Dynamic asset-liability management problem in a continuous-time model with delay,
{\em International Journal of Control}, {\bf95(5)}:1315-1336, 2022.


\bibitem{Bai2008}
L.H. Bai and J.Y. Guo,
Optimal proportional reinsurance and investment with multiple risky assets and no-shorting constraint,
{\em Insurance: Mathematics and Economics}, {\bf42(3)}:968-975, 2008.

\bibitem{Bai2022}
Y.F. Bai, Z.B. Zhou, H.L. Xiao, R. Gao and F.M. Zhong,
A hybrid stochastic differential reinsurance and investment game with bounded memory,
{\em European Journal of Operational Research}, {\bf296(2)}:717-737, 2022.

\bibitem{Balter2021}
A.G. Balter, A. Mahayni and N. Schweizer,
Time-consistency of optimal investment under smooth ambiguity,
{\em European Journal of Operational Research}, {\bf293(2)}:643-657, 2021.

\bibitem{Bensoussan2022}
A. Bensoussan, G.Y. Ma, C.C. Siu and S.C.P. Yam,
Dynamic mean-variance problem with frictions,
{\em Finance and Stochastics}, {\bf26(2)}:267-300, 2022.


\bibitem{Bi2019}
J. Bi and J. Cai,
 Optimal investment-reinsurance strategies with state dependent risk aversion and VaR constraints in correlated markets,
{\em Insurance: Mathematics and Economics}, {\bf85}:1-14, 2019.

\bibitem{Bjork2017}
T. Bj{\"{o}}rk, M. Khapko and A. Murgoci,
On time-inconsistent stochastic control in continuous time,
{\em Finance and Stochastics}, {\bf21}:331-360, 2017.


\bibitem{Bjork2021}
T. Bj{\"{o}}rk, M. Khapko and A. Murgoci,
{\em Time-inconsistent Control Theory with Finance Applications},
Springer, Cham, 2021.

\bibitem{Bjork2014}
T. Bj{\"{o}}rk and A. Murgoci,
 A theory of Markovian time-inconsistent stochastic control in discrete time,
{\em Finance and Stochastics}, {\bf18}:545-592, 2014.


\bibitem{Bollerslev2011}
T. Bollerslev, M. Gibson and H. Zhou,
 Dynamic estimation of volatility risk premia and investor risk aversion from option-implied and realized volatilities,
{\em Journal of Econometrics}, {\bf160(1)}:235-245, 2011.

\bibitem{Browne1995}
S. Browne,
 Optimal investment policies for a firm with a random risk process: exponential utility and minimizing the probability of ruin,
{\em Mathematics of Operations Research}, {\bf20(4)}:937-958, 1995.

\bibitem{Burgaard2020}
J. Burgaard and M. Steffensen,
Eliciting risk preferences and elasticity of substitution,
{\em Decision Analysis}, {\bf17(4)}:314-329, 2020.


\bibitem{Chang2011}
M.H. Chang, T. Pang and Y. Yang,
A stochastic portfolio optimization model with bounded memory,
{\em Mathematics of Operations Research}, {\bf36(4)}:604-619, 2011.


\bibitem{Chetty2006}
R. Chetty,
 A new method of estimating risk aversion,
{\em American Economic Review}, {\bf96(5)}:1821-1834, 2006.

\bibitem{Desmettre2023}
S. Desmettre and M. Steffensen,
Equilibrium investment with random risk aversion,
{\em Mathematical Finance}, {\bf33}:946-975, 2023.

\bibitem{Elliott2011}
R.J. Elliott and T.K. Siu,
A stochastic differential game for optimal investment of an insurer with regime switching,
{\em Quantitative Finance}, {\bf11}:365-380, 2011.


\bibitem{Elsanosi2000}
I. Elsanosi, B. {{\O}}ksendal and A. Sulem,
Some solvable stochastic control problems with delay,
{\em Stochastics: An International Journal of Probability and Stochastic Processes}, {\bf71(1-2)}:69-89, 2000.


\bibitem{Federico2011}
S. Federico,
A stochastic control problem with delay arising in a pension fund model,
{\em Finance and Stochastics}, {\bf15(3)}:421-459, 2011.


\bibitem{Grandell1990}
J. Grandell,
{\em Aspects of Risk Theory},
Springer, New York, 1990.

\bibitem{Hojgaard1998}
B. H{{\o}}jgaard and M. Taksar,
Optimal proportional reinsurance policies for diffusion models,
{\em Scandinavian Actuarial Journal}, {\bf2}:166-180, 1998.


\bibitem{Hu2012}
Y. Hu, H.Q. Jin and X.Y. Zhou,
 Time-consistent stochastic linear-quadratic control,
{\em SIAM Journal on Control and Optimization}, {\bf50(3)}:1548-1572, 2012.

\bibitem{Hu2017}
Y. Hu, H.Q. Jin and X.Y. Zhou,
 Time-inconsistent stochastic linear-quadratic control: characterization and uniqueness of equilibrium,
{\em SIAM Journal on Control and Optimization}, {\bf55(2)}:1261-1279, 2017.

\bibitem{Kang2026}
J.H. Kang, Z. Gou and N.J. Huang,
Equilibrium reinsurance and investment strategies for insurers with random risk aversion under Heston's SV model,
{\em Mathematics and Computers in Simulation}, {\bf242}:343-365, 2026.

\bibitem{LiY2013Optimal}
Y.W. Li and Z.F. Li,
 Optimal time-consistent investment and reinsurance strategies for mean-variance insurers with state dependent risk aversion,
{\em Insurance: Mathematics and Economics}, {\bf53(1)}:86-97, 2013.

\bibitem{Li2012}
Z.F. Li, Y. Zeng and Y.Z. Lai,
 Optimal time-consistent investment and reinsurance strategies for insurers under Heston's SV model,
{\em Insurance: Mathematics and Economics}, {\bf51(1)}:191-203, 2012.

\bibitem{Meng2025}
W.J. Meng, J.T. Shi, T.X. Wang and J.F. Zhang,
A general maximum principle for optimal control of stochastic differential delay systems,
{\em SIAM Journal on Control and Optimization}, {\bf63(1)}:175-205, 2025.

\bibitem{Pham2009}
H. Pham,
{\em Continuous-time Stochastic Control and Optimization with Financial Applications},
Springer, Berlin, 2009.

\bibitem{Shen2014}
Y. Shen and Y. Zeng,
Optimal investment-reinsurance with delay for mean-variance insurers: A maximum principle approach,
{\em Insurance: Mathematics and Economics}, {\bf57}:1-12, 2014.

\bibitem{Strotz1955}
R.H. Strotz,
 Myopia and inconsistency in dynamic utility maximization,
{\em Review of Economic Studies}, {\bf23(3)}:165-180, 1955.

\bibitem{Wang2024}
Y.K. Wang, J.Z. Liu and T.K. Siu,
Investment-consumption-insurance optimisation problem with multiple habit formation and non-exponential discounting,
{\em Finance and Stochastics}, {\bf28(1)}:161-214, 2024.

\bibitem{Yan2022}
T.J. Yan and H.Y. Wong,
Equilibrium pairs trading under delayed cointegration,
{\em Automatica}, {\bf144}:110498, 2022.

\bibitem{Yang2005}
H.L. Yang and L.H. Zhang,
 Optimal investment for insurer with jump-diffusion risk process,
{\em Insurance: Mathematics and Economics}, {\bf37(3)}:615-634, 2005.

\bibitem{Yong2012}
J.M. Yong,
 Time-inconsistent optimal control problems and the equilibrium HJB equation,
{\em Mathematical Control and Related Fields}, {\bf2(3)}:271-329, 2012.

\bibitem{Yong2017}
J.M. Yong,
 Linear-quadratic optimal control problems for mean-field stochastic differential equations--time-consistent solutions,
{\em Transactions of the American Mathematical Society}, {\bf369(8)}:5467-5523, 2017.

\bibitem{Yong1999}
J.M. Yong and X.Y. Zhou,
{\em Stochastic Controls: Hamiltonian Systems and HJB Equations},
Springer, New York, 1999.

\bibitem{Yuan2023}
Y. Yuan, X. Han, Z.B. Liang and K.C. Yuen,
Optimal reinsurance-investment strategy with thinning dependence and delay factors under mean-variance framework,
{\em European Journal of Operational Research}, {\bf311(2)}:581-595, 2023.

\bibitem{Zhang2024}
C.B. Zhang, Z.B. Liang and Y. Yuan,
Stochastic differential investment and reinsurance game between an insurer and a reinsurer under thinning dependence structure,
{\em European Journal of Operational Research}, {\bf315(1)}:213-227, 2024.


\bibitem{ZengX2013}
X.D. Zeng and M. Taksar,
 A stochastic volatility model and optimal portfolio selection,
{\em Quantitative Finance}, {\bf13(10)}:1547-1558, 2013.

\bibitem{Zwillinger1998}
D. Zwillinger,
{\em Handbook of Differential Equations},
Academic Press, San Diego, 1997.



\end{thebibliography}
\end{document}